\documentclass[12pt]{amsart}

\usepackage{amscd,amssymb,amsthm}
\usepackage{graphicx}

\newcommand{\al}{\alpha}
\newcommand{\area}{\operatorname{area}}

\newcommand{\Aut}{\operatorname{Aut}}

\newcommand{\C}{{\mathbb C}}

\newcommand{\Coker}{\operatorname{Coker}}

\newcommand{\const}{\operatorname{const.}}

\newcommand{\D}{\partial}

\newcommand{\diam}{\operatorname{diam}}
\newcommand{\Diff}{\operatorname{Diff}}

\newcommand{\dist}{\operatorname{dist}}

\newcommand{\dvol}{\operatorname{dvol}}

\newcommand{\GL}{\operatorname{GL}}

\newcommand{\HH}{\operatorname{H}}

\newcommand{\hU}{\widehat{U}}

\newcommand{\Image}{\operatorname{Im}}

\newcommand{\Int}{\operatorname{int}}

\newcommand{\Isom}{\operatorname{Isom}}
\newcommand{\Ker}{\operatorname{Ker}}

\newcommand{\Nil}{\operatorname{Nil}}

\newcommand{\Or}{{\mathcal O}}

\newcommand{\PSL}{\operatorname{PSL}}

\newcommand{\R}{{\mathbb R}}

\newcommand{\Ric}{\operatorname{Ric}}
\newcommand{\Rm}{\operatorname{Rm}}

\newcommand{\SL}{\operatorname{SL}}
\newcommand{\SO}{\operatorname{SO}}
\newcommand{\Sol}{\operatorname{Sol}}

\newcommand{\supp}{\operatorname{supp}}

\newcommand{\vol}{\operatorname{vol}}
\newcommand{\Z}{{\mathbb Z}}
\newcommand{\cangle}{\widetilde{\angle}}

\newcommand{\eps}{\epsilon}
\newcommand{\ga}{\gamma}

\newcommand{\ra}{\rightarrow}

\numberwithin{equation}{section}
\theoremstyle{plain}
\newtheorem{assumption}[equation]{Assumption}
\newtheorem{lemma}[equation]{Lemma}
\newtheorem{theorem}[equation]{Theorem}
\newtheorem{proposition}[equation]{Proposition}
\newtheorem{corollary}[equation]{Corollary}

\theoremstyle{definition}
\newtheorem{definition}[equation]{Definition}

\theoremstyle{remark}
\newtheorem{remark}[equation]{Remark}

\begin{document}

\begin{abstract}
A three-dimensional closed orientable orbifold (with no bad
suborbifolds) is known to have a geometric decomposition from
work of Perelman \cite{Perelman,Perelman2} 
in the manifold case, along with earlier work of
Boileau-Leeb-Porti \cite{BLP}, Boileau-Maillot-Porti \cite{BMP},
Boileau-Porti \cite{BP},
Cooper-Hodgson-Kerckhoff \cite{CHK} and Thurston \cite{Thurston2}.
We give a new, logically independent, unified
proof of the geometrization of orbifolds, using Ricci flow.
Along the way we develop
some tools for the geometry of orbifolds that may be of independent interest.
\end{abstract}

\title{Geometrization of Three-Dimensional Orbifolds via Ricci Flow}
\author{Bruce Kleiner}
\address{Courant Institute of Mathematical Sciences\\
251 Mercer St. \\ 
New York, NY  10012}
\email{bkleiner@cims.nyu.edu}

\author{John Lott}
\address{Department of Mathematics\\
University of California at Berkeley\\
Berkeley, CA 94720}
\email{lott@math.berkeley.edu}
\thanks{Research supported by NSF grants DMS-0903076 and DMS-1007508}
\date{May 28, 2014}
\maketitle

\tableofcontents

\section{Introduction} \label{sect1}

\subsection{Orbifolds and geometrization}
Thurston's geometrization conjecture for $3$-manifolds states that 
every closed 
 orientable 
$3$-manifold has a canonical decomposition
into geometric pieces.  In the early 1980's Thurston announced a proof
of the conjecture for Haken manifolds \cite{Thurston3}, with written 
proofs appearing much later \cite{Kapovich,McMullen,Otal1,Otal2}.  
The conjecture was settled completely a few years ago by Perelman in his
spectacular work using Hamilton's Ricci flow \cite{Perelman,Perelman2}. 

Thurston also formulated a geometrization conjecture
for orbifolds.    We recall that orbifolds are 
similar to manifolds, except that they are locally modelled
on quotients of the form $\R^n/G$, where $G\subset O(n)$ is
a finite subgroup of the orthogonal group.  
Although the terminology is relatively recent, orbifolds have a long
history in mathematics, going back to the classification of 
crystallographic groups and Fuchsian groups.   In this paper,
using Ricci flow,
we will give a new proof of the 
geometrization conjecture for orbifolds:

\begin{theorem} \label{theorem1.1}
Let $\Or$ be a closed connected orientable three-dimensional orbifold
which does not contain any bad embedded  $2$-dimensional suborbifolds.
Then $\Or$ has a geometric decomposition.
\end{theorem}

The existing proof of Theorem \ref{theorem1.1} is based on
a canonical splitting
of  $\Or$ along spherical and Euclidean
$2$-dimensional suborbifolds, which is analogous to the prime
and JSJ decomposition of 
$3$-manifolds.  This splitting reduces Theorem \ref{theorem1.1} to two separate 
cases --  when 
$\Or$ is a manifold, and when  
$\Or$ has a nonempty singular locus and satisfies an irreducibility condition.  
The first case is Perelman's theorem for manifolds.   Thurston
announced a proof of the latter case in \cite{Thurston2} and gave an outline.
A detailed  proof of the latter case
was given by Boileau-Leeb-Porti \cite{BLP},
after work of Boileau-Maillot-Porti \cite{BMP}, Boileau-Porti
\cite{BP}, Cooper-Hodgson-Kerckhoff \cite{CHK} and
Thurston \cite{Thurston2}.
The monographs \cite{BMP,CHK} give excellent expositions of
$3$-orbifolds and their geometrization.

\subsection{Discussion of the proof}
The main purpose of this paper is to provide a new proof of Theorem \ref{theorem1.1}.
Our proof is an extension of   Perelman's proof of geometrization for 
$3$-manifolds to orbifolds, 
bypassing \cite{BLP,BMP,BP,CHK,Thurston2}.
The motivation for this alternate approach
is twofold.    First,  anyone interested in the geometrization of general 
orbifolds as in Theorem \ref{theorem1.1} will necessarily 
have to go through Perelman's Ricci flow proof in the manifold case, and 
also absorb foundational results about orbifolds.   At that point, the 
additional effort required to deal with general orbifolds is relatively 
minor 
in comparison to the proof
in \cite{BLP}. This latter proof involves a number of 
ingredients, including Thurston's 
geometrization of Haken manifolds, 
the deformation and collapsing theory of hyperbolic cone manifolds, and some Alexandrov
space theory.   Also, in contrast to the existing proof of Theorem \ref{theorem1.1},  the Ricci flow
argument gives a unified approach to  geometrization for both manifolds and  orbifolds.

Many of the steps in Perelman's proof have evident
orbifold generalizations, whereas some do not. 
It would be
unwieldy to rewrite all the details of Perelman's proof, on the level of
\cite{Kleiner-Lott}, while making notational changes from manifolds
to orbifolds.  Consequently, we focus on the steps in Perelman's proof 
where an orbifold extension is not immediate. For a step where the
orbifold extension is routine, we make the precise orbifold statement 
and indicate where the analogous manifold proof occurs in \cite{Kleiner-Lott}.

In the course of proving  Theorem \ref{theorem1.1}, we needed to develop a 
number of foundational results about the geometry of orbifolds.  Some of these
may be of independent interest, or of use for subsequent work in this area,
such as
the compactness theorem for Riemannian orbifolds, critical point theory,
and the soul theorem.

Let us mention one of the steps where the orbifold
extension could
{\it a priori} be an issue. 
This is where one characterizes the topology of the thin part of
the large-time orbifold.  
To do this, one
first needs a sufficiently flexible proof in the manifold case.
We provided such a proof in \cite{Kleiner-Lott2}.
The proof in \cite{Kleiner-Lott2} uses some basic techniques from 
Alexandrov geometry, combined with smoothness results in appropriate places.
It provides a decomposition of the thin part into various pieces
which together give an explicit realization of the thin part as
a graph manifold.  When combined with 
preliminary results that are proved in this
paper, we can extend the techniques of \cite{Kleiner-Lott2} to orbifolds.
We get a decomposition of the thin part of the large-time orbifold
into various pieces, similar to those in \cite{Kleiner-Lott2}. 
We show that these pieces give an explicit realization of each
component of the thin part as
either a graph orbifold or one of a few exceptional cases.
This is more involved to prove in the orbifold case than in the
manifold case but the basic strategy is the same.

\subsection{Organization of the paper}

The structure of this paper is as follows.  One of our tasks is to
provide a framework for the topology and Riemannian
geometry of orbifolds, so that results about Ricci flow
on manifolds extend as easily as possible to orbifolds. In Section
\ref{sect2} we recall the relevant notions that we need from
orbifold topology.  We then introduce
Riemannian orbifolds and prove the orbifold
versions of some basic results from Riemannian geometry, such as
the de Rham decomposition and critical point theory.

Section \ref{sect3} is concerned with noncompact nonnegatively curved 
orbifolds.  
We prove the orbifold version of 
the Cheeger-Gromoll soul theorem.
We list the diffeomorphism
types of noncompact nonnegatively curved 
orbifolds with dimension at most three.

In Section \ref{sect4} we prove a compactness theorem for Riemannian
orbifolds. Section \ref{sect5} contains some preliminary information about
Ricci flow on orbifolds, along with the classification of 
the diffeomorphism types of
compact nonnegatively curved three-dimensional orbifolds. 
We also show how to extend
Perelman's no local collapsing theorem to orbifolds. 

Section \ref{sect6} is devoted to $\kappa$-solutions. 
Starting in Section \ref{sect7}, we specialize to three-dimensional
orientable orbifolds with no bad $2$-dimensional suborbifolds.
We show how to extend Perelman's results in order to construct
a Ricci flow with surgery.

In Section \ref{sect8} we show that the thick part of the
large-time geometry approaches a finite-volume orbifold of constant negative
curvature.  Section \ref{sect9} contains the topological
characterization of the thin
part of the large-time geometry.

Section \ref{sect10} concerns the incompressibility of hyperbolic
cross-sections. Rather than using minimal disk techniques as
initiated by Hamilton \cite{Hamilton (1999)}, we
follow an approach introduced by Perelman \cite[Section 8]{Perelman2}
that uses a monotonic quantity, as modified in 
\cite[Section 93.4]{Kleiner-Lott}. 

The appendix contains topological facts about graph orbifolds.
We show that a ``weak'' graph orbifold 
is the result of performing $0$-surgeries (i.e. connected sums)
on a ``strong'' graph orbifold.
This material is probably known to some experts but we were unable
to find references in the literature, so we include complete proofs.

After writing this paper we learned that Daniel Faessler
independently proved Proposition 9.7, which is the orbifold version of the
collapsing theorem \cite{Faessler}.

\subsection{Acknowledgements}
We thank Misha Kapovich and Sylvain Maillot for orbidiscussions.
We thank the referee for a careful reading of the paper and for corrections.

\section{Orbifold topology and geometry} \label{sect2}

In this section we first review the differential topology of orbifolds.
Subsections \ref{subsect2.1} 
and \ref{subsect2.2} contain information about orbifolds in
any dimension.
In some cases we give precise definitions and in other
cases we just recall salient properties, referring to 
the monographs \cite{BMP,CHK} for more detailed information. 
Subsections \ref{subsect2.3} and \ref{subsect2.4} 
are concerned with low-dimensional orbifolds. 

We then give a short exposition of aspects of the 
differential geometry of orbifolds, in Subsection \ref{subsect2.5}. 
It is hard to find a comprehensive reference for this material and
so we flag the relevant notions; see
\cite{Borzellino1} for further discussion of some points.
Subsection \ref{subsect2.6} shows how to
do critical point theory on orbifolds. Subsection 
\ref{subsect2.7} discusses
the smoothing of functions on orbifolds.

For notation, $B^n$ is the open unit $n$-ball, $D^n$ is the closed 
unit $n$-ball and $I = [-1,1]$. We let $D_k$ denote the dihedral
group of order $2k$.

\subsection{Differential topology of orbifolds} \label{subsect2.1}

An {\em orbivector space} is a triple $(V, G, \rho)$, where
\begin{itemize}
\item $V$ is a vector space, 
\item $G$ is a finite group and
\item $\rho : G \rightarrow \Aut(V)$ is a 
faithful 
linear representation.
\end{itemize}
A (closed/ open/ convex/...) {\em subset} of $(V, G, \rho)$ is a
$G$-invariant subset of $V$ which is (closed/ open/ convex/...)
A {\em linear map} from $(V, G, \rho)$ to $(V^\prime, G^\prime, \rho^\prime)$
consists of a linear map $T : V \rightarrow V^\prime$ and a homomorphism
$h : G \rightarrow G^\prime$ so that for all $g \in G$,
$\rho^\prime(h(g)) \circ T = T \circ \rho(g)$.
The linear map is {\em injective} (resp. {\em surjective}) if
$T$ is {\em injective} (resp. {\em surjective}) and
$h$ is {\em injective} (resp. {\em surjective}).
An {\em action} of a group $K$ on $(V,G,\rho)$ is given by a
short exact sequence $1 \rightarrow G \rightarrow L \rightarrow K
\rightarrow 1$ and a homomorphism $L \rightarrow \Aut(V)$ that
extends $\rho$.

A {\em local model} is a pair $(\hU, G)$, where $\hU$ is a connected
open subset of a Euclidean space
and $G$ is a finite group that acts smoothly and effectively on $\hU$, 
on the right. 
(Effectiveness means that the homomorphism $G \rightarrow \Diff(\hU)$ is
injective.) We will sometimes write $U$ for $\hU/G$,
endowed with the quotient topology.

A {\em smooth map} between local models $(\hU_1, G_1)$ and $(\hU_2, G_2)$
is given by a smooth map $\widehat{f} \: : \: \hU_1 \rightarrow \hU_2$ and a
homomorphism $\rho \: : \: G_1 \rightarrow G_2$ so that 
$\widehat{f}$ is $\rho$-equivariant, i.e. 
$\widehat{f}(x g_1) \: = \: \widehat{f}(x) \rho(g_1)$. We do not
assume that $\rho$ is injective or surjective. The map between local
models is an {\em embedding} if $\widehat{f}$ is an embedding; 
it follows from effectiveness that
$\rho$ is injective 
in this case.

\begin{definition} \label{defn2.1}
An {\em atlas} for an $n$-dimensional
orbifold $\Or$ consists of \\
1. A Hausdorff paracompact
topological space $|\Or|$, \\
2. An open covering
$\{U_\alpha\}$ of $|\Or|$, \\ 
3. Local models $\{(\hU_\alpha, G_\alpha)\}$ with
each $\hU_\alpha$ a connected open subset of $\R^n$ and\\
4. Homeomorphisms $\phi_\alpha \: : \: 
U_\alpha \rightarrow \hU_\alpha/G_\alpha$
so that \\
5. If $p \in U_1 \cap U_2$ then there is a local model $(\hU_3, G_3)$ with
$p \in U_3$ along with embeddings $(\hU_3, G_3) \rightarrow (\hU_1, G_1)$
and $(\hU_3, G_3) \rightarrow (\hU_2, G_2)$.
\end{definition}

An {\em orbifold} $\Or$
is an equivalence class of such atlases, where two atlases are
equivalent if they are both included in a third atlas.
With a given atlas, the orbifold $\Or$ is
{\em oriented} if each $\hU_\alpha$ is oriented, the action of
$G_\alpha$ is orientation-preserving, and the embeddings
$\hU_3 \rightarrow \hU_1$ and $\hU_3 \rightarrow \hU_2$ are
orientation-preserving. We say that $\Or$ is {\em connected} (resp.
{\em compact}) if
$|\Or|$ is connected (resp. compact).

An {\em orbifold-with-boundary} $\Or$ is defined similarly, with
$\hU_\alpha$ being a connected open subset of $[0, \infty) \times 
\R^{n-1}$. The {\em boundary} $\partial \Or$ is a boundaryless
$(n-1)$-dimensional orbifold, with $|\partial \Or|$ consisting of
points in $|\Or|$ whose local lifts lie in $\{0\} \times \R^{n-1}$. 
Note that it is possible that $\partial \Or \: = \: \emptyset$ while
$|\Or|$ is a topological manifold with a nonempty boundary.

\begin{remark} In this paper we only deal with {\em effective}
orbifolds, meaning that in a local model $(\widehat{U}, G)$, the
group $G$ always acts effectively.  It would be more natural
in some ways to remove this effectiveness assumption.  However,
doing so would hurt the readability of the paper, so we will stick to
effective orbifolds.
\end{remark}

Given a point $p \in |\Or|$ and a local model $(\hU, G)$ around $p$, let
$\widehat{p} \in \hU$ project to $p$. The {\em local group} $G_p$ is the
stabilizer group
$\{g \in G \: : \: \widehat{p}g \: = \: \widehat{p}$ \}. Its isomorphism
class is independent of the choices made. We can always find a local
model with $G = G_p$.

The {\em regular part} 
$|\Or|_{reg} \subset |\Or|$ consists of the points with $G_p \: = \: \{e\}$. It
is a smooth manifold that forms an open dense subset of $|\Or|$.

Given an open subset $X \subset |\Or|$, there is an induced orbifold
$\Or \Big|_X$ with $\left| \Or \Big|_X \right| = X$. In some cases
we will have a subset $X \subset |\Or|$, possibly not open, for which
$\Or \Big|_X$ is an orbifold-with-boundary.

The {\em ends} of $\Or$ are the ends of $|\Or|$.

A {\em smooth map} $f \: : \: \Or_1 \rightarrow \Or_2$ between orbifolds
is given by a
continuous map $|f| \: : \: |\Or_1| \rightarrow |\Or_2|$ with the 
property that for each $p \in |\Or_1|$, there are
\begin{itemize}
\item Local models
$(\hU_1, G_1)$ and $(\hU_2, G_2)$ for $p$ and $f(p)$, respectively, and
\item A smooth map $\widehat{f} \: : \:
(\hU_1, G_1) \rightarrow (\hU_2, G_2)$ between local models 
\end{itemize}
so that the
diagram
\begin{equation}
\begin{matrix}
\hU_1 & \stackrel{\widehat{f}}{\rightarrow} & \hU_2 \\
\downarrow & & \downarrow \\
U_1 & \stackrel{|f|}{\rightarrow} & U_2
\end{matrix}
\end{equation}
commutes.

There is an induced homomorphism from $G_p$ to $G_{f(p)}$.
We emphasize that to define a smooth map $f$ between two orbifolds,
one must first define a map $|f|$ between their underlyihg spaces.

We write $C^\infty(\Or)$ for the space of smooth maps $f : \Or \rightarrow \R$.

A smooth map $f : \Or_1 \rightarrow \Or_2$ is {\em proper} if
$|f| : |\Or_1| \rightarrow |\Or_2|$ is a proper map.

A {\em diffeomorphism} $f \: : \: \Or_1 \rightarrow \Or_2$ is a smooth map
with a smooth inverse. Then $G_p$ is isomorphic to $G_{f(p)}$.

If a discrete group $\Gamma$ acts properly discontinuously on a
manifold $M$ then there is a quotient orbifold, which we denote
by $M//\Gamma$. It has $|M//\Gamma| = M/\Gamma$.
Hence if $\Or$ is an orbifold and $(\hU, G)$ is a local model for $\Or$
then we can say that $\Or \Big|_{U}$ is diffeomorphic
to $\hU//G$. An orbifold $\Or$ is {\em good} if $\Or = M//\Gamma$ for
some manifold $M$ and some discrete group $\Gamma$. It is
{\em very good} if $\Gamma$ can be taken to be finite.
A {\em bad} orbifold is one that is not good.

Similarly, suppose that a discrete group $\Gamma$ acts 
by diffeomorphisms on an orbifold 
$\Or$. We say that it acts {\em properly
discontinuously} if the action of $\Gamma$ on $|\Or|$ is properly 
discontinuous. Then there is a quotient orbifold $\Or//\Gamma$, with
$|\Or//\Gamma| = |\Or|/\Gamma$; see Remark \ref{remark2.10}.

An {\em orbifiber bundle} 
consists of a smooth map
$\pi : \Or_1 \rightarrow \Or_2$ between two orbifolds, along with
a third orbifold $\Or_3$ such that 
\begin{itemize}
\item $|\pi|$ is surjective, and
\item For each $p \in |\Or_2|$, there is a local model
$(\hU, G_p)$ around $p$,
where $G_p$ is the local group at $p$,
along with an action of $G_p$ on $\Or_3$
and a diffeomorphism $(\Or_3 \times \hU)//G_p \rightarrow 
\Or_1 \Big|_{|\pi|^{-1}(U)}$
so that the diagram
\begin{equation}
\begin{matrix}
(\Or_3 \times \hU)//G_p &  \rightarrow & \Or_1 \\
 \downarrow & & \downarrow  \\
\widehat{U}//G_p &  \rightarrow  & \Or_2
\end{matrix}
\end{equation}
commutes.
\end{itemize}

(Note that if $\Or_2$ is a manifold then 
the orbifiber bundle $\pi:\Or_1\ra \Or_2$ 
has a local
product structure.)
The {\em fiber} of the orbifiber bundle is $\Or_3$. Note that for
$p_1 \in |\Or_1|$, the homomorphism
$G_{p_1} \rightarrow G_{|\pi|(p_1)}$ is surjective.
 
A {\em section} of 
an orbifiber bundle $\pi : \Or_1 \rightarrow \Or_2$
is a smooth map $s : \Or_2 \rightarrow \Or_1$
such that $\pi \circ s$ is the identity on $\Or_2$.

A {\em covering map} $\pi : \Or_1 \rightarrow \Or_2$ is a orbifiber 
bundle with a 
zero-dimensional fiber.
Given $p_2 \in |\Or_2|$ and $p_1 \in |\pi|^{-1}(p_2)$, 
there are a local model $(\hU, G_2)$ around $p_2$ and
a subgroup 
$G_1 \subset G_2$
so that $(\hU, G_1)$ is a local model around
$p_1$ and the map $\pi$ is locally $(\hU, G_1) \rightarrow (\hU, G_2)$. 

A rank-$m$ {\em orbivector bundle} ${\mathcal V} \rightarrow \Or$ over $\Or$ 
is locally isomorphic to $(V \times \hU)/G_p$, where 
$V$ is an $m$-dimensional orbivector space on which 
$G_p$ acts linearly.

The {\em tangent bundle} $T\Or$ of an orbifold ${\Or}$ is an
orbivector bundle which is locally diffeomorphic to 
$T\hU_\al//G_\al$. Given $p \in |\Or|$, if $\widehat{p} \in \hU$ covers $p$
then the {\em tangent space}
$T_p \Or$ is isomorphic to the orbivector space $(T_{\widehat{p}} \hU, G_p)$.
The {\em tangent cone} at $p$ is $C_p|\Or| \cong T_{\widehat{p}} \hU/G_p$.

A smooth {\em vector field} $V$ is a smooth section of $T\Or$. 
In terms of a local model
$(\hU, G)$, the vector field $V$ restricts to a vector field on $\hU$ 
which is $G$-invariant. 

A smooth map $f : \Or_1 \rightarrow \Or_2$ gives rise to the 
{\em differential}, an orbivector bundle map
$df : T\Or_1 \rightarrow T\Or_2$. At a point $p \in |\Or|$, in terms of
local models we have a map
$\widehat{f} : (\hU_1, G_1) \rightarrow (\hU_2, G_2)$ 
which gives rise to a $G_p$-equivariant map 
$d\widehat{f}_p : T_{\widehat{p}} \hU_1
\rightarrow T_{\widehat{f}(\widehat{p})} \hU_2$ and hence to a 
linear map $df_p : T_p\Or_1 \rightarrow T_{|f|(p)} \Or_2$.

Given a smooth map $f : \Or_1 \rightarrow \Or_2$ and a point $p \in |\Or_1|$,
we say that $f$ is a {\em submersion at $p$}
(resp. {\em immersion at $p$}) if
the map $df_p : T_p \Or_1 \rightarrow
T_{|f|(p)} \Or_2$ is surjective (resp. injective).

\begin{lemma} \label{lemma2.3}
If $f$ is a submersion at $p$ then there 
is an orbifold
$\Or_3$ on which $G_{|f|(p)}$ acts, 
along with a local model $(\hU_2, G_{|f|(p)})$ around
$|f|(p)$,  
so that
$f$ is equivalent near $p$ to the
projection map $(\Or_3 \times \hU_2)//G_{|f|(p)} \rightarrow 
\hU_2//G_{|f|(p)}$.
\end{lemma}
\begin{proof}
Let $\rho : G_p \rightarrow G_{|f|(p)}$ be the surjective homomorphism
associated to $df_p$.
Let $\widehat{f} \: : \: (\widehat{U}_1, G_p) \rightarrow
(\widehat{U}_2, G_{|f|(p)})$ be a local model for $f$ near $p$; it is
necessarily $\rho$-equivariant.
Let $\widehat{p} \in \widehat{U}_1$ be a lift of $p \in U_1$. Put
$\widehat{W} = \widehat{f}^{-1}(\widehat{f}(\widehat{p}))$.
Since $\widehat{f}$ is a submersion at $\widehat{p}$, after reducing
$\widehat{U}_1$ and $\widehat{U}_2$ if necessary, there is a
$\rho$-equivariant diffeomorphism $\widehat{W} \times \widehat{U}_2 \rightarrow
\widehat{U}_1$ so that the diagram
\begin{equation}
\begin{matrix}
\widehat{W} \times \widehat{U}_2 & \rightarrow & \widehat{U}_1 \\
\downarrow & & \downarrow \\
\widehat{U}_2 & \rightarrow & \widehat{U}_2
\end{matrix}
\end{equation}
commutes and is $G_p$-equivariant. Now $\Ker(\rho)$ acts on $\widehat{W}$.
Put ${\mathcal O}_3 = \widehat{W}//\Ker(\rho)$.
Then there is a commuting diagram of orbifold maps
\begin{equation}
\begin{matrix}
{\mathcal O}_3 \times \widehat{U}_2 & 
\rightarrow & \widehat{U}_1//\Ker(\rho) \\
\downarrow & & \downarrow \\
\widehat{U}_2 & \rightarrow & \widehat{U}_2
\end{matrix}.
\end{equation}
Further quotienting by $G_{|f|(p)}$ gives a commutative diagram
\begin{equation}
\begin{matrix}
({\mathcal O}_3 \times \widehat{U}_2)//G_{|f|(p)} & 
\rightarrow &  \widehat{U}_1//G_p \\ \downarrow & & \downarrow \\
\widehat{U}_2//G_{|f|(p)} & \rightarrow & \widehat{U}_2//G_{|f|(p)}
\end{matrix}
\end{equation}
whose top horizontal line is an orbifold diffeomorphism.
\end{proof}

We say that $f : \Or_1 \rightarrow \Or_2$ is a 
{\em submersion} (resp. {\em immersion}) 
if it is a submersion (resp. immersion) at $p$ for all
$p \in |\Or_1|$.

\begin{lemma} \label{lemma2.4}
A proper surjective submersion $f : \Or_1 \rightarrow \Or_2$, with
$\Or_2$ connected,
defines an orbifiber bundle with compact fibers.
\end{lemma}

We will sketch a proof of Lemma \ref{lemma2.4} in Remark \ref{remark2.10.5}.

In particular, a proper surjective local diffeomorphism to a connected
orbifold is a covering
map with finite fibers.

An immersion $f : \Or_1 \rightarrow \Or_2$ has a
{\em normal bundle} $N\Or_1 \rightarrow \Or_1$ whose fibers have the 
following local description. Given $p \in |\Or_1|$, let $f$
be described in terms of local models
$(\hU_1, G_p)$ and $(\hU_2, G_{|f|(p)})$
by a $\rho$-equivariant immersion $\widehat{f} \: : \:
\hU_1 \rightarrow \hU_2$. Let $F_p \subset G_{|f|(p)}$ be the
subgroup which fixes $\Image(d\widehat{f}_p)$. Then the
{\em normal space} $N_p \Or_1$ is the orbivector space $\left( 
\Coker(d\widehat{f}_p), F_p \right)$.

A {\em suborbifold} of $\Or$ is 
given by an orbifold  $\Or^\prime$ and an immersion
$f : \Or^\prime
\rightarrow \Or$ for which $|f|$ maps $|\Or^\prime|$ homeomorphically
to its image in $|\Or|$. From effectiveness, for each
$p \in |\Or^\prime|$, the homomorphism
$\rho_p : G_p \rightarrow G_{|f|(p)}$ is injective.
Note that $\rho_p$ need not be an isomorphism.
We will identify $\Or^\prime$ with its image in $\Or$.
There is a neighborhood of $\Or^\prime$ which is diffeomorphic to
the normal bundle $N\Or^\prime$.
We say that the suborbifold $\Or^\prime$ is {\em embedded} if
$\Or \Big|_{|\Or^\prime|} = \Or^\prime$. 
Then for each $p \in |\Or^\prime|$, the homomorphism $\rho_p$ is an
isomorphism.

If $\Or^\prime$ is an 
embedded
codimension-$1$ suborbifold of $\Or$ then
we say that $\Or^\prime$ is {\em two-sided} if the normal bundle
$N\Or^\prime$ has a nowhere-zero section.  
If $\Or$ and $\Or^\prime$ are both orientable then $\Or^\prime$ is
two-sided. 
We say that $\Or^\prime$ is {\em separating} if
$|\Or^\prime|$ is separating in $|\Or|$.

We can talk about two suborbifolds meeting {\em transversely}, as
defined using local models. 

Let $\Or$ be an oriented orbifold
(possibly disconnected). Let $D_1$ and $D_2$ be disjoint codimension-zero
embedded
suborbifolds-with-boundary, both oriented-diffeomorphic to $D^n//\Gamma$. Then
the operation of performing {\em $0$-surgery  along $D_1$, $D_2$} produces the new oriented orbifold
$\Or^\prime = (\Or - \Int(D_1) - \Int(D_2)) 
\bigcup_{\partial D_1 \sqcup \partial D_2} 
(I \times (D^n//\Gamma))$.  In the manifold case, a connected sum is the
same thing as a $0$-surgery along a pair $\{D_1, D_2\}$ which lie in different
connected components of $\Or$.  Note
that unlike in the manifold case, $\Or^\prime$ is generally not uniquely
determined up to diffeomorphism 
by knowing the connected components containing $D_1$ and
$D_2$. 
For example, even if $\Or$ is connected, 
$D_1$ and $D_2$ may or may not lie on the same connected component
of the singular set.

If $\Or_1$ and $\Or_2$ are oriented orbifolds, with $D_1 \subset \Or_1$ and
$D_2 \subset \Or_2$  both oriented diffeomorphic to $D^n//\Gamma$,
then we may write $\Or_1 \#_{S^{n-1}//\Gamma} \Or_2$ for the connected
sum.  This notation is slightly ambiguous since
the location of $D_1$ and $D_2$ is implicit. We will write
$\Or \#_{S^{n-1}//\Gamma}$ to denote a $0$-surgery on a single
orbifold $\Or$. Again the notation is slightly ambiguous, since the
location of $D_1, D_2 \subset \Or$ is implicit.

An {\em involutive distribution} on $\Or$ is a subbundle $E \subset T\Or$ with 
the property that for any two sections $V_1, V_2$ of $E$, the Lie bracket
$[V_1, V_2]$ is also a section of $E$.

\begin{lemma} \label{lemma2.5}
Given an involutive distribution $E$ on $\Or$,
for any $p \in |\Or|$ there is a unique maximal suborbifold passing
through $p$ which is tangent to $E$.
\end{lemma}

Orbifolds have partitions of unity.

\begin{lemma} \label{lemma2.6}
Given an open cover $\{U_\al\}_{\al \in A}$ of $|\Or|$, there is a collection
of functions $\rho_\al \in C^\infty(\Or)$ such that
\begin{itemize}
\item $0 \le \rho_\al \le 1$.
\item 
$\supp(\rho_\al) \subset U_{\al^\prime}$ for some $\alpha^\prime
= \alpha^\prime(\alpha) \in A$.
\item For all $p \in |\Or|$, $\sum_{\al \in A} \rho_\al(p) = 1$. 
\end{itemize}
\end{lemma}
\begin{proof}
The proof is similar to the manifold case, using local models
$(\hU, G)$ consisting of coordinate neighborhoods, along with compactly
supported $G$-invariant smooth functions on $\hU$.
\end{proof}

A {\em curve} in an orbifold is a smooth map 
$\gamma \: : \: I \rightarrow \Or$
defined on an interval $I \subset \R$. 
A {\em loop} is a curve $\gamma$ with $|\gamma|(0) = |\gamma|(1)
\in |\Or|$.

\subsection{Universal cover and fundamental group} \label{subsect2.2}

We follow the presentation in \cite[Chapter 2.2.1]{BMP}.
Choose a regular point $p \in |\Or|$.
A {\em special curve} from $p$
is a curve $\gamma : [0,1] \rightarrow \Or$ such that
\begin{itemize}
\item $|\gamma|(0) = p$ and 
\item $|\gamma|(t)$ lies in $|\Or|_{reg}$ for all but a finite number of $t$.
\end{itemize}

Suppose that $(\hU, G)$ is a local model and that $\widehat{\gamma} :
[a, b] \rightarrow \hU$ is a lifting of $\gamma_{[a,b]}$, for some
$[a,b] \subset [0,1]$. An {\em elementary homotopy} between two special
curves is a smooth homotopy
of $\widehat{\gamma}$ in $\hU$, relative to $\widehat{\gamma}(a)$ and
$\widehat{\gamma}(b)$. A {\em homotopy} of $\gamma$ is what's generated
by elementary homotopies. 

If $\Or$ is connected then the {\em universal cover} $\widetilde{\Or}$ of $\Or$
can be constructed as the set of special curves starting at 
$p$, modulo homotopy. 
It has a natural orbifold structure. The {\em fundamental group}
$\pi_1(\Or, p)$ is given by special loops (i.e. special curves $\gamma$ with
$|\gamma|(1) = p$) modulo homotopy. Up to isomorphism, 
$\pi_1(\Or, p)$ is independent of
the choice of $p$.

If $\Or$ is connected and a discrete group 
$\Gamma$ acts properly discontinuously on $\Or$
then there is a short exact sequence
\begin{equation}
1 \longrightarrow \pi_1(\Or, p) \longrightarrow \pi_1(\Or//\Gamma, p\Gamma)
\longrightarrow \Gamma \longrightarrow 1.
\end{equation}

\begin{remark} \label{remark2.8}
A more enlightening way to think of an orbifold is to consider it
as a smooth effective proper \'etale groupoid ${\mathcal G}$, as 
explained in 
\cite{Adem-Leida-Ruan (2007),Bridson-Haefliger (1999),Moerdijk}.  
We recall that a {\em Lie groupoid} ${\mathcal G}$ essentially consists of a 
smooth manifold
${\mathcal G}^{(0)}$ (the space of units), another smooth manifold 
${\mathcal G}^{(1)}$
and submersions $s,r : {\mathcal G}^{(1)} \rightarrow 
{\mathcal G}^{(0)}$ (the source
and range maps), along with a partially defined multiplication
${\mathcal G}^{(1)} \times {\mathcal G}^{(1)} \rightarrow {\mathcal G}^{(1)}$
which satisfies certain compatibility conditions.
A Lie groupoid is {\em \'etale} if $s$ and $r$ are local
diffeomorphisms.  It is {\em proper} if $(s,r) : 
{\mathcal G}^{(1)} \rightarrow
{\mathcal G}^{(0)} \times {\mathcal G}^{(0)}$ 
is a proper map. There is also a notion
of an \'etale groupoid being {\em effective}. 

To an orbifold one can associate an 
effective proper \'etale groupoid as follows.
Given an orbifold $\Or$, 
a local model $(\widehat{U}_\alpha, G_\alpha)$ and some 
$\widehat{p}_\alpha \in \widehat{U}_\alpha$,
let $p \in |\Or|$ be the corresponding point. 
There is a quotient map
$A_{\widehat{p}_\alpha} : T_{\widehat{p}_\alpha} \widehat{U}_\alpha
\rightarrow C_p|\Or|$.
The unit space ${\mathcal G}^{(0)}$ is the
disjoint union of the $\widehat{U}_\alpha$'s. 
And ${\mathcal G}^{(1)}$ consists of the triples $(\widehat{p}_\alpha,
\widehat{p}_\beta, B_{\widehat{p}_\alpha,\widehat{p}_\beta})$ where
\begin{enumerate}
\item $\widehat{p}_\alpha \in \widehat{U}_\alpha$ and
$\widehat{p}_\beta \in \widehat{U}_\beta$,
\item $\widehat{p}_\alpha$ and $\widehat{p}_\beta$ map to the same point
$p \in |\Or|$ and
\item $B_{\widehat{p}_\alpha,\widehat{p}_\beta} : 
T_{\widehat{p}_\alpha} \widehat{U}_\alpha \rightarrow 
T_{\widehat{p}_\beta} \widehat{U}_\beta$ is an invertible linear map
so that $A_{\widehat{p}_\alpha} = A_{\widehat{p}_\beta} \circ
B_{\widehat{p}_\alpha,\widehat{p}_\beta}$.
\end{enumerate}
There is an obvious way to compose triples
$(\widehat{p}_\alpha,
\widehat{p}_\beta, B_{\widehat{p}_\alpha,\widehat{p}_\beta})$ and
$(\widehat{p}_\beta,
\widehat{p}_\gamma, B_{\widehat{p}_\beta,\widehat{p}_\gamma})$.
One can show that this gives rise to a smooth effective proper \'etale
groupoid.

Conversely, given a smooth effective proper \'etale groupoid
${\mathcal G}$, for any $\widehat{p} \in {\mathcal G}^{(0)}$
the isotropy group ${\mathcal G}_{\widehat{p}}^{\widehat{p}}$ is a
finite group. To get an orbifold, one can take local models of the form
$(\widehat{U}, {\mathcal G}_{\widehat{p}}^{\widehat{p}})$ where
$\widehat{U}$ is a ${\mathcal G}_{\widehat{p}}^{\widehat{p}}$-invariant
neighborhood of $\widehat{p}$.

Speaking hereafter just of smooth effective proper \'etale groupoids,
Morita-equivalent groupoids give equivalent orbifolds.

A groupoid morphism gives rise to an orbifold map. Taking
into account Morita equivalence, from the
groupoid viewpoint the right notion of an orbifold map would
be a Hilsum-Skandalis map between groupoids.
These turn out
to correspond to {\em good maps} between orbifolds, as later
defined by Chen-Ruan
\cite{Adem-Leida-Ruan (2007)}. This is a more restricted class of maps
between orbifolds than what we consider. The distinction is that
one can pull back orbivector bundles under good maps, but
not always under smooth maps in our sense.  Orbifold
diffeomorphisms in our sense are automatically good maps.
For some purposes it would be preferable to only deal with
good maps, but for simplicity we will stick with our
orbifold definitions. 

A Lie groupoid ${\mathcal G}$ has a classifying space $B{\mathcal G}$.
In the orbifold case, if ${\mathcal G}$ is the \'etale groupoid 
associated to an orbifold $\Or$ then
$\pi_1(\Or) \cong \pi_1(B{\mathcal G})$. The definition of the latter
can be made explicit in terms of paths and homotopies; see
\cite{Bridson-Haefliger (1999),Haefliger (1990)}.
In the case of effective orbifolds, the definition is equivalent to
the one of the present paper.

More information is
in \cite{Adem-Leida-Ruan (2007),Moerdijk} and references therein.
\end{remark}

\subsection{Low-dimensional orbifolds} \label{subsect2.3}

We list the connected compact boundaryless orbifolds of low dimension.
We mostly restrict here to the orientable case. (The nonorientable ones also 
arise; even if the total space of an orbifiber bundle
is orientable, the base may fail to be orientable.)

\subsubsection{Zero dimensions}

The only possibility is a point.

\subsubsection{One dimension}

There are two possibilities : $S^1$ and $S^1//\Z_2$. For the latter,
the nonzero element of $\Z_2$ acts by complex conjugation on
$S^1$, and $|S^1//\Z_2|$ is an interval. Note that
$S^1//\Z_2$ is not orientable.

\subsubsection{Two dimensions}

For notation, if $S$ is a connected oriented surface then
$S(k_1, \ldots, k_r)$ denotes the oriented orbifold $\Or$ with
$|\Or| = S$, having singular points of order $k_1, \ldots, k_r > 1$.
Any connected oriented $2$-orbifold can be written in this way.
An orbifold of the form $S^2(p,q,r)$ is called a {\em turnover}.

The {\em bad} orientable $2$-orbifolds are $S^2(k)$ and $S^2(k,k^\prime)$, 
$k \neq k^\prime$. The latter is simply-connected if and only if
$\gcd(k,k^\prime) = 1$.

The {\em spherical} $2$-orbifolds are of the form $S^2//\Gamma$, where
$\Gamma$ is a finite subgroup of $\Isom^+(S^2)$. The 
orientable ones are
$S^2$, $S^2(k,k)$, $S^2(2,2,k)$, $S^2(2,3,3)$, $S^2(2,3,4)$,
$S^2(2,3,5)$. (If $S^2(1,1)$ arises in this paper then it means $S^2$.)

The {\em Euclidean} $2$-orbifolds are of the form $T^2//\Gamma$, where
$\Gamma$ is a finite subgroup of $\Isom^+(T^2)$. The orientable ones are
$T^2$, $S^2(2,3,6)$, $S^2(2,4,4)$, $S^2(3,3,3)$, $S^2(2,2,2,2)$.
The latter is called a {\em pillowcase} and can be identified with the quotient
of $T^2 = \C/\Z^2$ by $\Z_2$, where the action of the 
nontrivial element of $\Z_2$ comes from the map $z \rightarrow -z$ on $\C$.

The other closed orientable $2$-orbifolds are hyperbolic.

We will also need some $2$-orbifolds with boundary, namely
\begin{itemize}
\item The {\em discal} $2$-orbifolds
$D^2(k) = D^2//\Z_k$.
\item The {\em half-pillowcase} $D^2(2,2) = I \times_{\Z_2} S^1$.
Here the nontrivial element of $\Z_2$ acts by involution on $I$
and by complex conjugation on $S^1$.
We can also write $D^2(2,2)$ as the quotient 
$\{z \in \C \: : \: \frac12 \le |z| \le 2\}//\Z_2$, where
the nontrivial element of $\Z_2$ sends $z$ to $z^{-1}$.
\item $D^2//\Z_2$, where $\Z_2$ acts by complex conjugation on
$D^2$. Then 
$\partial|D^2//\Z_2|$
is a circle with one
orbifold boundary component and one reflector component.
See Figure \ref{fig-1}, where the dark line indicates the
reflector component.

\begin{figure}[htb] 

\begin{center}  
\includegraphics[scale=1]{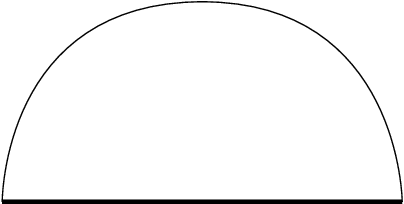} 
\caption{\label{fig-1}}
\end{center}

\end{figure}
\begin{figure}[h] 

\begin{center}  
\includegraphics[scale=.8]{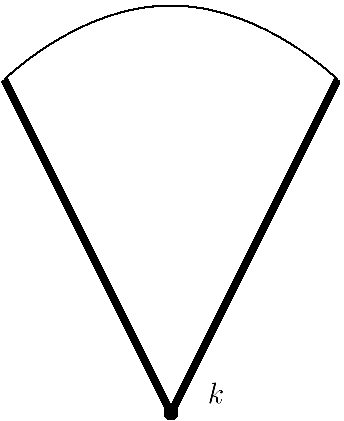} 
\caption{\label{fig-2}}
\end{center}

\end{figure}

\begin{figure}[h] 

\begin{center}  
\includegraphics[scale=1]{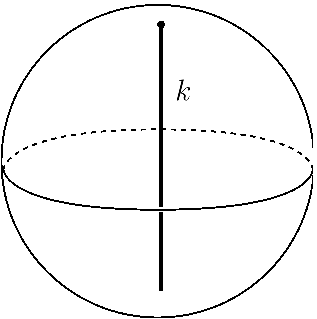} 
\caption{\label{fig-3}}
\end{center}
\end{figure}
\item $D^2//D_k = D^2(k)//\Z_2$, 
for $k > 1$, where $D_k$ is the dihedral group
and $\Z_2$ acts by complex conjugation on
$D^2(k)$. Then 
$\partial |D^2//D_k|$
is a circle with one
orbifold boundary component, one corner reflector point
of order $k$ and two reflector components. 
See Figure \ref{fig-2}.

\end{itemize}

\subsubsection{Three dimensions}

If $\Or$ is an orientable three-dimensional orbifold then
$|\Or|$ is an orientable topological $3$-manifold. 
If $\Or$ is boundaryless then $|\Or|$ is boundaryless.
Each component of the singular locus in $|\Or|$ is either 

\begin{enumerate}
\item a knot or arc (with endpoints on $\partial |\Or|$), 
labelled by an integer greater than one, or 
\item 
a trivalent graph with each edge labelled by an integer
greater than one, under the constraint that if edges with labels
$p,q,r$ meet at a vertex then $\frac{1}{p} + \frac{1}{q} + \frac{1}{r}
> 1$. That is, there is a neighborhood of the vertex which is a cone over an
orientable spherical $2$-orbifold.
\end{enumerate}

Specifying such a topological $3$-manifold and such a labelled 
graph is equivalent to specifying an orientable three-dimensional orbifold.

We write $D^3//\Gamma$ for a {\em discal} $3$-orbifold whose boundary is
$S^2//\Gamma$. 
They are
\begin{itemize}
\item $D^3$. There is no singular locus.
\item $D^3(k,k)$. The singular locus is a line segment through $D^3$.
See Figure \ref{fig-3}.

\item $D^3(2,2,k)$, $D^3(2,3,3)$, $D^3(2,3,4)$ and
$D^3(2,3,5)$. The singular locus is a tripod in $D^3$.
See Figure \ref{fig-4}.
\end{itemize}
\begin{figure}[h] 

\begin{center}  
\includegraphics[scale=1]{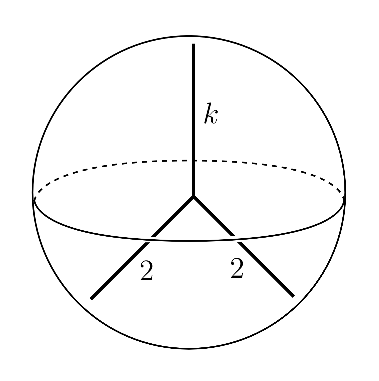} 
\caption{\label{fig-4}}
\end{center}
\end{figure}

\begin{figure}[h] 

\begin{center}  
\includegraphics[scale=1]{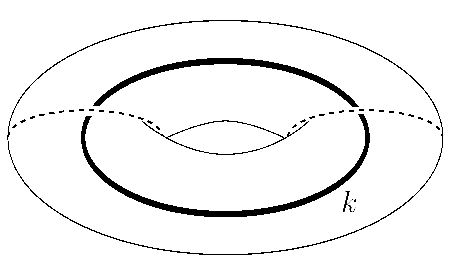} 
\caption{\label{fig-5}}
\end{center}
\end{figure}

\begin{figure}[h] 

\begin{center}  
\includegraphics[scale=1]{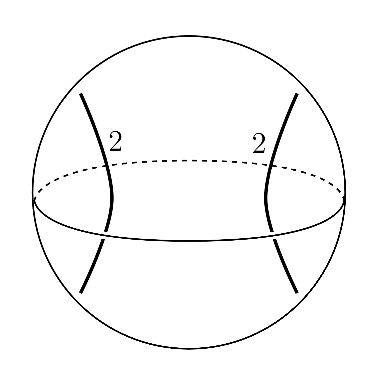} 
\caption{\label{fig-6}}
\end{center}
\end{figure}

The {\em solid-toric} $3$-orbifolds are 
\begin{itemize}
\item $S^1 \times D^2$. There is no singular locus.
\item $S^1 \times D^2(k)$. The singular locus is a core curve in a
solid torus.
See Figure \ref{fig-5}
\item
$S^1 \times_{Z_2} D^2$. The singular locus consists of two arcs in 
a $3$-disk, each labelled by $2$. The boundary is $S^2(2,2,2,2)$.
See Figure \ref{fig-6}.
\item $S^1 \times_{Z_2} D^2(k)$. The singular locus consists of
two arcs in a $3$-disk, each labelled by $2$, joined in their
middles by an arc
labelled by $k$. The boundary is $S^2(2,2,2,2)$.
See Figure \ref{fig-7}.
\end{itemize}

\begin{figure}[h] 

\begin{center}  
\includegraphics[scale=1]{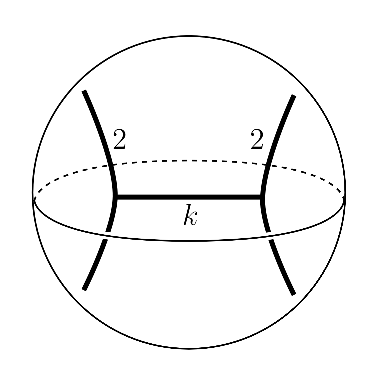} 
\caption{\label{fig-7}}
\end{center}
\end{figure}

\begin{figure}[h] 

\begin{center}  
\includegraphics[scale=1]{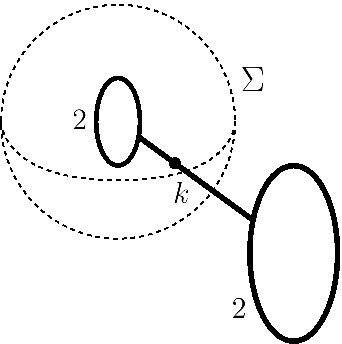} 
\caption{\label{fig-8}}
\end{center}
\end{figure}

\begin{figure}[h] 

\begin{center}  
\includegraphics[scale=1]{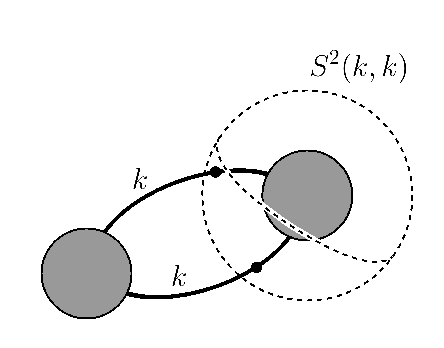} 
\caption{An essential spherical suborbifold\label{fig-9}}
\end{center}
\end{figure}

Given $\Gamma \in \Isom^+(S^2)$, we can consider the quotient
$S^3//\Gamma$ where $\Gamma$ acts on $S^3$ by the suspension of
its action on $S^2$. 
That is, we are identifying $\Isom^+(S^2)$ with $\SO(3)$ and using
the embedding $\SO(3) \rightarrow \SO(4)$ to let $\Gamma$ act on
$S^3$.

An orientable three-dimensional orbifold $\Or$ is {\em irreducible} if
it contains no embedded bad $2$-dimensional suborbifolds, and
any embedded orientable spherical $2$-orbifold $S^2//\Gamma$ bounds
a discal $3$-orbifold $D^3//\Gamma$ in $\Or$. 
Figure \ref{fig-8} shows an embedded bad $2$-dimensional suborbifold
$\Sigma$.
Figure \ref{fig-9} shows an embedded spherical
$2$-suborbifold $S^2(k,k)$ that does not bound a discal $3$-orbifold;
the shaded regions are meant to indicate some complicated
orbifold regions.

If $S$ is an orientable embedded $2$-orbifold in $\Or$ then $S$ is
{\em compressible} if there is an embedded discal 2-orbifold $D \subset \Or$ 
so that $\partial D$ lies in $S$, but $\partial D$ does not bound a discal
$2$-orbifold in $S$. (We call $D$ a {\em compressing discal orbifold}.) 
Otherwise, $S$ is {\em incompressible}. Note that any
embedded copy of a turnover $S^2(p,q,r)$ is automatically incompressible,
since any embedded circle in $S^2(p,q,r)$ bounds a discal $2$-orbifold
in $S^2(p,q,r)$.

If $\Or$ is a compact orientable $3$-orbifold then there is a
compact orientable irreducible $3$-orbifold $\Or^\prime$ so that
$\Or$ is the result of performing $0$-surgeries on $\Or^\prime$;
see \cite[Chapter 3]{BMP}.
The orbifold $\Or^\prime$ can be obtained by taking an appropriate
{\em spherical system} on $\Or$, cutting along the spherical
$2$-orbifolds and adding discal $3$-orbifolds to the ensuing
boundary components. 
If we take a minimal such spherical system then $\Or^\prime$ is
canonical.

Note that if $\Or = S^1 \times S^2$ then $\Or^\prime = S^3$.
This shows that if $\Or$ is a $3$-manifold then $\Or^\prime$ is not
just the disjoint components in the prime decomposition.
That is,
we are not
dealing with a direct generalization of the Kneser-Milnor prime decomposition 
from $3$-manifold theory. Because the
notion of connected sum is more involved for orbifolds than for
manifolds, the notion of a prime decomposition is also more
involved; see \cite{Hog-Angeloni-Matveev,Petronio}. 
It is not needed for the present paper.

We assume now that $\Or$ is irreducible. 
The {\em geometrization conjecture} says that if 
$\partial \Or = \emptyset$ and $\Or$ does not have
any embedded bad $2$-dimensional suborbifolds then there is a 
finite collection $\{S_i\}$ of incompressible orientable Euclidean 
$2$-dimensional suborbifolds of $\Or$ so that each connected component
of $\Or^\prime - \bigcup_i S_i$ is diffeomorphic to a 
quotient of one of the
eight Thurston geometries. Taking a minimal such collection of
Euclidean $2$-dimensional suborbifolds, the ensuing geometric pieces
are canonical.
References for the statement of the orbifold
geometrization conjecture are
\cite[Chapter 3.7]{BMP},\cite[Chapter 2.13]{CHK}.

Our statement of the orbifold geometrization conjecture is a
generalization of the manifold geometrization conjecture, as
stated in \cite[Section 6]{Scott} and \cite[Conjecture 1.1]{Thurston3}.
The cutting of the orientable three-manifold is along two-spheres and
two-tori.  An alternative version of the geometrization
conjecture requires the pieces to have finite volume
\cite[Conjecture 2.2.1]{Morgan}. In this version one must
also allow cutting along one-sided Klein bottles.
A relevant example to illustrate this point 
is when the three-manifold is the result
of gluing $I \times_{\Z_2} T^2$ to a cuspidal truncation of
a one-cusped complete noncompact finite-volume hyperbolic $3$-manifold.

\subsection{Seifert $3$-orbifolds} \label{subsect2.4}

A Seifert orbifold is the orbifold version of the total space of
a circle bundle.
We refer to \cite[Chapters 2.4 and 2.5]{BMP} for
information about Seifert $3$-orbifolds.  We just recall a few
relevant facts.

A Seifert $3$-orbifold fibers $\pi : \Or \rightarrow {\mathcal B}$
over a $2$-dimensional orbifold
${\mathcal B}$, with circle fiber. If $(\widehat{U}, G_p)$ is a local
model around $p \in |{\mathcal B}|$ then there is
a neighborhood $V$ of $|\pi|^{-1}(p) \subset |\Or|$ 
so that $\Or \Big|_{V}$ is diffeomorphic to 
$(S^1 \times \widehat{U})//G_p$, where $G_p$ acts on $S^1$ via a
representation $G_p \rightarrow O(2)$.
We will only consider orientable Seifert $3$-orbifolds.
so the elements of $G_p$ that preserve orientation
on $\widehat{U}$ will act on $S^1$ via $\SO(2)$, while the elements of
$G_p$ that reverse orientation on $\widehat{U}$ will act on $S^1$ via
$O(2) - \SO(2)$. In particular, if $p \in |{\mathcal B}|_{reg}$
then $|f|^{-1}(p)$ is a circle, while if $p \notin |{\mathcal B}|_{reg}$ 
then $|f|^{-1}(p)$ may be an interval.
We may loosely talk about the circle fibration of $\Or$.

As $\partial \Or$ is an orientable $2$-orbifold which fibers over a 
$1$-dimensional orbifold, with circle fibers, any connected component
of $\partial \Or$ must be $T^2$ or $S^2(2,2,2,2)$.
In the case of a boundary component
$S^2(2,2,2,2)$, the generic fiber is
a circle on $|S^2(2,2,2,2)|$ which separates it into two $2$-disks,
each containing two singular points. That is, the pillowcase is
divided into two half-pillowcases.

A solid-toric orbifold $S^1 \times D^2$ or $S^1 \times D^2(k)$ has an
obvious Seifert fibering over $D^2$ or $D^2(k)$.
Similarly, a solid-toric orbifold $S^1 \times_{\Z_2} D^2$ or 
$S^1 \times_{\Z_2} D^2(k)$ fibers over
$D^2//\Z_2$ or $D^2(k)//\Z_2$. 

\subsection{Riemannian geometry of orbifolds} \label{subsect2.5}

\begin{definition}
A {\em Riemannian metric} on an orbifold 
$\Or$ is given by 
an atlas for $\Or$ along with
a collection of Riemannian
metrics on the $\hU_\alpha$'s so that 
\begin{itemize}
\item $G_\alpha$ acts isometrically on $\hU_\alpha$ and 
\item The embeddings $(\hU_3, G_3) \rightarrow (\hU_1, G_1)$
and $(\hU_3, G_3) \rightarrow (\hU_2, G_2)$ from part 5 of
Definition \ref{defn2.1} are isometric.
\end{itemize}
\end{definition}

We say that the Riemannian orbifold $\Or$ has sectional curvature bounded
below by $K \in \R$ if the Riemannian metric on each $\hU_\alpha$ has
sectional curvature bounded below by $K$, and similarly for other
curvature bounds.

A Riemannian orbifold has an {\em orthonormal frame bundle} $F\Or$, a smooth
manifold with a
locally free (left) $O(n)$-action whose quotient space is
homeomorphic to $|\Or|$. Local charts for $F\Or$ are given by
$O(n) \times_G \hU$.
Fixing a bi-invariant Riemannian metric on $O(n)$, there is a
canonical $O(n)$-invariant Riemannian metric on $F\Or$.

Conversely, if $Y$ is a smooth connected
manifold with
a locally free $O(n)$-action then the slice theorem 
\cite[Corollary VI.2.4]{Bredon}
implies
that for each $y \in Y$, the $O(n)$-action near the orbit
$O(n) \cdot y$ is modeled by
the left $O(n)$-action on $O(n) \times_{G_y} \R^N$, where the
finite stabilizer
group $G_y \subset O(n)$ acts linearly on $\R^N$.  There is a corresponding
$N$-dimensional
orbifold $\Or$ with local models given by the pairs $(\R^N, G_y)$. 
If $Y_1$ and $Y_2$ are 
two such manifolds and $F \: : \: Y_1 \rightarrow Y_2$
is an $O(n)$-equivariant diffeomorphism then there is an induced
quotient
diffeomorphism $f \: : \: \Or_1 \rightarrow \Or_2$, as can be seen
by applying the slice theorem. 

If $Y$ has an $O(n)$-invariant
Riemannian metric then $\Or$ obtains a quotient Riemannian metric.

\begin{remark} \label{remark2.10}
Suppose that a discrete group $\Gamma$ acts properly discontinuously
on an orbifold $\Or$. Then there is a $\Gamma$-invariant Riemannian
metric on $\Or$. Furthermore, $\Gamma$ acts freely on $F\Or$, commuting
with the $O(n)$-action.  Hence there is a locally free $O(n)$-action
on the manifold $F\Or/\Gamma$ and a corresponding orbifold
$\Or//\Gamma$.
\end{remark}

There is a horizontal distribution $T^H F\Or$ on $F\Or$ coming from the
Levi-Civita connection on $\hU$. 
If $\gamma$ is a loop at $p \in |\Or|$ then
a horizontal lift of $\gamma$ allows one to define the {\em holonomy}
$H_\gamma$,
a linear map from $T_p\Or$ to itself.

If $\gamma \: : \: [a,b] \rightarrow \Or$ is a smooth map to a 
Riemannian orbifold then its
{\em length} is $L(\gamma) \: = \: \int_a^b |\gamma^\prime(t)| \: dt$,
where $|\gamma^\prime(t)|$ can be defined by a local lifting of $\gamma$
to a local model. This induces a length structure on $|\Or|$.
The {\em diameter} of $\Or$ is the diameter of $|\Or|$.
We say that $\Or$ is {\em complete} if $|\Or|$ is a complete metric space.
If $\Or$ has sectional curvature bounded below by $K \in \R$ then
$|\Or|$ has Alexandrov curvature bounded below by $K$,
as can be seen from the fact that the Alexandrov condition is
preserved upon quotienting by a finite group acting isometrically
\cite[Proposition 10.2.4]{Burago-Burago-Ivanov}.

It is useful to think of $\Or$ as consisting of an Alexandrov space 
equipped with an
additional structure that allows one to make sense of smooth functions.

We write $\dvol$ for the $n$-dimensional Hausdorff measure on $|\Or|$.
Using the above-mentioned relationship between the sectional curvature
of $\Or$ and the Alexandrov curvature of $|\Or|$, we can use
\cite[Chapter 10.6.2]{Burago-Burago-Ivanov} to extend
the Bishop-Gromov inequality from Riemannian manifolds with a
lower sectional curvature bound, to 
Riemannian orbifolds with a lower sectional curvature bound.
We remark that a Bishop-Gromov inequality for an orbifold with a
lower Ricci curvature bound appears in \cite{Borzellino2}.

A {\em geodesic} is a smooth curve $\gamma$ which, in local charts,
satisfies the geodesic equation.  Any length-minimizing curve 
$\gamma$ between
two points is a geodesic, as can be seen by looking in a local model
around $\gamma(t)$.

\begin{lemma}
If $\Or$ is a complete Riemannian orbifold then for any
$p \in |\Or|$ and any 
$v \in C_p|\Or|$,
there is a unique
geodesic $\gamma \: : \: \R \rightarrow \Or$ such that
$|\gamma|(0) = p$
and $\gamma^\prime(0) = v$.
\end{lemma}
\begin{proof}
The proof is similar to the proof of the corresponding part of the
Hopf-Rinow theorem, as in
\cite[Theorem 4.1]{Kobayashi-Nomizu (1963)}.
\end{proof}

The {\em exponential map} 
of a complete orbifold $\Or$
is defined as follows. Given $p \in |\Or|$ and $v \in C_p|\Or|$,
let $\gamma : [0,1] \rightarrow \Or$ be the unique geodesic with
$|\gamma|(0) = p$ and $|\gamma^\prime|(0) = v$. Put
$|\exp|(p,v) = (p, |\gamma|(1)) \in |\Or| \times |\Or|$.
This has the local lifting property to define a smooth orbifold
map $\exp \: : \: T\Or \rightarrow \Or \times \Or$.

Given $p \in |\Or|$, the restriction of $\exp$ to $T_p \Or$
gives an orbifold map $\exp_p : T_p \Or \rightarrow \Or$ so that
$|\exp|(p,v) = (p, |\exp_p|(v))$.   

Similarly, if $\Or^\prime$ is a suborbifold of $\Or$ then
there is a {\em normal exponential map}
$\exp : N\Or^\prime \rightarrow \Or$. If
$\Or^\prime$ is compact then for small $\epsilon > 0$,
the restriction of $\exp$ to the open $\epsilon$-disk bundle in $N\Or^\prime$
is a diffeomorphism to $\Or \Big|_{N_\epsilon(|\Or^\prime|)}$.

\begin{remark} \label{remark2.10.5}
To prove Lemma \ref{lemma2.4}, we can give the proper surjective
submersion $f : \Or_1 \rightarrow \Or_2$ a Riemannian submersion
metric in the orbifold sense.  Given $p \in |\Or_2|$,
let $U$ be a small $\epsilon$-ball around $p$ and
let $(\widehat{U}, G_p)$ be a local model with $\widehat{U}/G_p = U$.
Pulling back $f \Big|_{f^{-1}(U)} : f^{-1}(U) \rightarrow U$
to $\widehat{U}$, we obtain 
a $G_p$-equivariant Riemannian submersion $\widehat{f}$ to
$\widehat{U}$. If $\widehat{p} \in \widehat{U}$ covers $p$ then
$\widehat{f}^{-1}(\widehat{p})$ is a compact orbifold on which
$G_p$ acts. 
Using the submersion structure,
its normal bundle $N \widehat{f}^{-1}(\widehat{p})$ is $G_p$-diffeomorphic to
$\widehat{f}^{-1}(\widehat{p}) \times T_{\widehat{p}} \widehat{U}$.
If $\epsilon$ is sufficiently small then the normal exponential
map on the $\epsilon$-disk bundle in $N \widehat{f}^{-1}(\widehat{p})$
provides a $G_p$-equivariant product neighborhood
$\widehat{f}^{-1}(\widehat{p}) \times \widehat{U}$ of
$\widehat{f}^{-1}(\widehat{p})$;
cf. \cite[Pf. of Theorem 9.42]{Besse}.
This passes to a diffeomorphism between $f^{-1}(U)$ and
$(\widehat{f}^{-1}(\widehat{p}) \times \widehat{U})//G_p$.
\end{remark}

If $f \: : \: \Or_1 \rightarrow \Or_2$ is a local diffeomorphism 
and $g_2$ is a Riemannian metric on $\Or_2$ then there is a pullback
Riemannian metric $f^* g_2$ on $\Or_1$, which makes $f$ into a local
isometry.

We now give a useful criterion for a local isometry to be a covering map.

\begin{lemma} \label{lemma2.11}
If $f \: : \: \Or_1 \rightarrow \Or_2$ is a local isometry,
$\Or_1$ is complete and $\Or_2$ is connected then $f$ is a covering map.
\end{lemma}
\begin{proof}
The proof is along the lines of the corresponding manifold statement,
as in \cite[Theorem 4.6]{Kobayashi-Nomizu (1963)}.
\end{proof}

There is an orbifold version of the de Rham decomposition theorem.

\begin{lemma} \label{lemma2.12}
Let $\Or$ be connected, simply-connected and complete.
Given $p \in |\Or|_{reg}$, suppose that there is an orthogonal splitting
$T_p\Or = E_1 \oplus E_2$ which is invariant under holonomy around loops
based at $p$. Then there is an isometric splitting $\Or = \Or_1 \times \Or_2$
so that if we write $p = (p_1, p_2)$ then $T_{p_1} \Or_1 = E_1$ and
$T_{p_2} \Or_2 = E_2$.
\end{lemma}
\begin{proof}
The parallel transport of $E_1$ and $E_2$ defines involutive distributions
$D_1$ and $D_2$, respectively,
on $\Or$. Let $\Or_1$ and $\Or_2$ be maximal integrable
suborbifolds through $p$ for $D_1$ and $D_2$, respectively. 

Given a smooth curve $\gamma : [a,b] \rightarrow \Or$ starting at $p$,
there is a development $C : [a,b] \rightarrow T_p\Or$ of $\gamma$, as
in \cite[Section III.4]{Kobayashi-Nomizu (1963)}. Let $C_1 : [a,b] \rightarrow E_1$ and
$C_2 : [a,b] \rightarrow E_2$ be the orthogonal projections of $C$.
Then there are undevelopments $\gamma_1 : [a,b] \rightarrow \Or_1$ and
$\gamma_2 : [a,b] \rightarrow \Or_2$ of $C_1$ and $C_2$, respectively.

As in \cite[Lemma IV.6.6]{Kobayashi-Nomizu (1963)}, 
one shows that 
$(|\gamma_1|(b), |\gamma_2|(b))$
only depends on $|\gamma|(b)$.
In this way, one defines a map
$f : \Or \rightarrow \Or_1 \times \Or_2$. As in 
\cite[p. 192]{Kobayashi-Nomizu (1963)}, one shows that $f$ is a local isometry.
As in \cite[p. 188]{Kobayashi-Nomizu (1963)}, one shows that 
$\Or_1$ and $\Or_2$
are simply-connected. The lemma now follows from Lemma \ref{lemma2.11}.
\end{proof}

The regular part $|\Or|_{reg}$ inherits a Riemannian metric.
The corresponding volume form equals the $n$-dimensional Hausdorff 
measure on $|\Or|_{reg}$. We
define $\vol(\Or)$, or $\vol(|\Or|)$, 
to be the volume of the Riemannian manifold
$|\Or|_{reg}$, which equals the $n$-dimensional Hausdorff mass
of the metric space $|\Or|$. 

If $f \: : \: \Or_1 \rightarrow \Or_2$ is a diffeomorphism between
Riemannian orbifolds 
$(\Or_1, g_1)$ and $(\Or_2, g_2)$
then we can
define the $C^K$-distance between $g_1$ and $f^* g_2$,
using local models for $\Or_1$.

A {\em pointed orbifold} $(\Or, p)$ consists of an orbifold $\Or$ and
a basepoint $p \in |\Or|$. Given $r > 0$, we can consider the
pointed suborbifold $\check{B}(p, r) = \Or \Big|_{B(p,r)}$. 

\begin{definition}
Let $(\Or_1, p_1)$ and $(\Or_2, p_2)$ be pointed connected 
orbifolds with complete
Riemannian metrics $g_1$ and $g_2$ that are
$C^K$-smooth.  (That is, the orbifold transition maps are $C^{K+1}$
and the metric tensor in a local model is $C^K$.) 
Given $\epsilon > 0$, 
we say
that the $C^K$-distance between $(\Or_1, p_1)$ and $(\Or_2, p_2)$ is bounded
above by $\epsilon$ if
there is a $C^{K+1}$-smooth map $f \: : \: 
\check{B}(p_1, \epsilon^{-1}) \rightarrow \Or_2$ 
that is a diffeomorphism onto its image, such that
\begin{itemize}
\item The $C^K$-distance between $g_1$ and
$f^* g_2$ on $B(p_1, \epsilon^{-1})$ is at most $\epsilon$, and
\item $d_{|\Or_2|}(|f|({p}_1), {p}_2) \: \le \: \epsilon$.
\end{itemize}

Taking the infimum of all such possible $\epsilon$'s defines 
the $C^K$-distance between $(\Or_1, p_1)$ and $(\Or_2, p_2)$.
\end{definition}

\begin{remark} \label{remark2.14}
It may seem more natural to
require $|f|$ to be basepoint-preserving.  However, this would cause problems.
For example, given $k \: \ge \: 2$,
take $\Or = \R^2//\Z_k$.
Let $\pi \: : \: \R^2 \rightarrow |\Or|$ be the quotient map.
We would like to say that if $i$ is large then
the pointed orbifold $(\Or, \pi(i^{-1}, 0))$ is close to
$(\Or, \pi(0, 0))$. However, there is no {\em basepoint-preserving} map
$f \: : \: \check{B}(\pi(i^{-1}, 0), 1) \rightarrow 
(\Or, \pi(0,0))$ which is a diffeomorphism onto its image, due to the
difference between the local groups at the two basepoints.
\end{remark}

\subsection{Critical point theory for distance functions} \label{subsect2.6}

Let $\Or$ be a complete Riemannian orbifold and let $Y$ be a closed
subset of $|\Or|$. A point $p \in |\Or| - Y$ is {\em noncritical} if
there is a nonzero $G_p$-invariant
vector $v \in T_p \Or \cong T_{\widehat{p}} \hU$ making an angle strictly
larger than $\frac{\pi}{2}$ with any lift to $T_{\widehat{p}} \hU$
of the initial velocity of any
minimizing geodesic segment from $p$ to $Y$. 

In the next lemma we give an equivalent formulation in terms
of noncriticality on $|\Or|$.

\begin{lemma} \label{lemma2.15}
A point $p \in |\Or| - Y$ is noncritical if and only if
there is some $w \in C_p|\Or| \cong
T_{\widehat{p}}\hU/G_p$ 
so that the comparison angle between $w$ and
any minimizing geodesic from $p$ to $Y$ is strictly greater than
$\frac{\pi}{2}$.
\end{lemma}
\begin{proof}
Suppose that $p$ is noncritical. Given $v$ as in the
definition of noncriticality, put $w = vG_p$.

Conversely, suppose that
$w \in C_p|\Or| \cong
T_{\widehat{p}}\hU/G_p$ 
is such that the comparison angle between $w$ and
any minimizing geodesic from $p$ to $Y$ is strictly greater than
$\frac{\pi}{2}$. Let $v_0$ be a preimage of $w$ in
$T_{\widehat{p}}\hU$. Then $v_0$ makes an angle greater than $\frac{\pi}{2}$
with 
any lift to $T_{\widehat{p}}\hU$ of
the initial velocity of any minimizing geodesic from $p$ to $Y$.
As the set of such initial velocities is $G_p$-invariant, 
for any $g \in G_p$ the vector
$v_0 g$ also makes an angle greater than $\frac{\pi}{2}$
with any lift to $T_{\widehat{p}}\hU$ of
the initial velocity of any minimizing geodesic from $p$ to $Y$.
As $\{v_0 g\}_{g \in G_p}$ lies in an open half-plane, we can take
$v$ to be the nonzero vector $\frac{1}{|G_p|} \sum_{g \in G_p} v_0 g$.
\end{proof}

We now prove the main topological implications of noncriticality.

\begin{lemma} \label{lemma2.16}
If $Y$ is compact and
there are no critical points in the set $d_Y^{-1}(a,b)$ then there is a
smooth vector field $\xi$ on $\Or \Big|_{d_Y^{-1}(a,b)}$ so that 
$d_Y$ has uniformly positive directional derivative in the $\xi$ direction.
\end{lemma}
\begin{proof}
The proof is similar to that of \cite[Lemma 1.4]{Cheeger (1991)}.
For any $p \in |\Or| - Y$, there are a precompact neighborhood 
$U_p$ of $p$ in $|\Or| - Y$ and a smooth vector field
$V_p$ on $U_p$ so that $d_Y$ has positive directional derivative
in the $V_p$ direction, on $U_p$. Let $\{U_{p_i}\}$ be a finite
collection that covers $d_Y^{-1}(a,b)$. From Lemma \ref{lemma2.6},
there is a subordinate partition of unity $\{\rho_i\}$.
Put $\xi = \sum_i \rho_i V_i$.
\end{proof}

\begin{lemma} \label{lemma2.17}
If $Y$ is compact and
there are no critical points in the set $d_Y^{-1}(a,b)$ then
$\Or \Big|_{d_Y^{-1}(a,b)}$ is diffeomorphic to a product orbifold
$\R \times \Or^\prime$.
\end{lemma}
\begin{proof}
Construct $\xi$ as in Lemma \ref{lemma2.16}.
Choose $c \in (a,b)$. Then $\Or \Big|_{d_Y^{-1}(c)}$ is a
Lipschitz-regular suborbifold of $\Or$ which is transversal to
$\xi$, as can be seen in local models.
Working in local models, inductively from lower-dimensional strata
of $|\Or|$ to higher-dimensional strata,
we can slightly smooth $\Or \Big|_{d_Y^{-1}(c)}$ to form a smooth
suborbifold $\Or^\prime$ of $\Or$ which is transverse to $\xi$.
Flowing  (which is defined using local
models) in the direction of $\xi$ gives an orbifold diffeomorphism between
$\Or \Big|_{d_Y^{-1}(a,b)}$ and $\R \times \Or^\prime$.
\end{proof}

\subsection{Smoothing functions} \label{subsect2.7}

Let $\Or$ be a Riemannian orbifold.
Let $F$ be a Lipschitz function on $|\Or|$. Given $p \in |\Or|$,
we define the generalized gradient $\nabla^{gen}_p F \subset
T_p\Or$ as follows. Let $(\hU, G)$ be a local model around $p$.
Let $\widehat{F}$ be the lift of $F$ to $\hU$.
Choose $\widehat{p} \in \hU$ covering $p$.
Let $\epsilon > 0$ be small enough so that $\exp_{\widehat{p}} : 
B(0, \epsilon) \rightarrow \hU$ is a diffeomorphism onto its image. If
$\widehat{x} \in B(\widehat{p}, \epsilon)$ is a point of differentiability
of $\widehat{F}$ then compute $\nabla_{\widehat{x}} \widehat{F}$ and
parallel transport it along the minimizing geodesic to $\widehat{p}$.
Take the closed convex hull of the vectors so obtained and then take
the intersection as $\epsilon \rightarrow 0$. 
This gives a closed
convex $G_p$-invariant subset of $T_{\widehat{p}} \hU$,
or equivalently a closed convex subset of $T_p\Or$; we denote this set by
$\nabla^{gen}_p F$. The union $\bigcup_{p \in |\Or|} 
\nabla^{gen}_p F \subset T\Or$ will be denoted $\nabla^{gen}F$.

\begin{lemma} \label{lemma2.18}
Let $\Or$ be a complete Riemannian orbifold and let
$|\pi| : |T\Or| \rightarrow |\Or|$ be the projection map. Suppose
that $U \subset |\Or|$ is an open set, $C \subset U$ is a compact
subset and $S$ is an open fiberwise-convex subset of 
$T\Or \Big|_{|\pi|^{-1}(U)}$. (That is, $S$ is an open subset of
$|\pi|^{-1}(U)$ and for each $p \in |\Or|$, the preimage of $(S \cap
|\pi|^{-1}(p)) \subset C_p|\Or|$ in $T_p \Or$ is convex.)

Then for any $\epsilon > 0$ and any Lipschitz function $F : |\Or| \rightarrow 
\R$ whose generalized gradient over $U$ lies in $S$, there is a Lipschitz
function $F^\prime : |\Or| \rightarrow \R$ such that :
\begin{enumerate}
\item There is an open subset of $|\Or|$ containing $C$ on which
$F^\prime$ is a smooth orbifold function.
\item The generalized gradient of $F^\prime$, over $U$, lies in $S$.
\item $|F^\prime - F|_\infty \le \epsilon$.
\item $F^\prime \Big|_{|\Or| - U} = F \Big|_{|\Or| - U}$.
\end{enumerate}
\end{lemma}
\begin{proof}
The proof proceeds by mollifying the Lipschitz function $F$ as in
\cite[Section 2]{Grove-Shiohama (1977)}. The mollification there is
clearly $G$-equivariant in a local model $(\hU, G)$.
\end{proof}

\begin{corollary} \label{corollary2.19}
For all $\epsilon > 0$ there is a $\theta > 0$ with the following property.

Let $\Or$ be a complete Riemannian orbifold, let $Y \subset |\Or|$ be a
closed subset and let $d_Y : |\Or| \rightarrow \R$ be the distance function
from $Y$. Given $p \in |\Or| - Y$, let $V_p \subset C_p|\Or|$ be the
set of initial velocities of minimizing geodesics from $p$ to $Y$.
Suppose that $U \subset |\Or| - Y$ is an open subset such that  for all
$p \in U$, one has $\diam(V_p) < \theta$. Let $C$ be a compact subset of
$U$. Then for every $\epsilon_1 > 0$, there is a Lipschitz function
$F^\prime : |\Or| \rightarrow \R$ such that
\begin{itemize}
\item $F^\prime$ is smooth on a neighborhood of $C$. 
\item $\parallel F^\prime - d_Y \parallel_\infty < \epsilon_1$.
\item $F^\prime \Big|_{M-U} = d_Y \Big|_{M-U}$
\item For every $p \in C$, the angle between $- \nabla_p F^\prime$ and
$V_p$ is at most $\epsilon$.
\item $F^\prime - d_Y$ is $\epsilon$-Lipschitz.
\end{itemize}
\end{corollary}

\section{Noncompact nonnegatively curved orbifolds} \label{sect3}

In this section we extend the splitting theorem and the
soul theorem from Riemannian manifolds to Riemannian orbifolds.
We give an argument to rule out tight necks in a
noncompact nonnegatively curved orbifold.  We give the
topological description of noncompact nonnegatively curved orbifolds
of dimension two and three.

\begin{assumption}
In this section, $\Or$ will be a complete 
nonnegatively curved Riemannian orbifold.
\end{assumption}
We may emphasize in some places that $\Or$ is nonnegatively curved.

\subsection{Splitting theorem} \label{subsect3.1}

\begin{proposition} \label{prop3.1}
If $|\Or|$ contains a line then $\Or$ is an isometric product
$\R \times \Or^\prime$ for some complete Riemannian orbifold $\Or^\prime$.
\end{proposition}
\begin{proof}
As $|\Or|$ contains a line,
the splitting theorem for nonnegatively curved Alexandrov spaces
\cite[Chapter 10.5]{Burago-Burago-Ivanov} implies that
$|\Or|$ is an isometric product $\R \times Y$ for some complete
nonnegatively curved Alexandrov space $Y$. The isometric splitting 
lifts to local models, showing that
$\Or \Big|_Y$ is an Riemannian orbifold $\Or^\prime$ and that
the isometry $|\Or| \rightarrow \R \times Y$ is a smooth
orbifold splitting
$\Or \rightarrow \R \times \Or^\prime$.
\end{proof}

\begin{corollary} \label{corollary3.2}
If $\Or$ has more than one end then it has two ends and
$\Or$ is an isometric product
$\R \times \Or^\prime$ for some compact Riemannian orbifold $\Or^\prime$.
\end{corollary}

\begin{remark}
A splitting theorem for orbifolds with nonnegative Ricci curvature
appears in \cite{Borzellino3}. As the present paper deals with lower sectional
curvature bounds, the more elementary Proposition \ref{prop3.1} is
sufficient for our purposes.
\end{remark}

\subsection{Cheeger-Gromoll-type theorem} \label{subsect3.2}

A subset $Z \subset |\Or|$ is {\em totally convex} if
any geodesic segment (possibly not minimizing) with endpoints in $Z$
lies entirely in $Z$.

\begin{lemma} \label{lemma3.3}
Let $Z \subset |\Or|$ be totally convex and let $(\hU, G)$ be a local model.
Put $U = \hU/G$ and let $q : \hU \rightarrow U$ be the quotient map.
If $\gamma$ is a geodesic segment in $\hU$ with endpoints in
$q^{-1}(U \cap Z)$ then $\gamma$ lies in $q^{-1}(U \cap Z)$.
\end{lemma}
\begin{proof}
Suppose that $\gamma(t) \notin q^{-1}(U \cap Z)$ for some $t$.
Then $q \circ \gamma$ is a geodesic in $\Or$ with endpoints in $Z$,
but $q(\gamma(t)) \notin Z$. This is a contradiction.
\end{proof}

\begin{lemma} \label{lemma3.4}
Let $Z \subset |\Or|$ be a closed totally convex set. Let $k$ be the
Hausdorff dimension of $Z$. Let ${\mathcal N}$ be the union of the 
$k$-dimensional
suborbifolds ${\mathcal S}$ of $\Or$ with $|{\mathcal S}| \subset Z$. 
Then ${\mathcal N}$ is a totally geodesic
$k$-dimensional suborbifold of $|\Or|$ and $Z = \overline{|{\mathcal N}|}$.
Furthermore, if $Y$ is a closed subset of $|{\mathcal N}|$ and
$p \in Z - |{\mathcal N}|$ then there is a
$v \in C_p|\Or|$
so that the initial velocity of 
any minimizing geodesic from $p$ to $Y$ makes an angle greater than 
$\frac{\pi}{2}$ with $v$.
\end{lemma}
\begin{proof}
Using Lemma \ref{lemma3.3}, the proof is along the lines of that in 
\cite[Chapter 3.1]{Gromoll-Walschap (2009)}.
\end{proof}

We put $\partial Z = Z - |{\mathcal N}|$. Note that in the definition
of ${\mathcal N}$ we are dealing with orbifolds as opposed to manifolds.
For example, if
$\Or \Big|_Z$ is a boundaryless $k$-dimensional orbifold then
$\partial Z = \emptyset$.

A function $f : |\Or| \rightarrow \R$ is {\em concave} if for any
geodesic segment $\gamma : [a,b] \rightarrow \Or$,
for all $c \in [a,b]$ one has
\begin{equation} \label{3.5}
f(|\gamma|(c)) \ge \frac{b-c}{b-a} f(|\gamma|(a)) + 
\frac{c-a}{b-a} f(|\gamma|(b)).
\end{equation}

\begin{lemma} \label{lemma3.6}
It is equivalent to require (\ref{3.5}) for all geodesic segments or just for
minimizing geodesic segments.
\end{lemma}
\begin{proof}
Suppose that (\ref{3.5}) holds for all minimizing geodesic segments.
Let $\gamma : [a,b] \rightarrow \Or$ be a geodesic segment, maybe
not minimizing.  For any $t \in [a,b]$, we can find a neighborhood
$I_t$ of $t$ in $[a,b]$ so that the restriction of $\gamma$ to
$I_t$ is minimizing.  Then (\ref{3.5}) holds on $I_t$. It follows
that (\ref{3.5}) holds on $[a,b]$.
\end{proof}

Any superlevel set $f^{-1}[c, \infty)$ of a concave function
is closed and totally convex.

Let $f$ be a proper concave function on $|\Or|$ which is bounded above.
Then there is a maximal $c \in \R$ so that the superlevel set
$f^{-1}[c, \infty)$ is nonempty, and so
$f^{-1}[c, \infty) = f^{-1}\{c\}$ is a closed totally convex set.

Suppose for the rest of this subsection that $\Or$ is noncompact.

\begin{lemma} \label{lemma3.7}
Let $Z \subset |\Or|$ be a closed totally convex set with $\partial Z \neq
\emptyset$. Then $d_{\partial Z}$ is a concave function on $Z$. 
Furthermore, suppose that for a minimizing geodesic 
$\gamma : [a,b] \rightarrow Z$ in $Z$, the
restriction of 
$d_{\partial Z} \circ |\gamma|$
is a constant positive function
on $[a,b]$. Let 
$t \rightarrow exp_{\gamma(a)} tX(a)$
be a minimizing
unit-speed geodesic from 
$|\gamma|(a)$
to $\partial Z$, defined for
$t \in [0,d]$. Let $\{X(s)\}_{s \in [a,b]}$ be the
parallel transport of $X(a)$ along $\gamma$. Then for any
$s \in [a,b]$, the curve $t \rightarrow exp_{\gamma(s)} tX(s)$ is a
minimal geodesic from 
$|\gamma|(s)$
to $\partial Z$, of length $d$.
Also, the rectangle $V : [a,b] \times [0,d] \rightarrow Z$ given by
$V(s,t) = \exp_{\gamma(s)} tX(s)$ is flat and totally geodesic.
\end{lemma}
\begin{proof}
The proof is similar to that of \cite[Theorem 3.2.5]{Gromoll-Walschap (2009)}.
\end{proof}

Fix a basepoint $\star \in |\Or|$. Let $\eta$ be a unit-speed
ray in $|\Or|$ starting
from $\star$; note that $\eta$ is automatically a geodesic. Let
$b_\eta : |\Or| \rightarrow \R$ be the Busemann function;
\begin{equation} \label{3.8}
b_\eta(p) = \lim_{t \rightarrow \infty} (d(p, \eta(t)) - t).
\end{equation}

\begin{lemma} \label{lemma3.9}
The Busemann function $b_\eta$ is concave.
\end{lemma}
\begin{proof}
The proof is similar to that of \cite[Theorem 3.2.4]{Gromoll-Walschap (2009)}.
\end{proof}

\begin{lemma} \label{lemma3.10}
Putting $f = \inf_{\eta} b_\eta$, where $\eta$ runs over unit speed rays
starting at $\star$, gives a proper 
concave
function on $|\Or|$ which is bounded above.
\end{lemma}
\begin{proof}
The proof is similar to that of 
\cite[Proposition 3.2.1]{Gromoll-Walschap (2009)}.
\end{proof}

We now construct the soul of $\Or$, following Cheeger-Gromoll
\cite{Cheeger-Gromoll}.  Let
$C_0$ be the minimal nonempty superlevel set of $f$. For $i \ge 0$,
if $\partial C_i \neq \emptyset$ then
let $C_{i+1}$ be the minimal nonempty superlevel set of $d_{\partial C_i}$
on $C_i$. Let $S$ be the nonempty $C_i$ so that 
$\partial C_{i} = \emptyset$.
Define the {\em soul} to be
${\mathcal S} = \Or \Big|_{S}$. Then ${\mathcal S}$ is a totally geodesic
suborbifold of $\Or$.

\begin{proposition} \label{prop3.11}
$\Or$ is diffeomorphic to the normal bundle ${\mathcal N}{\mathcal S}$ of 
${\mathcal S}$.
\end{proposition}
\begin{proof}
Following \cite[Lemma 3.3.1]{Gromoll-Walschap (2009)}, we claim that
$d_S$ has no critical points on $|\Or| - S$. To see this, choose
$p \in |\Or| - S$. There is a totally convex set 
$Z \subset |\Or|$ for which 
$p \in \partial Z$; either a superlevel set of $f$ or one of the
sets $C_i$. Defining ${\mathcal N}$ as in Lemma \ref{lemma3.4}, we also know
that $S \subset |{\mathcal N}|$. By Lemma \ref{lemma3.4}, $p$ is noncritical
for $d_S$.
  
From Lemma \ref{lemma2.17}, for small $\epsilon > 0$, we know that
$\Or$ is diffeomorphic to $\Or \Big|_{N_\epsilon(S)}$.
However, if $\epsilon$ is small then the normal exponential map
gives a diffeomorphism between ${\mathcal N}{\mathcal S}$ and 
$\Or \Big|_{N_\epsilon(S)}$.
\end{proof}

\begin{remark} One can define a soul for a general
complete nonnegatively
curved Alexandrov space $X$.  The soul will be homotopy equivalent
to $X$. However, $X$ need not be homeomorphic to a fiber bundle
over the soul,
as shown by an example of Perelman
\cite[Example 10.10.9]{Burago-Burago-Ivanov}.
\end{remark}

We include a result that we will need later about
orbifolds with locally convex boundary.

\begin{lemma} \label{lemma3.12}
Let $\Or$ be a compact connected orbifold-with-boundary with
nonnegative sectional curvature.
Suppose that $\partial \Or$ is nonempty and has positive-definite
second fundamental form.  Then there is some
$p \in |\Or|$ so that $\partial \Or$ is diffeomorphic to the unit 
distance sphere from the vertex in $T_p\Or$.
\end{lemma}
\begin{proof}
Let $p \in |\Or|$ be a point of maximal distance from 
$|\partial \Or|$. We claim that $p$ is unique. If not, let
$p^\prime$ be another such point and let $\gamma$ be a 
minimizing geodesic between them.
Applying Lemma \ref{lemma3.7} with $Z = |\Or|$,
there is a nontrivial geodesic $s \rightarrow V(s,d)$ of
$\Or$ that lies in $|\partial \Or|$. This contradicts the
assumption on $\partial \Or$. Thus $p$ is unique.
The lemma now follows from the proof of Lemma \ref{prop3.11},
as we are effectively in a situation where the soul is a point.
\end{proof}

\subsection{Ruling out tight necks in nonnegatively curved orbifolds}
\label{subsect3.3}

\begin{lemma} \label{lemma3.13}
Suppose that $\Or$ is a complete connected Riemannian orbifold
with nonnegative sectional curvature.
If $X$ is a compact connected $2$-sided codimension-$1$
suborbifold of $\Or$ then precisely one of the following occurs :
\begin{itemize}
\item $X$ is the boundary of a compact suborbifold of $\Or$. 
\item $X$ is nonseparating, $\Or$ is compact and
$X$ lifts to a $\Z$-cover $\Or^\prime \rightarrow \Or$, where
$\Or^\prime = \R \times \Or^{\prime \prime}$ with
$\Or^{\prime \prime}$ compact.
\item $X$ separates $\Or$ into two unbounded connected components
and $\Or = \R \times \Or^\prime$ with
$\Or^\prime$ compact.
\end{itemize}
\end{lemma}
\begin{proof}
Suppose that $X$ separates $\Or$. If both components
of $|\Or| - |X|$ are unbounded
then $\Or$ contains a line. From Proposition \ref{prop3.1},
$\Or = \R \times \Or^\prime$ for some $\Or^\prime$. As
$X$ is compact, $\Or^\prime$ must be compact.

The remaining case is when $X$ does not separate $\Or$. 
If $\gamma$ is a smooth closed curve in $\Or$ which is transversal to
$X$ (as defined in local models) then there is a well-defined
intersection number $\gamma \cdot X \in \Z$.
This gives a homomorphism $\rho : \pi_1(\Or, p) \rightarrow \Z$.
Since $X$ is nonseparating, there is a $\gamma$ so that
$\gamma \cdot X \neq 0$; hence the image of $\rho$ is an 
infinite cyclic group.
Put $\Or^\prime = \widetilde{\Or}/\Ker(\rho)$;
it is an infinite cyclic cover of $\Or$.
As $\Or^\prime$ contains
a line, the lemma follows from Proposition \ref{prop3.1}.
\end{proof}

\begin{lemma}
\label{lem-c'nm}
Suppose that $\R^n//G$ is a Euclidean 
orbifold with $G$ a finite subgroup of $O(n)$.
If $X\subset \R^n//G$ is a connected
compact $2$-sided codimension-$1$ suborbifold,  then $X$ bounds 
some $D\subset \R^n//G$ with $\diam_\Or(D)< 4|G| \diam_X(X)$,
where $\diam_\Or(D)$ denote the extrinsic diameter of $D$ in $|\Or|$ while
$\diam_X(X)$ denotes the intrinsic diameter of $X$.
\end{lemma}
\begin{proof}
Let $\widehat{X}$ be the preimage of $X$ in $\R^n$.
Let $\Delta$ be any number greater than $\diam_X(X)$.
Let $x$ be a point in $|X|$. 
Let $\{\widehat{x}_i\}_{i \in I}$ be the
preimages of $x$ in $\widehat{X}$. Here the cardinality of $I$ 
is bounded above by $|G|$. We claim that 
$\widehat{X} = \bigcup_{i \in I} B(\widehat{x}_i, \Delta)$,
where $B(\widehat{x}_i, \Delta)$ denotes a distance ball in $\widehat{X}$
with respect to its intrinsic metric.
To see this,
let $\widehat{y}$ be an arbitrary point in $\widehat{X}$. Let
$y$ be its image in $X$. Join $y$ to $x$ by a minimizing geodesic
$\gamma$
in $X$,
which is necessarily of length at most $\Delta$. Then a horizontal
lift of $\gamma$, starting at $\widehat{y}$, joins $\widehat{y}$ to
some $\widehat{x}_i$ and also has length at most $\Delta$.

Let $\widehat{C}$ be a connected component of $\widehat{X}$. 
Since $\widehat{C}$ is connected,
it has a covering by a subset of $\{B(\widehat{x}_i, 2\diam_X(X))\}_{i \in I}$
with connected nerve, and so $\widehat{C}$ has diameter at
most $4|G|\diam_X(X)$. Furthermore, from the Jordan separation theorem,
$\widehat{C}$ is the boundary of a
domain $\widehat{D} \in \R^n$ with extrinsic diameter at most
$4|G|\diam_X(X)$. 
Letting $D \in \Or$ be the projection of $\widehat{D}$, the
lemma follows.
\end{proof}

\begin{proposition} \label{prop3.14}
Suppose that $\Or$ is a complete connected noncompact Riemannian $n$-orbifold
with nonnegative sectional curvature.
Then there is a number $\delta > 0$ (depending on $\Or$) so that
the following holds. 
Let $X$ be a connected compact
$2$-sided codimension-$1$ suborbifold of $\Or$.
Then either
\begin{itemize}
\item $X$ bounds a connected suborbifold $D$ of $\Or$ with
$\diam_{\Or}(D) < 8 (\sup_{p \in |\Or|} |G_p|) \cdot \diam(X)$, or
\item $\diam(X) > \delta$.
\end{itemize}
\end{proposition}
\begin{proof}
Suppose that the proposition is not true.  Then there is a sequence
$\{X_i\}_{i=1}^\infty$ of connected compact $2$-sided
codimension-$1$ suborbifolds of $\Or$ so that 
$\lim_{i \rightarrow \infty} \diam(X_i) = 0$ but each
$X_i$ fails to bound a connected suborbifold whose
extrinsic diameter is at most $8 \sup_{p \in |\Or|} |G_p|$ times
as much.

If all of the $|X_i|$'s lie in a compact subset of $|\Or|$ then
a subsequence converges in the Hausdorff topology to a point 
$p \in |\Or|$. As a sufficiently small neighborhood of $p$ can be 
well approximated metrically
by a neighborhood of $0 \in |\R^n//G_p|$ after rescaling,
Lemma \ref{lem-c'nm} implies that
for large $i$ we can find $D_i \subset \Or$ with
$X_i = \partial D_i$ and $\diam_{\Or_i}(D_i) < 
8 \left( \sup_{p \in |\Or|} |G_p| \right) \cdot \diam(X_i)$.
This is a contradiction.  Hence we can assume that
the sets $|X_i|$ tend to infinity.

If some $X_i$ does not
bound a compact suborbifold of $\Or$ then 
by Lemma \ref{lemma3.13}, there is an isometric splitting $\Or = \R \times
\Or^\prime$ with $\Or^\prime$ compact. This contradicts the
assumed existence of the sequence $\{X_i\}_{i=1}^\infty$ with
$\lim_{i \rightarrow \infty} \diam(X_i) = 0$.
Thus we can assume that $X_i = \partial D_i$ for
some compact suborbifold $D_i$ of $\Or$.
If $\Or$ had more than one end then it would split off an $\R$-factor
and as before, the sequence $\{X_i\}_{i=1}^\infty$
would not exist.  Hence $\Or$ is one-ended and after passing to a
subsequence, we can assume that $D_1 \subset D_2 \subset \ldots$.
Fix a basepoint $\star \in |D_1|$. Let $\eta$ be a unit-speed
ray in $|\Or|$ starting from $\star$ and let $b_\eta$ be the
Busemann function from (\ref{3.8}).

Suppose that
$p, p^\prime \in |\Or|$ are such that 
$b_\eta(p) = b_\eta(p^\prime)$. For $t$ large, consider a
geodesic triangle
with vertices $p, p^\prime, \eta(t)$. Given $X_i$ with $i$ large, if $t$ is
sufficiently large then
$\overline{p \eta(t)}$ and $\overline{p^\prime \eta(t)}$ pass through
$X_i$. Taking $t \rightarrow \infty$, 
triangle comparison implies that $d(p, p^\prime) \le \diam(X_i)$.
Taking $i \rightarrow \infty$ gives $p = p^\prime$. Thus
$b_\eta$ is injective.  This is a contradiction.
\end{proof}

\subsection{Nonnegatively curved $2$-orbifolds} \label{subsect3.4}

\begin{lemma} \label{lemma3.15}
Let $\Or$ be a complete connected orientable $2$-dimensional
orbifold with nonnegative sectional curvature
which is $C^K$-smooth, $K \ge 3$. We have the
following classification of the diffeomorphism type, based on the
number of ends. For notation, $\Gamma$ denotes a finite subgroup
of the oriented isometry group of the relevant orbifold and
$\Sigma^2$ denotes a simply-connected bad $2$-orbifold with
some Riemannian metric.
\begin{itemize}
\item $0$ ends : $S^2//\Gamma$, $T^2//\Gamma$, $\Sigma^2//\Gamma$.
\item $1$ end : $\R^2//\Gamma$, $S^1 \times_{\Z_2} \R$.
\item $2$ ends : $\R \times S^1$.
\end{itemize}
\end{lemma}
\begin{proof}
If $\Or$ has zero ends then it is compact and the classification
follows from the orbifold Gauss-Bonnet theorem \cite[Proposition 2.9]{BMP}.
If $\Or$ has more than one end then Proposition \ref{prop3.1} implies that
$\Or$ has two ends and isometrically splits off an $\R$-factor.
Hence it must be diffeomorphic to $\R \times S^1$. 
Suppose that
$\Or$ has one end. The soul ${\mathcal S}$ has dimension $0$ or $1$.
If ${\mathcal S}$ has dimension zero then ${\mathcal S}$ is a point
and $\Or$ is diffeomorphic to the normal bundle of ${\mathcal S}$,
which is $\R^2//\Gamma$. If ${\mathcal S}$ has dimension one then
it is $S^1$ or $S^1//\Z_2$ and $\Or$ is diffeomorphic to the normal bundle of
${\mathcal S}$. As $S^1 \times \R$ has two ends, the only possibility
is $S^1 \times_{\Z_2} \R$.
\end{proof}

\subsection{Noncompact nonnegatively curved $3$-orbifolds}
\label{subsect3.5}

\begin{lemma} \label{lemma3.16}
Let $\Or$ be a complete connected noncompact
orientable $3$-dimensional
orbifold with nonnegative sectional curvature
which is $C^K$-smooth, $K \ge 3$. We have the
following classification of the diffeomorphism type, based on the
number of ends. For notation, $\Gamma$ denotes a finite subgroup
of the oriented isometry group of the relevant orbifold and
$\Sigma^2$ denotes a simply-connected bad $2$-orbifold with
some Riemannian metric.
\begin{itemize}
\item $1$ end : $\R^3//\Gamma$, $S^1 \times \R^2$,
$S^1 \times \R^2(k)$, $S^1 \times_{\Z_2} \R^2$,
$S^1 \times_{\Z_2} \R^2(k)$, 
$\R \times_{\Z_2} (S^2//\Gamma)$,
$\R \times_{\Z_2} (T^2//\Gamma)$ or
$\R \times_{\Z_2} (\Sigma^2//\Gamma)$.
\item $2$ ends : $\R \times (S^2//\Gamma)$, $\R \times (T^2//\Gamma)$ or
$\R \times (\Sigma^2//\Gamma)$.
\end{itemize}
\end{lemma}
\begin{proof}
Because $\Or$ is noncompact, it has at least one end.
If it has more than one end then Proposition \ref{prop3.1} implies that
$\Or$ has two ends and isometrically splits off an $\R$-factor. This gives rise
to the possibilities listed for two ends.

Suppose that 
$\Or$ has one end. The soul ${\mathcal S}$ has dimension $0$, $1$ or $2$.
If ${\mathcal S}$ has dimension zero then ${\mathcal S}$ is a point
and $\Or$ is diffeomorphic to the normal bundle of ${\mathcal S}$,
which is $\R^3//\Gamma$. If ${\mathcal S}$ has dimension one then
it is $S^1$ or $S^1//\Z_2$ and $\Or$ is diffeomorphic to the normal bundle of
${\mathcal S}$, which is $S^1 \times \R^2$,
$S^1 \times \R^2(k)$, $S^1 \times_{\Z_2} \R^2$ or
$S^1 \times_{\Z_2} \R^2(k)$. 
If ${\mathcal S}$ has dimension two then since it has nonnegative
curvature,
it is diffeomorphic to a quotient of $S^2$, $T^2$ or $\Sigma^2$. 
Then $\Or$ is diffeomorphic to the normal bundle of
${\mathcal S}$, which is 
$\R \times_{\Z_2} (S^2//\Gamma)$,
$\R \times_{\Z_2} (T^2//\Gamma)$ or
$\R \times_{\Z_2} (\Sigma^2//\Gamma)$, since
$\Or$ has one end.
\end{proof}

\subsection{$2$-dimensional nonnegatively curved orbifolds 
that are pointed Gromov-Hausdorff close
to an interval} \label{subsect3.6}

We include a result that we will need later about
$2$-dimensional nonnegatively curved orbifolds that are pointed 
Gromov-Hausdorff close to an interval.

\begin{lemma} \label{lemma3.17}
There is some ${\beta} > 0$ so that the following holds.
Suppose that $\Or$ is a pointed nonnegatively curved complete
orientable Riemannian $2$-orbifold which is $C^K$-smooth for some
$K \ge 3$. Let $\star \in |\Or|$ be a basepoint
and suppose that the pointed ball $(B(\star,10), \star) \subset |\Or|$
has pointed Gromov-Hausdorff distance at most $\beta$ from the
pointed interval $([0,10], 0)$. Then for every $r \in [1,9]$,
the orbifold $\Or \Big|_{\overline{B(\star, r)}}$ is a discal
$2$-orbifold or is diffeomorphic to $D^2(2,2)$.
\end{lemma}
\begin{proof}
As in \cite[Pf. of Lemma 3.12]{Kleiner-Lott2}, the distance function
$d_\star : A(\star, 1,9) \rightarrow
[1,9]$ defines a fibration with a circle fiber.
 
The possible diffeomorphism types of 
$\Or$ are listed in Lemma \ref{lemma3.15}.
Looking at them, if $\overline{B(\star, 1)}$ is not a topological disk then 
$\Or$ must be $T^2$ 
and we obtain a contradiction as in \cite[Pf. of Lemma 3.12]{Kleiner-Lott2}.
Hence $\overline{B(\star, 1)}$ is a topological disk.
If $\Or \Big|_{\overline{B(\star, 1)}}$ is not a discal
$2$-orbifold then it has at least two singular points,
say $p_1, p_2 \in |\Or|$. Choose $q \in |\Or|$ with 
$d(\star,q) = 2$. By triangle comparison, the comparison angles satisfy
$\cangle_{p_1}(p_2, q) \le \frac{2\pi}{|G_{p_1}|}$ and 
$\cangle_{p_2}(p_1, q) \le \frac{2\pi}{|G_{p_2}|}$.
If $\beta$ is 
small then $\cangle_{p_1}(p_2, q) + \cangle_{p_2}(p_1, q)$ is close to
$\pi$. It follows that $|G_{p_1}| = |G_{p_2}| =  2$.

Suppose that there are three distinct singular points
$p_1, p_2, p_3 \in |\Or|$. We know that they lie in
$\overline{B(\star, 1)}$.
Let $\overline{p_iq}$ and $\overline{p_k p_j}$ denote minimal geodesics.
If $\beta$ is small then the angle at $p_1$ between
$\overline{p_1q}$ and $\overline{p_1p_2}$ is close to $\frac{\pi}{2}$,
and similarly for the angle at $p_1$ between 
$\overline{p_1q}$ and $\overline{p_1p_3}$. As $\dim(\Or) = 2$, and
$p_1$ has total cone angle $\pi$,
it follows that if
$\beta$ is small then the angle at $p_1$ between
$\overline{p_1p_2}$ and $\overline{p_1p_2}$ is small.
The same reasoning applies at $p_2$ and $p_3$, so we have a geodesic
triangle in $|\Or|$ with small total interior angle, which violates the
fact that $|\Or|$ has nonnegative Alexandrov curvature.

Thus $\Or \Big|_{\overline{B(\star, 1)}}$
is diffeomorphic to $D^2(2,2)$.
\end{proof}

\section{Riemannian compactness theorem for orbifolds} \label{sect4}

In this section we prove a compactness result for Riemannian
orbifolds. 

The statement of the compactness result is slightly different
from the usual statement for Riemannian manifolds, which involves
a lower injectivity radius bound. The 
standard notion of 
injectivity radius is not
a useful notion for orbifolds.  For example, if $\Or$ is an
orientable $2$-orbifold with a singular point $p$ then a
geodesic from a regular point $q$ in $|\Or|$ to $p$ cannot minimize
beyond $p$. As $q$ could be arbitrarily close to $p$, we conclude
that the injectivity radius of $\Or$ would vanish. 
(We note, however, that there is a modified version of
the injectivity radius that does makes sense 
for constant-curvature cone manifolds 
\cite[Section 9.2.3]{BMP},\cite[Section 6.4]{CHK}.)

Instead, our compactness result is phrased in terms of local
volumes.  This fits well with Perelman's work on Ricci flow,
where local volume estimates arise naturally.

If one tried to prove a compactness result for Riemannian orbifolds
directly, following the proofs in the case of Riemannian manifolds,
then one would have to show that orbifold singularities do not
coalesce when taking limits. We avoid this issue by passing to
orbifold frame bundles, which are manifolds, and using equivariant
compactness results there.

Compactness theorems for Riemannian metrics and Ricci flows for
orbifolds with isolated singularities were proved in \cite{Lu}.
Compactness results for general orbifolds were stated in 
\cite[Chapter 3.3]{CCGGIIKLLN} with a short sketch of a proof.

\begin{proposition} \label{prop4.1}
Fix $K \in \Z^+ \cup \{\infty\}$.
Let $\{(\Or_i, p_i)\}_{i=1}^\infty$ be a sequence of pointed complete 
connected $C^{K+3}$-smooth Riemannian
$n$-dimensional
orbifolds.
Suppose that for each $j \in \Z^{\ge 0}$ with $j \le K$,
there is a function $A_j : (0, \infty) \rightarrow \infty$ so that
for all $i$,
$|\nabla^j \Rm| \: \le A_j(r)$ on $B(p_i, r) \subset |\Or_i|$. 
Suppose that for some $r_0 > 0$, there
is a $v_0 > 0$ so that for all $i$,
$\vol(B(p_i, r_0)) \: \ge \: v_0$. Then there is a subsequence of
$\{(\Or_i, p_i)\}_{i=1}^\infty$ that converges in the pointed 
$C^{K-1}$-topology 
to a pointed complete connected Riemannian $n$-dimensional orbifold
$(\Or_\infty, p_\infty)$.
\end{proposition}
\begin{proof}
Let $F\Or_i$ be the orthonormal frame bundle of $\Or_i$. Pick a basepoint
$\widehat{p}_i \in F\Or_i$ that projects to $p_i \in |\Or_i|$.
As in \cite[Section 6]{Fukaya}, after taking a subsequence we may assume that
the frame bundles $\{(F\Or_i, \widehat{p}_i)\}_{i=1}^\infty$ converge in
the pointed $O(n)$-equivariant Gromov-Hausdorff topology
to a $C^{K-1}$-smooth Riemannian manifold $X$ with an isometric 
$O(n)$-action and a basepoint $\widehat{p}_\infty$.
(We lose one derivative because we are working on the frame bundle.)
Furthermore, we may assume that the
convergence is realized as follows : Given 
any $O(n)$-invariant compact codimension-zero submanifold-with-boundary
$K \subset X$, for large $i$ there is an
$O(n)$-invariant compact codimension-zero submanifold-with-boundary 
$\widehat{K}_i \subset F\Or_i$
and a smooth $O(n)$-equivariant fiber bundle
$\widehat{K}_i \rightarrow K$ with nilmanifold fiber whose
diameter goes to zero as $i \rightarrow \infty$
\cite[Section 3]{Cheeger-Fukaya-Gromov}, \cite[Section 9]{Fukaya}.

Quotienting by $O(n)$, the underlying spaces 
$\{(|\Or_i|, {p}_i)\}_{i=1}^\infty$ converge in the pointed Gromov-Hausdorff
topology to $(O(n) \backslash X, p_\infty)$. 
Because of the lower volume bound 
$\vol({B}(p_i, r_0)) \: \ge \: v_0$,
a pointed Gromov-Hausdorff limit of the Alexandrov spaces 
$\{(|\Or_i|, {p}_i)\}_{i=1}^\infty$ is
an $n$-dimensional Alexandrov space 
\cite[Corollary 10.10.11]{Burago-Burago-Ivanov}.
Thus there is no collapsing and so for large $i$ the submersion
$\widehat{K}_i  \rightarrow K$ is an $O(n)$-equivariant 
$C^{K-1}$-smooth diffeomorphism. 
In particular, the $O(n)$-action on $X$ is locally free. There is a
corresponding quotient orbifold $\Or_\infty$ with $|\Or_\infty| \: = \:
O(n) \backslash X$.
As the manifolds $\{(F\Or_i, \widehat{p}_i)\}_{i=1}^\infty$
converge in a $C^{K-1}$-smooth pointed equivariant sense to 
$(X, \widehat{p}_\infty)$
we can take $O(n)$-quotients to conclude that
the orbifolds $\{(\Or_i, p_i)\}_{i=1}^\infty$ converge in the pointed
$C^{K-1}$-smooth topology to $(\Or_\infty, p_\infty)$.
\end{proof}

\begin{remark} \label{remark4.1.5}
As a consequence of Proposition \ref{prop4.1},
if there is a number $N$ so $|G_{q_i}| \le N$ for all $q_i \in |\Or|_i$
and all $i$ then $|G_{q_\infty}| \le N$ for all $q_\infty \in |\Or|_\infty$.
That is, under the hypotheses of Proposition \ref{prop4.1},
the orders of the isotropy groups cannot increase in the limit.
\end{remark}

\begin{remark} \label{remark4.2}
In the proof of Proposition \ref{prop4.1}, 
the submersions $\widehat{K}_i  \rightarrow K$ may not
be basepoint-preserving. This is where one has to leave the world of
basepoint-preserving maps.
\end{remark}

\section{Ricci flow on orbifolds} \label{sect5}

In this section we first make some preliminary remarks about
Ricci flow on orbifolds and we give the orbifold version of
Hamilton's compactness theorem. We then give the topological
classification of compact nonnegatively curved $3$-orbifolds.
Finally, we extend Perelman's no local collapsing theorem
to orbifolds.

\subsection{Function spaces on orbifolds} \label{subsect5.1}

Let $\rho \: : \: O(n) \rightarrow \R^N$ be a representation. Given
a local model $(\hU_\alpha, G_\alpha)$ and a $G_\alpha$-invariant
Riemannian metric on 
$\hU_\alpha$, let
$\widehat{V}_\alpha \: = \: \R^N \times_{O(n)} F\hU_\alpha$ be the associated
vector  bundle. If $\Or$ is a $n$-dimensional
Riemannian orbifold then there is an
associated orbivector bundle $V$ with local models
$(\widehat{V}_\alpha, G_\alpha)$. Its underlying space is
$|V| \: = \: \R^N \times_{O(n)} F\Or$. By construction, $V$ has an
inner product coming from the standard inner product on $\R^N$.
A section $s$ of $V$ is given by an $O(n)$-equivariant map
$s \: : \: F\Or \rightarrow \R^N$. In terms of local models, $s$ is described
by $G_\alpha$-invariant sections $s_\alpha$ of $\widehat{V}_\alpha$ 
that satisfy compatibility
conditions with respect to part 5 of Definition \ref{defn2.1}.

The $C^K$-norm of $s$ is defined to be the supremum of the $C^K$-norms of the
$s_\alpha$'s. Similarly, the square of the $H^K$-norm of $s$ is
defined to be
the integral over $|\Or|_{reg}$ of the local square $H^K$-norm, the latter
being defined using local models.
(Note that $|\Or|_{reg}$ has full Hausdorff $n$-measure in $|\Or|$.)
Then $H^{-K}$ can be defined by
duality. One has the rough Laplacian mapping
$H^K$-sections of $V$ to $H^{K-2}$-sections of $V$.

One can define differential operators and pseudodifferential operators
acting on $H^K$-sections of $V$.
Standard elliptic and parabolic regularity theory extends to
the orbifold setting, as can be seen by working equivariantly in
local models.

\subsection{Short-time existence for Ricci flow on orbifolds}

Suppose that $\{g(t)\}_{t \in [A, B]}$ is a smooth 
$1$-parameter family of Riemannian metrics on $\Or$. We will call $g$
a {\em flow of metrics} on $\Or$. 
The Ricci flow equation
$\frac{\partial g}{\partial t} \: = \: - \: 2 \Ric$ makes
sense in terms of local models.
Using the deTurck trick \cite{deTurck}, 
which is based on local differential analysis,
one can reduce the short-time existence problem for the Ricci flow to the
short-time existence problem for a parabolic PDE. Then any short-time
existence proof for parabolic PDEs on compact manifolds,
such as that of \cite[Proposition 15.8.2]{Taylor (1997)}, 
will extend from the manifold setting
to the orbifold setting. 

\begin{remark} \label{remark5.1}
Even in the manifold case, one needs a slight additional argument to reduce the
short-time existence of the
Ricci-de Turck equation to that of a standard quasilinear parabolic PDE.
In local coordinates the Ricci-de Turck equation takes the form
\begin{equation} \label{5.2}
\frac{\partial g_{ij}}{\partial t} = \sum_{kl} g^{kl} \partial_k \partial_l
g_{ij} + \ldots
\end{equation}
There is a slight issue since 
 (\ref{5.2}) is
not uniformly parabolic,
in that $g^{kl}$ could degenerate
with respect to, say, the initial metric $g_0$.
This issue does not seem to have been addressed in the literature.
However, it is easily circumvented.  
Let ${\mathcal M}$ be the space of smooth Riemannian metrics on a
compact manifold $M$.
Let $F : {\mathcal M} \rightarrow {\mathcal M}$ be a smooth map so that
for some $\epsilon > 0$,
we have $F(g) = g$ if $\parallel g - g_0 \parallel_{g_0} < \epsilon$,
and in addition $\epsilon g_0 \le F(g) \le \epsilon^{-1} g_0$ for all $g$.
(Such a map $F$ is easily constructed using the fact that the inner
products on $T_pM$, relative to $g_0(p)$, can be identified with
$\GL(n,\R)/O(n)$, along with the fact that $\GL(n,\R)/O(n)$ deformation
retracts onto a small ball around its basepoint.)
By \cite[Proposition 15.8.2]{Taylor (1997)}, there is a short-time
solution to
\begin{equation}
\frac{\partial g_{ij}}{\partial t} = \sum_{kl} F(g)^{kl} \partial_k \partial_l
g_{ij} + \ldots
\end{equation}
with $g(0) = g_0$. Given this solution, there is some $\delta > 0$ so 
that $\parallel g(t) - g_0 \parallel_{g_0} < \epsilon$ whenever 
$t \in [0, \delta]$. Then $\{g(t)\}_{t \in [0,\delta]}$ also solves
the Ricci-de Turck equation (\ref{5.2}).
\end{remark}

We remark that any Ricci flow results based on the maximum principle
will have evident extensions from manifolds to orbifolds. Such
results include
\begin{itemize}
\item The lower bound on scalar curvature
\item The Hamilton-Ivey pinching results for three-dimensional scalar
curvature
\item Hamilton's differential Harnack inequality for Ricci flow
solutions with nonnegative curvature operator
\item Perelman's differential Harnack inequality.
\end{itemize}

\subsection{Ricci flow compactness theorem for orbifolds} \label{subsect5.2}

Let $\Or_1$ and $\Or_2$ be two connected pointed $n$-dimensional
orbifolds, with flows of metrics $g_1$ and $g_2$.
If $f \: : \: \Or_1 \rightarrow \Or_2$ is a (time-independent) 
diffeomorphism then we can construct the pullback flow $f^* g_2$
and define
the $C^K$-distance between $g_1$ and $f^* g_2$, using local models for
$\Or_1$.

\begin{definition} \label{closenessdef}
Let $\Or_1$ and $\Or_2$ be
connected pointed $n$-dimensional orbifolds.  Given
numbers $A, B$ with $-\infty \: \le \: A \: < \: 0 \: \le \: B \: \le \: 
\infty$, suppose that $g_i$ is a flow
of metrics on $\Or_i$ that exists for the time interval $[A, B]$.
Suppose that $g_i(t)$ is complete for each $t$.
Given $\epsilon > 0$, suppose
that $f \: : \: 
\check{B}(p_1, \epsilon^{-1}) \rightarrow \Or_2$ is a smooth
map from the time-zero ball that is a diffeomorphism onto its image. Let
$|f| \: : \: B(p_1, \epsilon^{-1}) 
\rightarrow |\Or_2|$ be the underlying map. We say
that the $C^K$-distance between the flows
$(\Or_1, p_1, g_1)$ and $(\Or_2, p_2, g_2)$ is bounded
above by $\epsilon$ if \\
1. The $C^K$-distance between $g_1$ and
$f^* g_2$ on $([A, B] \cap (- \epsilon^{-1}, \epsilon^{-1})) \times
\check{B}(p_1, \epsilon^{-1})$ is at most $\epsilon$ and \\
2. The time-zero distance $d_{|\Or_2|}(|f|({p}_1), {p}_2)$ is at most
$\epsilon$.

Taking the infimum of all such possible $\epsilon$'s defines 
the $C^K$-distance between the flows $(\Or_1, p_1, g_1)$ and 
$(\Or_2, p_2, g_2)$.
\end{definition}

Note that time derivatives appear in the definition of the
$C^K$-distance between $g_1$ and $f^* g_2$.

\begin{proposition} \label{prop5.5}
Let $\{g_i \}_{i=1}^\infty$ be a sequence of Ricci flow solutions
on pointed connected $n$-dimensional orbifolds
$\{(\Or_i, p_i)\}_{i=1}^\infty$, defined for 
$t \in (A, B)$
and complete for each $t$, with 
$-\infty \: \le \: A \: < \: 0 \: \le \: B \: \le \: \infty$.
Suppose that the following two conditions are satisfied :\\
1. For every compact interval $I \subset (A,B)$, there is some
$K_I < \infty$ so that for all $i$, we have $\sup_{|\Or_i| \times I}
|\Rm_{g_i}(p,t))| \: \le \: K_I$, and \\
2. For some $r_0, v_0 > 0$ and all $i$, the time-zero volume
$\vol(B(p_i, r_0))$
is bounded below by $v_0$.

Then a subsequence of the solutions converges
in the sense of Definition \ref{closenessdef}
to a Ricci flow solution $g_\infty(t)$ on a pointed
connected $n$-dimensional orbifold $(\Or_\infty, p_\infty)$, defined for all 
$t \in (A, B)$.
\end{proposition}
\begin{proof}
Using Proposition \ref{prop4.1}, the proof is essentially
the same as that in
\cite[p. 548-551]{Hamiltonnn} and
\cite[p. 1116-1117]{Lu}. 
\end{proof}

\begin{remark}
There are variants of Proposition \ref{prop5.5} that hold, for example,
if one just assumes a uniform curvature bound on $r$-balls, for each $r>0$. 
These variants are orbifold versions of the results in 
\cite[Appendix E]{Kleiner-Lott}, to which we refer for details.
The proofs of these orbifold extensions use, among other things,
the orbifold version of the Shi estimates; the proof of the latter
goes through to the orbifold setting with no real change.
\end{remark}

\subsection{Compact nonnegatively curved $3$-orbifolds} \label{subsect5.3}

\begin{proposition} \label{prop5.6}
Any compact nonnegatively curved $3$-orbifold $\Or$ is diffeomorphic to one of
\begin{enumerate}
\item $S^3//\Gamma$ for some finite group $\Gamma \subset \Isom^+(S^3)$.
\item $T^3//\Gamma$ for some finite group $\Gamma \subset \Isom^+(T^3)$.
\item $S^1 \times (S^2//\Gamma)$ or $S^1 \times_{\Z_2} (S^2//\Gamma)$
for some finite group 
$\Gamma \subset \Isom(S^2)$.
\item $S^1 \times (\Sigma^2//\Gamma)$ or
$S^1 \times_{\Z_2} (\Sigma^2//\Gamma)$ for some finite group 
$\Gamma \subset \Isom(\Sigma^2)$, where $\Sigma^2$ is a
simply-connected bad $2$-orbifold equipped with its
unique 
(up to diffeomorphism)
Ricci soliton metric \cite[Theorem 4.1]{Wu}.
\end{enumerate}
\end{proposition}
\begin{proof}
Let $k$ be the largest number
so that the universal cover $\widetilde{\Or}$ isometrically splits off an 
$\R^k$-factor. Write $\widetilde{\Or} = \R^k \times \Or^\prime$.

If $\Or^\prime$ is noncompact then by the Cheeger-Gromoll
argument \cite[Pf. of Theorem 3]{Cheeger-Gromoll2},  
$|\Or^\prime|$ contains a line. Proposition \ref{prop3.1} implies that
$\Or^\prime$ splits off an $\R$-factor, which is a contradiction.
Thus $\Or^\prime$ is simply-connected and compact with
nonnegative sectional curvature.

If $k = 3$ then $\widetilde{\Or} = \R^3$ and $\Or$ is a quotient
of $T^3$.

If $k = 2$ then there is a contradiction, 
as there is no simply-connected compact
$1$-orbifold.

If $k = 1$ then $\Or^\prime$ is diffeomorphic to $S^2$ or $\Sigma^2$.
The Ricci flow on $\widetilde{\Or} = \R \times \Or^\prime$ splits
isometrically.  After rescaling, the Ricci flow on $\Or^\prime$
converges to a constant curvature metric on $S^2$ or to
the unique Ricci
soliton metric on $\Sigma^2$
\cite{Wu}. Hence $\pi_1(\Or)$ is a subgroup
of $\Isom(\R \times S^2)$ or $\Isom(\R \times \Sigma^2)$,
where the isometry groups are in terms of standard metrics. As
$\pi_1(\Or)$ acts properly discontinuously and cocompactly on 
$\widetilde{\Or}$,
there is a short exact sequence
\begin{equation}
1 \longrightarrow \Gamma_1 \longrightarrow \pi_1(\Or) \longrightarrow
\Gamma_2 \longrightarrow 1,
\end{equation}
where $\Gamma_1 \subset \Isom(\Or^\prime)$ and
$\Gamma_2$ is an infinite cyclic group or an infinite dihedral group.
It follows that
$\Or$ is finitely covered by $S^1 \times S^2$ or $S^1 \times \Sigma^2$.

Suppose that $k = 0$.
If $\Or$ is positively curved then any proof of Hamilton's
theorem about $3$-manifolds with positive Ricci curvature
\cite{Hamiltonnnnn}
extends to the orbifold case, to show that
$\Or$ admits a metric of constant positive curvature;
c.f. \cite{Hamilton (2003)}.
Hence we can reduce to the case when
$\Or$ does not have positive curvature and the Ricci flow
does not immediately give it positive curvature. From the
strong maximum principle as in
\cite[Section 8]{Hamiltonnnn}, for any
$p \in |\Or|_{reg}$ there is a nontrivial orthogonal splitting $T_p \Or = E_1
\oplus E_2$ which is invariant under holonomy around loops based
at $p$. The same will be true on $\widetilde{\Or}$. 
Lemma \ref{lemma2.12} implies that $\widetilde{\Or}$ splits off an 
$\R$-factor, which is a contradiction.
\end{proof}

\subsection{${\mathcal L}$-geodesics and noncollapsing} \label{subsect5.4}

Let $\Or$ be an
$n$-dimensional
orbifold and let $\{g(t)\}_{t \in [0, T)}$ be
a Ricci flow solution on $\Or$ so that 
\begin{itemize}
\item The time slices $(\Or, g(t))$ are complete. 
\item There is bounded curvature on compact subintervals of
$[0, T)$.
\end{itemize}

Given $t_0 \in [0,T)$ and $p \in |\Or|$, put $\tau = t_0 - t$. Let
$\gamma : [0, \overline{\tau}] \rightarrow \Or$ be a piecewise smooth
curve with $|\gamma|(0) = p$ and $\overline{\tau} \le t_0$. Put
\begin{equation}
{\mathcal L}(\gamma) =
\int_0^{\overline{\tau}} \sqrt{\tau} \left(
R(\gamma(\tau)) +  |\dot{\gamma}(\tau)|^2 \right)
\: d\tau,
\end{equation}
where the scalar curvature $R$ and the norm $|\dot{\gamma}(\tau)|$
are evaluated using the metric at time $t_0 - \tau$.
With $X = \frac{d\gamma}{d\tau}$, the 
{\em ${\mathcal L}$-geodesic equation} is
\begin{equation}
\nabla_X X - \frac12 \nabla R + \frac{1}{2\tau} X +
2 \Ric(X, \cdot) = 0.
\end{equation}
Given an ${\mathcal L}$-geodesic $\gamma$, 
its {\em initial velocity} is defined to be
$v = \lim_{\tau \rightarrow 0} \sqrt{\tau} \frac{d\gamma}{d\tau}
\in C_p |\Or|$.

Given $q \in |\Or|$, put
\begin{equation}
L(q, \overline{\tau}) = \inf \{ {\mathcal L}(\gamma) \: : \:
|\gamma|(\overline{\tau}) = q \},
\end{equation}
where the infimum runs over piecewise smooth curves $\gamma$ with
$|\gamma|(0) = p$ and $|\gamma|(\overline{\tau}) = q$.
Then any piecewise smooth curve $\gamma$ which is a minimizer for $L$ is a
smooth ${\mathcal L}$-geodesic.

\begin{lemma} \label{lemma5.11}
There is a minimizer $\gamma$ for $L$.
\end{lemma}
\begin{proof}
The proof is similar to that in
\cite[p. 2631]{Kleiner-Lott}.
We outline the steps.
Given $p$ and $q$, one considers 
piecewise smooth curves $\gamma$ as above. 
Fixing $\epsilon > 0$, one shows that the curves $\gamma$ with 
${\mathcal L}(\gamma) < L(q, \overline{\tau}) + \epsilon$
are uniformly continuous. In particular, there is an $R < \infty$
so that any such $\gamma$ lies in $B(p,R)$. Next, one shows that
there is some $\rho \in (0,R)$ so that for any $x \in B(p, R)$,
there is a local model $(\widehat{U}, G_x)$
with $\widehat{U}/G_x = B(x, \rho)$ such that
for any $p^\prime, q^\prime \in B(x, \rho)$ and any subinterval
$[\overline{\tau}_1, \overline{\tau}_2]  \subset [0, \overline{\tau}]$,
\begin{itemize}
\item There is a unique minimizer 
for the functional
$\int_{\overline{\tau}_1}^{\overline{\tau}_2} \sqrt{\tau} \left(
R(\gamma(\tau)) +  |\dot{\gamma}(\tau)|^2 \right)
\: d\tau$ among piecewise smooth curves 
${\gamma} : 
[\overline{\tau}_1, \overline{\tau}_2] \rightarrow \Or$ with
$|\gamma|(\overline{\tau}_1) = p^\prime$ and
$|\gamma|(\overline{\tau}_2) = q^\prime$.
\item The minimizing ${\gamma}$ is smooth and 
the image of $|\gamma|$
lies in $B(x,\rho)$.
\end{itemize}
This is shown by working in the local models.  Now cover
$\overline{B(p, R)}$ by a finite number of $\rho$-balls
$\{ B(x_i, \rho) \}_{i=1}^N$. Using the uniform continuity, 
let $A \in \Z^+$ be such that for any 
$\gamma : [0, \overline{\tau}] \rightarrow \Or$ with 
$|\gamma|(0) = p$, 
$|\gamma|(\overline{\tau}) = q$
and ${\mathcal L}(\gamma) < L(q, \overline{\tau}) + \epsilon$,
and any $[\overline{\tau}_1, \overline{\tau}_2]  \subset [0, \overline{\tau}]$
of length at most
$\frac{\overline{\tau}}{A}$, the distance between
$|\gamma|(\overline{\tau}_1)$ and $|\gamma|(\overline{\tau}_2)$
is less than the Lebesgue number of
the covering.  We can effectively reduce the problem of
finding a minimizer for $L$ to the problem of minimizing a
continuous function defined on tuples
$(p_0, \ldots, p_A) \in \overline{B(p, R)}^{A+1}$ with
$p_0 = p$ and $p_A = q$. This shows that the minimizer exists.
\end{proof}

Define the {\em ${\mathcal L}$-exponential map} 
$ \: : \: T_p \Or \rightarrow \Or$ by
saying that for $v \in C_p|\Or|$, we put
${\mathcal L}\exp_{\overline{\tau}}(v) = 
|\gamma|(\overline{\tau})$,
where $\gamma$ is the unique
${\mathcal L}$-geodesic from $p$ whose initial velocity is 
$v$. Then ${\mathcal L}\exp_{\overline{\tau}}$ is a smooth
orbifold map.

Let ${\mathcal B}_{\overline{\tau}} \subset |\Or|$ 
be the set of points $q$ which are
either endpoints of more than one minimizing ${\mathcal L}$-geodesic
$\gamma : [0, \overline{\tau}] \rightarrow \Or$, or 
are the endpoint of a minimizing geodesic
$\gamma_v : [0, \overline{\tau}] \rightarrow \Or$ where $v \in C_p|\Or|$
is a critical point of ${\mathcal L}\exp_{\overline{\tau}}$.
We call ${\mathcal B}_{\overline{\tau}}$ the {\em time-$\overline{\tau}$ 
${\mathcal L}$-cut locus} of $p$. It is a closed subset of $|\Or|$.
Let ${\mathcal G}_{\overline{\tau}} \subset |\Or|$ be the complement of 
${\mathcal B}_{\overline{\tau}}$
and let $\Omega_{\overline{\tau}} \subset C_p|\Or|$ be the
corresponding set of initial conditions for minimizing
${\mathcal L}$-geodesics. Then $\Omega_{\overline{\tau}}$ is an
open set, and the restriction of
${\mathcal L}\exp_{\overline{\tau}}$ to 
$T_p \Or \Big|_{\Omega_{\overline{\tau}}}$ is
an orbifold diffeomorphism to $\Or \Big|_{{\mathcal G}_{\overline{\tau}}}$.

\begin{lemma} \label{lemma5.12}
${\mathcal B}_{\overline{\tau}}$ has measure zero in $|\Or|$.
\end{lemma}
\begin{proof}
The proof is similar to that in \cite[p. 2632]{Kleiner-Lott}.
By Sard's theorem, it suffices to show that the
subset ${\mathcal B}^\prime_{\overline{\tau}} \subset 
{\mathcal B}_{\overline{\tau}}$, consisting of regular values of
${\mathcal L}\exp_{\overline{\tau}}$, has measure zero in $|\Or|$.
One shows that ${\mathcal B}^\prime_{\overline{\tau}}$ is contained
in the underlying spaces of a countable union of codimension-$1$ 
suborbifolds of $\Or$, which implies the lemma.
\end{proof}

Therefore one may compute the integral of any integrable function on
$|\Or|$ by pulling it back to $\Omega_{\overline{\tau}} \subset
C_p |\Or|$ and using the change of variable formula.

For $q \in |\Or|$, put $l(q, \overline{\tau}) = \frac{L(q, \overline{\tau})
}{
2\sqrt{\overline{\tau}}
}$.
Define the {\em reduced volume} by
\begin{equation}
\widetilde{V}(\overline{\tau}) = \overline{\tau}^{- \: \frac{n}{2}}
\int_{|\Or| }
e^{-l(q, \overline{\tau})} \: \dvol(q).
\end{equation}

\begin{lemma} \label{lemma5.14}
The reduced volume is monotonically nonincreasing in
$\overline{\tau}$.
\end{lemma}
\begin{proof}
The proof is similar to that in
\cite[Section 23]{Kleiner-Lott}. In the proof, one pulls back the integrand to
$C_p|\Or|$. 
\end{proof}

\begin{lemma} \label{lemma5.15}
For each $\overline{\tau} > 0$, there is some $q \in |\Or|$ so that
$l(q, \overline{\tau}) \le \frac{n}{2}$.
\end{lemma}
\begin{proof}
The proof is similar to that in 
\cite[Section 24]{Kleiner-Lott}.
It uses the maximum principle, which is valid for orbifolds.
\end{proof}

\begin{definition}
Given $\kappa, \rho > 0$,
a Ricci flow solution $g(\cdot)$ defined on a time interval
$[0, T)$ is {\em $\kappa$-noncollapsed on the scale $\rho$} if
for each $r < \rho$ and all $(x_0, t_0) \in |\Or| \times [0, T)$ 
with $t_0 \ge r^2$, whenever it is true that
$|\Rm(x,t)| \le r^{-2}$ for every $x \in B_{t_0}(x_0, r)$ and
$t \in [t_0 - r^2, t_0]$, then we also have
$\vol(B_{t_0}(x_0, r)) \ge \kappa r^n$.
\end{definition}

\begin{lemma} \label{lemma5.16.5}
If a Ricci flow solution is $\kappa$-noncollapsed on some scale
then there is a uniform upper bound $|G_p| \le N(n,\kappa)$ on the orders
of the isotropy groups at points $p \in |\Or|$.
\end{lemma}
\begin{proof}
Given $p \in |\Or|$, let $B_{t_0}(p,r)$ be a ball such that
$|\Rm(x,t_0)| \le r^{-2}$ for all $x \in B_{t_0}(p,r)$.
By assumption $r^{-n} \: \vol(B_{t_0}(x_0, r)) \ge \kappa$. 
Let $c_n$ denote the area of the unit $(n-1)$-sphere in $\R^n$.
Applying the Bishop-Gromov inequality to $B_{t_0}(p,r)$ gives
\begin{equation}
\frac{1}{|G_p|} \: \ge \:
\frac{r^{-n} \: \vol(B_{t_0}(x_0, r))}{c_n \int_0^1 \sinh^{n-1}(s) \: ds} \: 
\ge \:
\frac{\kappa}{c_n \int_0^1 \sinh^{n-1}(s) \: ds}.
\end{equation}
The lemma follows.
\end{proof}

\begin{proposition} \label{prop5.17}
Given numbers $n \in \Z^+$, $T < \infty$ and $\rho, K, c > 0$,
there is a number $\kappa = \kappa(n,K,c,\rho,T) > 0$ with the
following property.  Let $(\Or^n, g(\cdot))$ be a Ricci flow
solution defined on the time interval $[0,T)$,
with complete time slices,
such that the curvature $|\Rm|$ is bounded on every compact
subinterval $[0, T^\prime] \subset [0, T)$. Suppose that
$(\Or, g(0))$ has $|\Rm| \le K$ and $\vol(B(p, 1)) \ge c > 0$ for
every $p \in |\Or|$. Then the Ricci flow solution is 
$\kappa$-noncollapsed on the scale $\rho$.
\end{proposition}
\begin{proof}
The proof is similar to that in \cite[Section 26]{Kleiner-Lott}.
As in the proof there, we use the fact that
the initial conditions give uniformly bounded geometry in
a small time interval $[0, \overline{t}/2]$, as follows from
Proposition \ref{prop5.5} and derivative estimates.
\end{proof}

\begin{proposition} \label{prop5.18}
For any $A \in (0, \infty)$, there is some $\kappa = \kappa(A) > 0$
with the following property.  Let $(\Or, g(\cdot))$ be 
an $n$-dimensional Ricci flow solution defined for $t \in [0, r_0^2]$
having complete time slices and uniformly bounded curvature.  Suppose
that $\vol(B_0(p_0, r_0)) \ge A^{-1} r_0^n$ and that
$|\Rm|(q,t)| \le \frac{1}{nr_0^2}$ for all $(q,t) \in B_0(p_0,r_0)
\times [0,r_0^2]$. Then the solution cannot be $\kappa$-collapsed
on a scale less than $r_0$ at any point $(q, r_0^2)$ with
$q \in B_{r_0^2}(p_0, Ar_0)$. 
\end{proposition}
\begin{proof}
The proof is similar to that in \cite[Section 28]{Kleiner-Lott}.
\end{proof}

\section{$\kappa$-solutions} \label{sect6}

In this section we extend results about $\kappa$-solutions
from manifolds to orbifolds.

\begin{definition}
Given $\kappa > 0$, a {\em $\kappa$-solution} is a Ricci flow solution
$(\Or, g(t))$ that is defined on a time interval of the form
$(-\infty, C)$ (or $(-\infty, C]$) such that :
\begin{enumerate}
\item The curvature $|\Rm|$ is bounded on each compact time interval
$[t_1,t_2] \subset (-\infty, C)$ (or $(-\infty, C]$), and each
time slice $(\Or, g(t))$ is complete.
\item The curvature operator is nonnegative and the scalar curvature
is everywhere positive.
\item The Ricci flow is $\kappa$-noncollapsed at all scales.
\end{enumerate}
\end{definition}

Lemma \ref{lemma5.16.5} gives an upper bound on the orders of the
isotropy groups. In the rest of this section we will use this
upper bound without explicitly restating it.

\subsection{Asymptotic solitons} \label{subsect6.1}

Let $(p, t_0)$ be a point in a $\kappa$-solution
$(\Or, g(\cdot))$ so that $G_p$ has maximal order. Define
the reduced volume $\widetilde{V}(\overline{\tau})$ and
the reduced length $l(q, \overline{\tau})$ as in Subsection 
\ref{subsect5.4},
by means of curves starting from $(p, t_0)$, with $\tau = t_0 - t$.
From Lemma \ref{lemma5.15}, for each $\overline{\tau} > 0$ there is some
$q(\overline{\tau}) \in |\Or|$ such that $l(q(\overline{\tau}),
\overline{\tau}) \le \frac{n}{2}$. (Note that $l \ge 0$ from the
curvature assumption.)

\begin{proposition} \label{prop6.2}
There is a sequence $\overline{\tau}_i \rightarrow \infty$ so that
if we consider the solution $g(\cdot)$ on the time interval
$[t_0 - \overline{\tau}_i, t_0 - \frac12 \overline{\tau}_i]$ and
parabolically rescale it at the point $(q(\overline{\tau}_i),
t_0 - \overline{\tau}_i)$ by the factor $\overline{\tau}_i^{-1}$
then as $i \rightarrow \infty$, the rescaled solutions converge
to a nonflat gradient shrinking soliton (restricted to $[-1, - \frac12]$).
\end{proposition}
\begin{proof}
The proof is similar to that in \cite[Section 39]{Kleiner-Lott}.
Using estimates on the reduced length as defined with the
basepoint $(p, t_0)$,
one constructs a limit Ricci flow solution 
$(\Or_\infty, g_\infty(\cdot))$ defined for
$t \in [-1, - \frac12]$, which is a gradient shrinking soliton. The
only new issue is to show that it is nonflat.

As in \cite[Section 39]{Kleiner-Lott},
there is a limiting reduced length function 
$l_\infty(\cdot, \tau) \in C^\infty(\Or_\infty)$,
and a reduced volume which is a constant
$c$, strictly less than the $t \rightarrow t_0$ limit of the reduced
volume of $(\Or, g(\cdot))$. The latter equals
$\frac{(4\pi)^{\frac{n}{2}}}{|G_p|}$. If the limit solution were
flat then $l_\infty(\cdot, \tau)$ would have a constant 
positive-definite Hessian.
It would then have a unique critical point $q$. Using the gradient flow
of $l_\infty(\cdot, \tau)$, one deduces that $\Or_\infty$ is
diffeomorphic to $T_q \Or_\infty$.
As in \cite[Section 39]{Kleiner-Lott}, one concludes
that
\begin{equation}
c \: = \: 
\int_{C_q |\Or_\infty| \cong \R^n/G_q} 
\tau^{-\frac{n}{2}} e^{- \: \frac{|x|^2}{4\tau}} \dvol \: = \:
\frac{(4\pi)^{\frac{n}{2}}}{|G_q|}.
\end{equation}
As $|G_q| \le |G_p|$, we obtain a contradiction.
\end{proof}

\subsection{Two-dimensional $\kappa$-solutions} \label{subsect6.2}

\begin{lemma} \label{lemma6.4}
Any two-dimensional $\kappa$-solution $(\Or, g(\cdot))$
is an isometric quotient of the round shrinking $2$-sphere or
is a Ricci soliton metric on a bad $2$-orbifold.
\end{lemma}
\begin{proof}
The proof is similar to that in \cite[Theorem 4.1]{Ye2}.
One considers the asymptotic soliton and shows that
it has strictly positive scalar curvature outside of a
compact region (as in \cite[Lemma 1.2]{Perelman2}).
Using standard Jacobi field estimates, the asymptotic soliton
must be compact.
The lemma then follows from convergence results for $2$-dimensional 
compact Ricci flow
(using \cite{Wu} in the case of bad $2$-orbifolds).
\end{proof}
\begin{remark} \label{remark6.5}
One can alternatively prove Lemma \ref{lemma6.4} 
using the fact that if $(\Or, g(\cdot))$ is a
$\kappa$-solution then so is the pullback solution
$(\widetilde{\Or}, \widetilde{g}(\cdot))$ on the universal cover.
If $\Or$ is a bad $2$-orbifold then $\Or$ is compact and the
result follows from \cite{Wu}. If $\Or$ is a good $2$-orbifold
then $(\widetilde{\Or}, \widetilde{g}(\cdot))$ is a round
shrinking $S^2$ from \cite[Section 40]{Kleiner-Lott}.
\end{remark}

\subsection{Asymptotic scalar curvature and asymptotic volume ratio}
\label{subsect6.3}

\begin{definition}
If $\Or$ is a complete connected Riemannian orbifold then its
{\em asymptotic scalar curvature ratio} is
${\mathcal R} = \limsup_{q \rightarrow \infty} R(q) d(x,p)^2$. It
is independent of the basepoint $p \in |\Or|$.
\end{definition}

\begin{lemma} \label{lemma6.7}
Let $(\Or, g(\cdot))$ be a noncompact $\kappa$-solution.  Then the
asymptotic scalar curvature ratio is infinite for each time slice.
\end{lemma}
\begin{proof}
The proof is similar to that in \cite[Section 41]{Kleiner-Lott}.
Choose a time $t_0$.
If ${\mathcal R} \in (0, \infty)$ then after rescaling
$(\Or, g(t_0))$, one
obtains convergence to a smooth annular region in the Tits cone
$C_T \Or$ at time $t_0$.
(Here $C_T \Or$ denotes a smooth orbifold structure on the complement
of the vertex in the Tits cone $C_T |\Or|$.)
Working on the regular part of the
annular region, one obtains a contradiction from the curvature
evolution equation.

If ${\mathcal R} = 0$ then the rescaling limit is a smooth flat
metric on $C_T \Or$, away from the vertex.  The unit sphere $S_\infty$
in $C_T \Or$ has principal curvatures one.  It can be approximated
by a sequence of codimension-one compact suborbifolds $S_k$ in $\Or$ with
rescaled principal curvatures approaching one, which bound
compact suborbifolds $\Or_k \subset \Or$.

Suppose first that $n \ge 3$. By Lemma \ref{lemma3.12}, for large $k$ there is
some $p_k \in |\Or|$ so that the suborbifold $S_k$ is diffeomorphic to the 
unit sphere in $T_{p_k} \Or$. As $S_k$ is diffeomorphic to $S_\infty$
for large $k$, we conclude that $S_\infty$ is isometric to $S^{n-1}//\Gamma$
for some finite group $\Gamma \subset \Isom^+(S^{n-1})$. Let
$p \in |\Or|$ be a point with $G_p \cong \Gamma$. As $C_T |\Or|$
is isometric to $\R^n/\Gamma$, 
$\lim_{r \rightarrow \infty} r^{-n} \: \vol(B(p,r))$ exists and equals
the $\frac{1}{|\Gamma|}$ times the volume of the unit ball in $\R^n$. 
On the other hand, this equals $\lim_{r \rightarrow 0} r^{-n} \: \vol(B(p,r))$.
As we have equality in the Bishop-Gromov inequality, we conclude
that $\Or$ is flat, which is a contradiction.

If $n=2$ then we can adapt the argument in \cite[Section 41]{Kleiner-Lott}
to the orbifold setting.
\end{proof}

\begin{definition}
If $\Or$ is a complete $n$-dimensional Riemannian orbifold with
nonnegative Ricci curvature then its {\em asymptotic volume ratio}
is ${\mathcal V} = \lim_{r \rightarrow \infty} r^{-n} \vol(B(p, r))$.
It is independent of the choice of basepoint $p \in |\Or|$.
\end{definition}

\begin{lemma} \label{lemma6.9}
Let $(\Or, g(\cdot))$ be a noncompact $\kappa$-solution. Then the
asymptotic volume ratio ${\mathcal V}$ vanishes for each time slice
$(\Or, g(t_0))$. Moreover, there is a sequence of points
$p_k \in |\Or|$ going to infinity such that the pointed sequence
$\{ (\Or, (p_k, t_0), g(\cdot) \}_{k=1}^\infty$ converges,
modulo rescaling by $R(p_k, t_0)$, to a $\kappa$-solution which
isometrically splits off an $\R$-factor.
\end{lemma}
\begin{proof}
The proof is similar to that in \cite[Section 41]{Kleiner-Lott}
\end{proof}

\subsection{In a $\kappa$-solution, the curvature and the normalized
volume control each other} \label{subsect6.4}

\begin{lemma} \label{lemma6.10}
Given $n \in \Z^+$, we consider $n$-dimensional 
$\kappa$-solutions.
\begin{enumerate}
\item If $B(p_0, r_0)$ is a ball in a time slice of a $\kappa$-solution
then the normalized volume $r^{-n} \vol(B(p_0, r_0))$ is controlled
(i.e. bounded away from zero) $\Longleftrightarrow$ the normalized
scalar curvature $r_0^2 R(p_0)$ is controlled (i.e. bounded above)
\item (Precompactness) If $\{(\Or_k, (p_k, t_k), g_k(\cdot)) \}_{k=1}^\infty$
is a sequence of pointed $\kappa$-solutions and for some $r > 0$, the
$r$-balls $B(p_k, r) \subset (\Or_k, g_k(t_k))$ have controlled
normalized volume, then a subsequence converges to an ancient solution
$(\Or_\infty, (p_\infty, 0), g_\infty(\cdot))$ which has nonnegative
curvature operator, and is $\kappa$-noncollapsed (though {\it a priori}
the curvature may be unbounded on a given time slice).
\item There is a constant $\eta = \eta(n,\kappa)$ so that for all
$p \in |\Or|$, we have $|\nabla R|(p,t) \le \eta R^{\frac32}(p,t)$
and $|R_t|(p,t) \le \eta R^2(p,t)$. More generally, there are scale
invariant bounds on all derivatives of the curvature tensor, that
only depend on $n$ and $\kappa$.
\item There is a function $\alpha : [0, \infty) \rightarrow [0, \infty)$
depending only on $n$ and $\kappa$ such that
$\lim_{s \rightarrow \infty} \alpha(s) = \infty$, and for every
$p, p^\prime \in |\Or|$, we have
$R(p^\prime) d^2(p,p^\prime) \le \alpha \left( 
R(p) d^2(p,p^\prime) \right)$.
\end{enumerate}
\end{lemma}
\begin{proof}
The proof is similar to that in \cite[Section 42]{Kleiner-Lott}.
In the proof by contradiction of the implication $\Longleftarrow$ of part (1),
after passing to a subsequence we can assume that $|G_{p_k}|$
is a constant $C$. Then we use the argument in 
\cite[Section 42]{Kleiner-Lott} with $c_n$ equal to $\frac{1}{C}$
times the volume of the unit Euclidean $n$-ball.
\end{proof}

\subsection{A volume bound} \label{subsect6.5}

\begin{lemma} \label{lemma6.11}
For every $\epsilon > 0$, there is an $A < \infty$ with the following
property.  Suppose that we have a sequence of (not necessarily complete)
Ricci flow solutions $g_k(\cdot)$ with nonnegative curvature operator,
defined on $\Or_k \times [t_k, 0]$, such that:
\begin{itemize}
\item For each $k$, the time-zero ball $B(p_k, r_k)$ has compact closure in
$|\Or_k|$.
\item For all $(p,t) \in B(p_k, r_k) \times [t_k,0]$, we have 
$\frac12 R(p,t) \le R(p_k,0) = Q_k$. 
\item $\lim_{k \rightarrow \infty} t_k Q_k = - \infty$. 
\item $\lim_{k \rightarrow \infty} r_k^2 Q_k = \infty$. 
\end{itemize}
Then for large $k$, we have $\vol(B(p_k, AQ_k^{-\frac12})) \le
\epsilon (AQ_k^{- \frac12})^n$ at time zero.
\end{lemma}
\begin{proof}
The proof is similar to that in \cite[Section 44]{Kleiner-Lott}.
\end{proof}

\subsection{Curvature bounds for Ricci flow solutions with nonnegative
curvature operator, assuming a lower volume bound} \label{subsect6.6}
 
\begin{lemma} \label{lemma6.12}
For every $w > 0$, there are $B = B(w) < \infty$, $C = C(w) < \infty$
and $\tau_0 = \tau_0(w) > 0$ with the following properties.

(a) Take $t_0 \in [-r_0^2,0)$. Suppose that we have a (not necessarily
complete) Ricci flow solution $(\Or, g(\cdot))$, defined for
$t \in [t_0, 0]$, so that at time zero the metric ball $B(p_0, r_0)$
has compact closure.  Suppose that for each $t \in [t_0, 0]$,
$g(t)$ has nonnegative curvature operator and $\vol(B_t(p_0, r_0))
\ge w r_0^n$. Then
\begin{equation}
R(p,t) \le C r_0^{-2} + B (t - t_0)^{-1}
\end{equation}
whenever $\dist_t(p, p_0) \le \frac14 r_0$.

(b) Suppose that we have a (not necessarily
complete) Ricci flow solution $(\Or, g(\cdot))$, defined for
$t \in [-\tau_0 r_0^2, 0]$, so that at time zero the metric ball $B(p_0, r_0)$
has compact closure.  Suppose that for each $t \in [-\tau_0 r_0^2, 0]$,
$g(t)$ has nonnegative curvature operator. If we assume a time-zero volume
bound $\vol(B_0(p_0, r_0)) \ge w r_0^n$ then
\begin{equation}
R(p,t) \le C r_0^{-2} + B (t + \tau_0 r_0^2)^{-1}
\end{equation}
whenever $t \in [-\tau_0 r_0^2, 0]$ and $\dist_t(p, p_0) \le \frac14 r_0$.
\end{lemma}
\begin{proof}
The proof is similar to that in \cite[Section 45]{Kleiner-Lott}.
\end{proof}

\begin{corollary} \label{corollary6.15}
For every $w > 0$, there are $B = B(w) < \infty$, $C = C(w) < \infty$
and $\tau_0 = \tau_0(w) > 0$ with the following properties.
Suppose that we have a (not necessarily
complete) Ricci flow solution $(\Or, g(\cdot))$, defined for
$t \in [-\tau_0 r_0^2, 0]$, so that at time zero the metric ball $B(p_0, r_0)$
has compact closure.  Suppose that for each $t \in [-\tau_0 r_0^2, 0]$,
the curvature operator in the time-$t$ ball $B(p_0, r_0)$ is
bounded below by $- r_0^{-2}$. If we assume a time-zero volume
bound $\vol(B_0(p_0, r_0)) \ge w r_0^n$ then
\begin{equation}
R(p,t) \le C r_0^{-2} + B (t + \tau_0 r_0^2)^{-1}
\end{equation}
whenever $t \in [-\tau_0 r_0^2, 0]$ and $\dist_t(p, p_0) \le \frac14 r_0$.
\end{corollary}
\begin{proof}
The proof is similar to that in \cite[Section 45]{Kleiner-Lott}.
\end{proof}

\subsection{Compactness of the space of three-dimensional
$\kappa$-solutions} \label{subsect6.7}

\begin{proposition} \label{prop6.17}
Given $\kappa > 0$, the set of oriented three-dimensional
$\kappa$-solutions $(\Or, g(\cdot))$ is compact modulo scaling.
\end{proposition}
\begin{proof}
If $\{(\Or_k, (p_k, 0), g_k(\cdot))\}_{k=1}^\infty$ is a sequence of
such $\kappa$-solutions with $R(p_k, 0) = 1$ then parts (1) and (2)
of Lemma \ref{lemma6.10} imply that there is a subsequence that converges
to an ancient solution $(\Or_\infty, (p_\infty, 0), g_\infty(\cdot))$
which has nonnegative curvature operator and is $\kappa$-noncollapsed.
The remaining issue is to show that it has bounded curvature.
Since $R_t \ge 0$, it is enough to show that $(\Or_\infty, g_\infty(0))$
has bounded scalar curvature.
 
If not then there is a sequence of points $q_i$ going to infinity
in $|\Or_\infty|$ such that $R(q_i, 0) \rightarrow \infty$ and
$R(q,0) \le 2 R(q_i,0)$ for $q \in B(q_i, A_i R(q_i,0)^{- \frac12})$,
where $A_i \rightarrow \infty$. Using the $\kappa$-noncollapsing, a
subsequence of the rescalings $(\Or_\infty, q_i, R(q_i, 0) g_\infty)$
will converge to a limit orbifold $N_\infty$ that isometrically splits
off an $\R$-factor. By Lemma \ref{lemma6.4}, $N_\infty$ must be a standard
solution on $\R \times (S^2//\Gamma)$ or $\R \times (\Sigma^2//\Gamma)$.
Thus $(\Or_\infty, g_\infty)$ contains a sequence $X_i$ of
neck regions, with their cross-sectional radii tending to zero as
$i \rightarrow \infty$. This contradicts Proposition \ref{prop3.14}. 
\end{proof}

\subsection{Necklike behavior at infinity of a three-dimensional
$\kappa$-solution} \label{subsect6.8}

\begin{definition}
Fix $\epsilon > 0$. Let $(\Or, g(\cdot))$ be an oriented 
three-dimensional $\kappa$-solution.  We say that a point
$p_0 \in |\Or|$ is the {\em center of an $\epsilon$-neck} if
the solution $g(\cdot)$ in the set 
$\{(p,t) : -(\epsilon Q)^{-1} < t \le 0, \dist_0(p, p_0)^2
< (\epsilon Q)^{-1}\}$, where $Q = R(p_0, 0)$, is, after
scaling with the factor $Q$, $\epsilon$-close in some fixed
smooth topology to the corresponding subset of 
a $\kappa$-solution $\R \times \Or^\prime$
that splits off an $\R$-factor.  That is, $\Or^\prime$ is
the standard
evolving $S^2//\Gamma$ or $\Sigma^2//\Gamma$
with extinction time $1$.
Here $\Sigma^2$ is a simply-connected
bad $2$-orbifold with a Ricci soliton metric.
\end{definition}

We let $|\Or|_\epsilon$ denote the points in $|\Or|$ which are not
centers of $\epsilon$-necks.

\begin{proposition} \label{prop6.19}
For all $\kappa > 0$, there exists an $\epsilon_0 > 0$ such that
for all $0 < \epsilon < \epsilon_0$ there exists an
$\alpha = \alpha(\epsilon,\kappa) $ with the property that
for any oriented three dimensional $\kappa$-solution
$(\Or, g(\cdot))$, and at any time $t$, precisely one of the
following holds :
\begin{itemize}
\item $(\Or, g(\cdot))$ splits off an $\R$-factor and so every
point at every time is the center of an $\epsilon$-neck for all
$\epsilon > 0$.
\item $\Or$ is noncompact, $|\Or|_\epsilon \neq \emptyset$, and for
all $x, y \in |\Or|_\epsilon$, we have $R(x) \: d^2(x,y) \: < \: \alpha$.
\item $\Or$ is compact, and there is a pair of points $x,y \in |\Or|_\epsilon$
such that $R(x) \: d^2(x,y) \: > \: \alpha$,\
\begin{equation}
|\Or|_\epsilon \subset B \left(x, \alpha R(x)^{-\frac12} \right) \cup
B \left(y, \alpha R(y)^{-\frac12} \right),
\end{equation}
and there is a minimizing geodesic $\overline{xy}$ such that every
$z \in |\Or|- |\Or|_\epsilon$ satisfies $R(z) \: d^2(z, \overline{xy}) <
\alpha$.
\item $\Or$ is compact and there exists a point $x \in |\Or|_\epsilon$
such that $R(x) \: d^2(x,z) \: < \: \alpha$ for all $z \in |\Or|$.
\end{itemize}
\end{proposition}
\begin{proof}
The proof is similar to that in \cite[Section 48]{Kleiner-Lott}.
\end{proof}

\subsection{Three-dimensional gradient shrinking $\kappa$-solutions}
\label{subsect6.9}

\begin{lemma} \label{lemma6.20}
Any three-dimensional gradient shrinking 
$\kappa$-solution $\Or$ is one of the following:
\begin{itemize}
\item A finite isometric quotient of the round shrinking $S^3$.
\item $\R \times (S^2//\Gamma)$ or $\R \times_{\Z_2} (S^2//\Gamma)$ 
for some finite group $\Gamma
\subset \Isom(S^2)$.
\item $\R \times (\Sigma^2//\Gamma)$ or $\R \times_{\Z_2} (\Sigma^2//\Gamma)$ 
for some finite group $\Gamma \subset \Isom(\Sigma^2)$.
\end{itemize}
\end{lemma}
\begin{proof}
As $\Or$ is a $\kappa$-solution, we know that $\Or$ has nonnegative
sectional curvature.
If $\Or$ has positive sectional curvature then the proofs of
\cite[Theorem 3.1]{Ni-Wallach} or \cite[Theorem 1.2]{Petersen-Wylie}
show that $\Or$ is a finite isometric quotient of the round
shrinking $S^3$.

Suppose that $\Or$ does not have positive sectional curvature.
Let $f \in C^\infty(\Or)$ denote the soliton potential function.
Let $\widetilde{\Or}$ be the universal cover of $\Or$ and
let $\widetilde{f} \in C^\infty(\widetilde{\Or})$ be the pullback
of $f$ to $\widetilde{\Or}$.
The strong maximum principle, as in \cite[Section 8]{Hamiltonnnn},
implies that if $p \in |\Or|_{reg}$
then there is an orthogonal splitting $T_p \Or = E_1 \oplus E_2$
which is invariant under holonomy around loops based at $p$. 
The same will be true on $\widetilde{\Or}$. Lemma \ref{lemma2.12} implies
that $\widetilde{\Or} = \R \times \Or^\prime$ for some
two-dimensional simply-connected gradient shrinking 
$\kappa$-solution $\Or^\prime$. From Lemma \ref{lemma6.4},
$\Or^\prime$ is the round shrinking $2$-sphere or the Ricci soliton
metric on a bad $2$-orbifold $\Sigma^2$. Now
$\widetilde{f}$ must be $- \frac{s^2}{4} + f^\prime$, where
$s$ is a coordinate on the $\R$-factor and $f^\prime$ is the
soliton potential function on $\Or^\prime$. As $\pi_1(\Or)$
preserves $\widetilde{f}$, and acts properly discontinuously 
and isometrically on $\R \times \Or^\prime$, it follows that
$\pi_1(\Or)$ is a finite subgroup of $\Isom^+(\R \times \Or^\prime)$.
\end{proof}

\begin{remark} \label{remark6.21}
In the manifold case, the nonexistence of noncompact positively-curved
three-dimensional $\kappa$-noncollapsed gradient shrinkers was
first proved by Perelman \cite[Lemma 1.2]{Perelman2}. Perelman's argument
applied the Gauss-Bonnet theorem to level sets of the soliton
function. This argument could be extended to orbifolds if one
assumes that there are no bad $2$-suborbifolds, as in Theorem
\ref{theorem1.1}. However, it is not so clear how it would extend without
this assumption.  Instead we use the arguments of
\cite[Theorem 3.1]{Ni-Wallach} or \cite[Theorem 1.2]{Petersen-Wylie},
which do extend to the general orbifold setting.
\end{remark}

\subsection{Getting a uniform value of $\kappa$} \label{subsect6.10}

\begin{lemma} \label{lemma6.22}
Given $N \in \Z^+$, there is a $\kappa_0 = \kappa_0(N) > 0$ so that
if $(\Or, g(\cdot))$ is an oriented three-dimensional $\kappa$-solution
for some $\kappa > 0$,
with $|G_p| \le N$ for all $p \in |\Or|$, then it is a $\kappa_0$-solution
or it is a quotient of the round shrinking $S^3$.
\end{lemma}
\begin{proof}
The proof is similar to that in \cite[Section 50]{Kleiner-Lott}.
The bound on $|G_p|$ gives a finite number of possible noncompact
asymptotic solitons from Lemma \ref{lemma6.20},
since a given closed two-dimensional orbifold has a unique Ricci soliton
metric up to scaling, 
and the topological type of $S^2//\Gamma$ (or $\Sigma//\Gamma$)
is determined by the number of singular points (which is at most three) and
the isotropy groups of those points..
\end{proof}

\begin{lemma} \label{lemma6.23}
Given $N \in \Z^+$, there is a universal constant $\eta = \eta(N) > 0$
such that at each point of every 
three-dimensional ancient solution $(\Or, g(\cdot))$
that is a $\kappa$-solution for some $\kappa > 0$, and has
$|G_p| \le N$ for all $p \in |\Or|$, we have estimates
\begin{equation}
|\nabla R| < \eta R^{\frac32}, \: \: \: \: |R_t| \le \eta R^2.
\end{equation}
\end{lemma}
\begin{proof}
The proof is similar to that in \cite[Section 59]{Kleiner-Lott}.
\end{proof}

\section{Ricci flow with surgery for orbifolds} \label{sect7}

In this section we construct the Ricci-flow-with-surgery for
three-dimensional orbifolds.

Starting in Subsection \ref{subsect7.2}, we will assume that there are
no bad $2$-dimensional
suborbifolds. Starting in Subsection \ref{subsect7.5}, we will
assume that the Ricci flows have normalized initial conditions,
as defined there.

\subsection{Canonical neighborhood theorem} \label{subsect7.1}

\begin{definition}
Let $\Phi \in C^\infty(\R)$ be a positive nondecreasing function such 
that for positive $s$, $\frac{\Phi(s)}{s}$ is a decreasing function
which tends to zero as $s \rightarrow \infty$. A Ricci flow solution
is said to have {\em $\Phi$-almost nonnegative curvature} if for all
$(p, t)$, we have
\begin{equation}
\Rm(p,t) \ge - \Phi(R(p,t)).
\end{equation}
\end{definition}

Our example of $\Phi$-almost nonnegative curvature comes from the
Hamilton-Ivey pinching condition \cite[Appendix B]{Kleiner-Lott},
which is valid for any three-dimensional orbifold Ricci flow solution
which has complete time slices, bounded curvature on compact
time intervals, and initial curvature operator bounded below by
$- I$.

\begin{proposition} \label{prop7.3}
Given $\epsilon, \kappa, \sigma > 0$ and a function $\Phi$
as above, one can find $r_0 > 0$ with the following property.  Let
$(\Or, g(\cdot))$ be a Ricci flow solution on a three-dimensional
orbifold $\Or$, defined for
$0 \le t \le T$ with $T \ge 1$. We suppose that for each $t$,
$g(t)$ is complete, and the sectional curvature in bounded on compact
time intervals.  Suppose that the Ricci flow has $\Phi$-almost
nonnegative curvature and is $\kappa$-noncollapsed on scales less than
$\sigma$. Then for any point $(p_0, t_0)$ with $t_0 \ge 1$ and
$Q = R(p_0,t_0) \ge r_0^{-2}$, the solution in
$\{ (p,t) : \dist_{t_0}^2(p,p_0) < (\epsilon Q)^{-1},
t_0 - (\epsilon Q)^{-1} \le t \le t_0\}$ is, after scaling by the
factor $Q$, $\epsilon$-close to the corresponding subset of a
$\kappa$-solution.
\end{proposition}
\begin{proof}
The proof is similar to that in \cite[Section 52]{Kleiner-Lott}.
We have to allow for the possibility of neck-like
regions approximated by 
$\R \times (S^2//\Gamma)$ or $\R \times (\Sigma^2//\Gamma)$. 
In the proof of
\cite[Lemma 52.12]{Kleiner-Lott},
the ``injectivity radius'' can be replaced by the ``local volume''.  
\end{proof}

\subsection{Necks and horns} \label{subsect7.2}

\begin{assumption}
Hereafter, we only consider three-dimensional orbifolds that do not
contain embedded bad $2$-dimensional suborbifolds.
\end{assumption}

In particular, neck regions will be modeled on $\R \times (S^2//\Gamma)$,
where $S^2//\Gamma$ is a quotient of the round shrinking $S^2$. 

We let $B(p,t,r)$ denote the open metric ball of radius $r$, with respect to
the metric at time $t$, centered at $p \in |\Or|$.

We let $P(p,t,r,\Delta t)$ denote a parabolic neighborhood, that is the set
of all points $(p^\prime, t^\prime)$ with $p^\prime \in B(p,t,r)$ and
$t^\prime \in [t, t + \Delta t]$ or $t^\prime \in [t+ \Delta t, t]$,
depending on the sign of $\Delta t$.

\begin{definition}
An open set $U \subset |\Or|$ in a Riemannian $3$-orbifold $\Or$ is an 
{\em $\epsilon$-neck} if modulo rescaling, it has distance less than
$\epsilon$, in the $C^{[1/\epsilon]+1}$-topology, to a product
$(-L, L) \times (S^2//\Gamma)$, where $S^2//\Gamma$ has constant scalar
curvature $1$ and $L > \epsilon^{-1}$.
If a point $p \in |\Or|$ and a neighborhood $U$ of
$p$ are specified then we will understand that ``distance'' refers
to the pointed topology. 
With an $\epsilon$-approximation $f : (-L,L) \rightarrow (S^2//\Gamma)
\rightarrow U$ being understood, a {\em cross-section} of the neck
is the image of $\{\lambda \} \times (S^2//\Gamma)$ for some
$\lambda \in (-L, L)$. 
\end{definition}

\begin{definition}
A subset of the form $\Or \Big|_U \times [a,b] 
\subset \Or \times [a,b]$ sitting in the spacetime of a Ricci flow, where
$U \subset |\Or|$ is open,
is a {\em strong $\epsilon$-neck} if after parabolic rescaling and
time shifting, it has distance less than $\epsilon$ to the
product Ricci flow defined on the time interval $[-1,0]$ which,
at its final time, is isometric to 
$(-L, L) \times (S^2//\Gamma)$, where $S^2//\Gamma$ has constant scalar
curvature $1$ and $L > \epsilon^{-1}$.
\end{definition}

\begin{definition}
A metric on $(-1,1) \times (S^2//\Gamma)$ such that each point is contained
in an $\epsilon$-neck is called an {\em $\epsilon$-tube}, an
{\em $\epsilon$-horn} or a {\em double $\epsilon$-horn} if the scalar
curvature stays bounded on both ends, stays bounded on one end and
tends to infinity on the other, or tends to infinity on both ends,
respectively.

A metric on $B^3//\Gamma$ or $(-1,1) \times_{\Z_2} (S^2//\Gamma)$,
such that each point
outside some compact subset is contained in an $\epsilon$-neck,
is called an {\em $\epsilon$-cap} or a {\em capped $\epsilon$-horn},
if the scalar curvature stays bounded or tends to infinity on the
end, respectively.
\end{definition}

\begin{lemma} \label{lemma7.8}
Let $U$ be an $\epsilon$-neck in an $\epsilon$-tube (or horn) and
let $S = S^2//\Gamma$ be a cross-sectional $2$-sphere quotient in $U$.
Then $S$ separates the two ends of the tube (or horn).
\end{lemma}
\begin{proof}
The proof is similar to that in \cite[Section 58]{Kleiner-Lott}.
\end{proof}

\subsection{Structure of three-dimensional $\kappa$-solutions}
\label{subsect7.3}

Recall the definition of $|\Or|_\epsilon$ from Subsection
\ref{subsect6.8}.

\begin{lemma} \label{lemma7.9}
If $(\Or, g(t))$ is a time slice of a noncompact 
three-dimensional $\kappa$-solution
and $|\Or|_\epsilon \neq \emptyset$ then there is a compact
suborbifold-with-boundary $X \subset \Or$ so that
$|\Or|_\epsilon \subset X$, $X$ is diffeomorphic to 
$D^3//\Gamma$ or
$I \times_{\Z_2} (S^2//\Gamma)$, 
and $\Or - \Int(X)$ is diffeomorphic to $[0,\infty) \times
(S^2//\Gamma)$.
\end{lemma}
\begin{proof}
The proof is similar to that in \cite[Section 59]{Kleiner-Lott}.
\end{proof}

\begin{lemma} \label{lemma7.10}
If $(\Or, g(t))$ is a time slice of a 
three-dimensional $\kappa$-solution
with $|\Or|_\epsilon = \emptyset$ then
the Ricci flow is the evolving round cylinder $\R \times (S^2//\Gamma)$.
\end{lemma}
\begin{proof}
The proof is similar to that in \cite[Section 59]{Kleiner-Lott}.
\end{proof}

\begin{lemma} \label{lemma7.11}
If a three-dimensional $\kappa$-solution $(\Or, g(\cdot))$ is compact and has
a noncompact asymptotic soliton then $\Or$ is diffeomorphic to
$S^3//\Z_k$ or $S^3//D_k$ for some $k \ge 1$.
\end{lemma}
\begin{proof}
The proof is similar to that in \cite[Section 59]{Kleiner-Lott}.
\end{proof}

\begin{lemma} \label{lemma7.12}
For every sufficiently small $\epsilon > 0$ one
can find $C_1 = C_1(\epsilon)$ and $C_2 = C_2(\epsilon)$ such that
for each point $(p,t)$ in every $\kappa$-solution, there is a radius 
$r \in [R(p,t)^{-\frac12}, C_1 R(p,t)^{- \frac12}]$ and a 
neighborhood $B$, $\overline{B(p,t,r)} \subset B \subset B(p,t,2r)$,
which falls into one of the four categories :

(a) $B$ is a strong $\epsilon$-neck, or

(b) $B$ is an $\epsilon$-cap, or

(c) $B$ is a closed orbifold diffeomorphic to $S^3//\Z_k$ or
$S^3//D_k$ for some $k \ge 1$.

(d) $B$ is a closed orbifold of constant positive sectional curvature. \\ \\
Furthermore:
\begin{itemize}
\item The scalar curvature in $B$ at time $t$ is between
$C_2^{-1} R(p,t)$ and $C_2 R(p,t)$.
\item The volume of $B$ is cases (a), (b) and (c) is greater than
$C_2^{-1} R(p,t)^{- \frac32}$.
\item In case (b), there is an $\epsilon$-neck $U \subset B$ with
compact complement in $B$ such that the distance from $p$ to
$U$ is at least $10000 R(p,t)^{-\frac12}$.
\item In case (c) the sectional curvature in $B$ is greater than
$C_2^{-1} R(p,t)$.
\end{itemize}
\end{lemma}
\begin{proof}
The proof is similar to that in \cite[Section 59]{Kleiner-Lott}.
\end{proof}

\subsection{Standard solutions} \label{subsect7.4}

Put $\Or = \R^3//\Gamma$, where $\Gamma$ is a finite subgroup
of $\SO(3)$. We fix a smooth $\SO(3)$-invariant metric
$g_0$ on $\R^3$ which is the result of gluing a hemispherical-type cap
to a half-infinite cylinder $[0, \infty) \times S^2$
of scalar curvature $1$. We also use $g_0$ to denote the
quotient metric on $\Or$. Among other properties,
$g_0$ is complete and has nonnegative curvature operator. 
We also assume that $g_0$ has scalar curvature bounded below by $1$.  

\begin{definition}
A Ricci flow $(\R^3//\Gamma, g(\cdot))$ defined on a time interval
$[0,a)$ is a {\em standard solution} if 
it has complete time slices,
it has initial condition
$g_0$, the curvature $|\Rm|$ is bounded on compact time intervals
$[0,a^\prime] \subset [0,a)$, and it cannot be extended to a Ricci
flow with the same properties on a strictly longer time interval.
\end{definition}

\begin{lemma} \label{lemma7.14}
Let $(\R^3//\Gamma, g(\cdot))$ be a standard solution. Then:
\begin{enumerate}
\item The curvature operator of $g$ is nonnegative.
\item All derivatives of curvature are bounded for small time,
independent of the standard solution.
\item The blowup time is $1$ and the infimal scalar curvature on
the time-$t$ slice tends to infinity as $t \rightarrow 1^-$ uniformly
for all standard solutions.
\item $(\R^3//\Gamma, g(\cdot))$ is $\kappa$-noncollapsed at scales
below $1$ on any time interval contained in $[0,1)$, where
$\kappa$ depends only on $g_0$ and $|\Gamma|$.
\item  $(\R^3//\Gamma, g(\cdot))$ satisfies the conclusion of
Proposition \ref{prop7.3}.
\item $R_{\min}(t) \ge \const (1-t)^{-1}$, where the constant
does not depend on the standard solution.
\item The family ${\mathcal S} {\mathcal T}$ of pointed standard
solutions $\{({\mathcal M}, (p, 0)) \}$ is compact with respect to
pointed smooth convergence.
\end{enumerate}
\end{lemma}
\begin{proof}
Working equivariantly, the proof is the same as that in
\cite[Sections 60-64]{Kleiner-Lott}.
\end{proof}

\subsection{Structure at the first singularity time} \label{subsect7.5}

\begin{definition}
Given $v_0 > 0$, 
a compact Riemannian three-dimensional orbifold $\Or$ is {\em normalized} if
$|\Rm| \le 1$ everywhere and for every $p \in |\Or|$, we have
$\vol(B(p,1)) \ge v_0$.
\end{definition}

Here $v_0$ is a global parameter in the sense that it will be
fixed throughout the rest of the paper.
If $\Or$ is normalized then the Bishop-Gromov inequality implies that
there is a uniform upper bound
$N = N(v_0) < \infty$ on the order of the isotropy groups;
cf. the proof of Lemma \ref{lemma5.16.5}.
The next lemma says that by rescaling we can always achieve
a normalized metric.

\begin{lemma}
Given $N \in \Z^+$, there is a $v_0 = v_0(N) > 0$ with the
following property.
Let $\Or$ be a compact orientable Riemannian three-dimensional orbifold,
whose isotropy groups have order at most $N$. Then a rescaling
of $\Or$ will have a normalized metric.
\end{lemma}
\begin{proof}
Let $c_3$ be the volume of the unit ball in $\R^3$.
Consider a ball $B_r$ of radius $r > 0$ with arbitrary center in a
Euclidean orbifold $\R^3//G$,
where $G$ is a finite subgroup of $O(3)$ with order at most $N$.
Applying the Bishop-Gromov inequality to compare the volume of $B_r$
with the volume of a very large ball having the same center, we
see that $\vol(B_r) \ge \frac{c_3}{N} r^3$.
Put $v_0 = \frac{c_3}{2N}$.
We claim that this value of $v_0$ satisfies the lemma.

To prove this by contradiction, suppose that there is an orbifold
$\Or$ which satisfies the hypotheses of the lemma but for which
the conclusion fails. Then there is a sequence $\{r_i\}_{i=1}^\infty$
of positive numbers
with $\lim_{i \rightarrow \infty} r_i = 0$ along with points
$\{p_i\}_{i=1}^\infty$ in $|\Or|$ so that for each $i$,
we have $\vol(B(p_i, r_i)) < v_0 r_i^3$. After passing to
a subsequence, we can assume that $\lim_{i \rightarrow \infty} p_i
= p'$ for some $p' \in |\Or|$. 
Using the inverse exponential map,
for large $i$ the ball $B(p_i, r_i)$ will, up to small distortion,
correspond to a ball of radius $r_i$ in the tangent space
$T_{p^\prime} \Or$. In view of our choice of $v_0$, this
is a contradiction.
\end{proof}

\begin{assumption}
Hereafter we assume that
our Ricci flows have normalized initial condition.
\end{assumption}

Consider the labels on the edges in the singular part of the orbifold.
They clearly do not change under a smooth Ricci flow.
If some components of the orbifold are discarded at a singularity time
then the set of
edge labels can only change by deletion of some labels. Otherwise,
the surgery procedure will be such that the set of edge labels does not
change,  
although the singular graphs will change. 
Hence the normalized initial condition implies
a uniform upper bound on the orders of the isotropy groups
for all time. 

Let $\Or$ be a connected closed oriented $3$-dimensional orbifold.
Let $g(\cdot)$ be a Ricci flow on $\Or$ defined on a maximal time
interval $[0,T)$ with $T < \infty$. For any
$\epsilon > 0$, we know that there are numbers $r = r(\epsilon) > 0$
and $\kappa = \kappa(\epsilon) > 0$ so that for any point $(p,t)$ with
$Q = R(p,t) \ge r^{-2}$, the solution in
$P(p,t,(\epsilon Q)^{-\frac12}, (\epsilon Q)^{-1})$ is (after rescaling
by the factor $Q$) $\epsilon$-close to the corresponding subset of a
$\kappa$-solution.

\begin{definition}
Define a subset $\Omega$ of $|\Or|$ by
\begin{equation}
\Omega = \{p \in |\Or| : \sup_{t \in [0,T)} |\Rm|(p,t) < \infty \}.
\end{equation}
\end{definition}

\begin{lemma} \label{lemma7.17}
We have
\begin{itemize}
\item $\Omega$ is open in $|\Or|$.
\item Any connected component of $\Omega$ is noncompact.
\item If $\Omega = \emptyset$ then $\Or$ is diffeomorphic to
$S^3//\Gamma$ or $(S^1 \times S^2)//\Gamma$.
\end{itemize}
\end{lemma}
\begin{proof}
The proof is similar to that in
\cite[Section 67]{Kleiner-Lott}.
\end{proof}

\begin{definition}
Put $\overline{g} = \lim_{t \rightarrow T^-} g(t) \Big|_\Omega$,
a smooth Riemannian metric on $\Or \Big|_\Omega$.
Let $\overline{R}$ denote its scalar curvature.
\end{definition}

\begin{lemma} \label{lemma7.19}
$(\Omega, \overline{g})$ has finite volume.
\end{lemma}
\begin{proof}
The proof is similar to that in
\cite[Section 67]{Kleiner-Lott}.
\end{proof}

\begin{definition}
For $\rho < \frac{r}{2}$, put 
$\Omega_\rho = \{p \in |\Omega| : \overline{R}(p) \le \rho^{-2} \}$.
\end{definition}

\begin{lemma} \label{lemma7.21}
We have
\begin{itemize}
\item $\Omega_\rho$ is a compact subset of $|\Or|$.
\item If $C$ is a connected component of $\Omega$ which does not
intersect $\Omega_\rho$ then $C$ is a double $\epsilon$-horn or a
capped $\epsilon$-horn.
\item There is a finite number of connected components of $\Omega$
that intersect $\Omega_\rho$, each such component having a finite
number of ends, each of which is an $\epsilon$-horn.
\end{itemize}
\end{lemma}
\begin{proof}
The proof is similar to that in
\cite[Section 67]{Kleiner-Lott}.
\end{proof}

\subsection{$\delta$-necks in $\epsilon$-horns} \label{subsect7.6}

We define a {\em Ricci flow with surgery}
${\mathcal M}$  to be the obvious orbifold extension of
\cite[Section 68]{Kleiner-Lott}. 
The objects defined there have evident analogs in the orbifold setting.

The {\em $r$-canonical neighborhood assumption} is the obvious orbifold
extension of what's in \cite[Section 69]{Kleiner-Lott}, with condition
(c) replaced by ``$\Or$ is a closed orbifold diffeomorphic to an
isometric quotient of $S^3$''. 

The {\em $\Phi$-pinching assumption} is the same as in 
\cite[Section 69]{Kleiner-Lott}.  

The {\em a priori assumptions} consist of the $\Phi$-pinching
assumption and the $r$-canonical neighborhood assumption.

\begin{lemma} \label{lemma7.23}
Given the pinching function $\Phi$, a number
$\widehat{T} \in (0, \infty)$, a positive nonincreasing function
$r : [0, \widehat{T}] \rightarrow \R$ and a number $\delta \in (0, \frac12)$,
there is a nonincreasing function $h : [0, \widehat{T}] \rightarrow \R$
with $0 < h(t) < \delta^2 r(t)$ so that the following property is satisfied.
Let ${\mathcal M}$ be a Ricci flow with surgery defined on $[0,T)$,
with $T < \widehat{T}$, which satisfies the a priori assumptions
and which goes singular at time $T$. Let $(\Omega, \overline{g})$ denote
the time-$T$ limit. Put $\rho = \delta r(T)$ and
\begin{equation}
\Omega_\rho = \{(p,T) \in \Omega : \overline{R}(p,T) \le \rho^{-2}\}.
\end{equation}
Suppose that $(p,T)$ lies in an $\epsilon$-horn ${\mathcal H} \subset
\Omega$ whose boundary is contained in $\Omega_\rho$. Suppose also that
$\overline{R}(p,T) \ge h^{-2}(T)$. Then the parabolic region
$P(p,T,\delta^{-1} \overline{R}(p,T)^{- \frac12}, -
\overline{R}(p,T)^{-1})$ is contained in a strong $\delta$-neck.
\end{lemma}
\begin{proof}
The proof is similar to that in \cite[Section 71]{Kleiner-Lott}.
\end{proof}

\subsection{Surgery and the pinching condition} \label{subsect7.7}

\begin{lemma} \label{lemma7.25}
There exists $\delta^\prime = \delta^\prime(\delta)
> 0$ with $\lim_{\delta \rightarrow 0} \delta^\prime(\delta) = 0$
and a constant $\delta_0 > 0$ such that the following
holds. Suppose that $\delta < \delta_0$, $p \in \{0\} \times 
(S^2//\Gamma)$ and $h_0$ is a Riemannian metric on $(-A, \frac{1}{\delta})
\times (S^2//\Gamma)$ with $A > 0$ and $R(p) > 0$ such that:
\begin{itemize}
\item $h_0$ satisfies the time-$t$ Hamilton-Ivey pinching condition.
\item $R(p) h_0$ is $\delta$-close to $g_{cyl}$ in the 
$C^{[\frac{1}{\delta}]+1}$-topology.
\end{itemize}
Then there are a $B = B(A) > 0$ and a smooth metric $h$ on
$\R^3//\Gamma = (D^3//\Gamma) \cup ((-B, \frac{1}{\delta}) \times (S^2//\Gamma)
)$ such that
\begin{itemize}
\item $h$ satisfies the time-$t$ pinching condition.
\item The restriction of $h$ to $[0, \frac{1}{\delta}) \times (S^2//\Gamma)$
is $h_0$.
\item The restiction of $R(p) h$ to $(-B, -A) \times (S^2//\Gamma)$ is $g_0$,
the initial metric of a standard solution.
\end{itemize}
\end{lemma}
\begin{proof}
The proof is the same as that in \cite[Section 72]{Kleiner-Lott},
working equivariantly.
\end{proof}

We define a {\em Ricci flow with $(r, \delta)$-cutoff}
by the obvious orbifold extension of the definition
in \cite[Section 73]{Kleiner-Lott}.

In the surgery procedure, one first throws away all connected
components of $\Omega$ which do not intersect $\Omega_\rho$.
For each connected component $\Omega_j$ of $\Omega$ that intersects
$\Omega_\rho$ and for each $\epsilon$-horn of $\Omega_j$, take
a cross-sectional $S^2$-quotient that lies far in the $\epsilon$-horn.
Let $X$ be what's left after cutting the $\epsilon$-horns at the
$2$-sphere quotients and removing the tips. The (possibly disconnected)
postsurgery orbifold $\Or^\prime$ is the result of capping off
$\partial X$ by discal $3$-orbifolds.

\begin{lemma} \label{lemma7.26}
The presurgery orbifold can be obtained from the postsurgery orbifold
by applying the following operations finitely many times:
\begin{itemize}
\item Taking the disjoint union with a finite isometric quotient of
$S^1 \times S^2$ or $S^3$. 
\item Performing a $0$-surgery.
\end{itemize}
\end{lemma}
\begin{proof}
The proof is similar to that in \cite[Section 73]{Kleiner-Lott}.
\end{proof}

\subsection{Evolution of a surgery cap} \label{subsect7.8}

\begin{lemma} \label{lemma7.27}
For any $A < \infty$, $\theta \in (0,1)$ and
$\widehat{r} > 0$, one can find $\widehat{\delta} = 
\widehat{\delta}(A,\theta,\widehat{r}) > 0$ with the following property.
Suppose that we have a Ricci flow with $(r, \delta)$-cutoff defined
on a time interval $[a,b]$ with $\min r = r(b) \ge \widehat{r}$.
Suppose that there is a surgery time $T_0 \in (a,b)$ with
$\delta(T_0) \le \widehat{\delta}$. Consider a given surgery at the
surgery time and let $(p, T_0) \in {\mathcal M}^+_{T_0}$ be the
center of the surgery cap. Let
$\widehat{h} = h(\delta(T_0),\epsilon,r(T_0), \Phi)$ be the surgery
scale given by Lemma \ref{lemma7.23} and put $T_1 = \min(b, T_0 + \theta
\widehat{h}^2)$. Then one of the two following possibilities occurs:
\begin{enumerate}
\item The solution is unscathed on $P(p,T_0,A\widehat{h},T_1 - T_0)$. The
pointed solution there is, modulo parabolic rescaling, $A^{-1}$-close to
the pointed flow on $U_0 \times [0,(T_1-T_0) \widehat{h}^{-2}]$, where
$U_0$ is an open subset of the initial time slice $|{\mathcal S}_0|$ of
a standard solution ${\mathcal S}$ and the basepoint is the center of the
cap in $|{\mathcal S}_0|$.
\item Assertion (1) holds with $T_1$ replaced by some $t^+ \in [T_0, T_1)$,
where $t^+$ is a surgery time.  Moreover, the entire ball
$B(p,T_0, A\widehat{h})$ becomes extinct at time $t^+$.
\end{enumerate}
\end{lemma}
\begin{proof}
The proof is similar to that in \cite[Section 74]{Kleiner-Lott}.
\end{proof}

\subsection{Existence of Ricci flow with surgery} \label{subsect7.9}

\begin{proposition} \label{prop7.29}
There exist decreasing sequences $0 < r_j < \epsilon^2$,
$\kappa_j > 0$, $0 < \overline{\delta}_j < \epsilon^2$ for
$1 \le j \le \infty$, such that for any normalized initial data on
an orbifold $\Or$ and any
nonincreasing function $\delta : [0, \infty) \rightarrow (0, \infty)$
such that $\delta \le \overline{\delta}_j$ on
$[2^{j-1} \epsilon, 2^j \epsilon]$, the Ricci flow with
$(r, \delta)$-cutoff is defined for all time and is $\kappa$-noncollapsed
at scales below $\epsilon$. Here $r$ and $\kappa$ are the functions on
$[0, \infty)$ so that $r \Big|_{[2^{j-1}\epsilon, 2^j \epsilon]} = r_j$ and
$\kappa \Big|_{[2^{j-1}\epsilon, 2^j \epsilon]} = \kappa_j$, and
$\epsilon > 0$ is a global constant.
\end{proposition}
\begin{proof}
The proof is similar to that in \cite[Sections 77-80]{Kleiner-Lott}.
\end{proof}

\begin{remark} \label{remark7.30}
We restrict to $3$-orbifolds without bad $2$-suborbifolds in order to
perform surgery. Without this assumption, there could be a neckpinch
whose cross-section is a bad $2$-orbifold $\Sigma$.  In the case of a
nondegenerate neckpinch, the blowup limit
would be the product of $\R$ with an evolving Ricci soliton metric on
$\Sigma$. The problem in performing surgery is that after
slicing at a bad cross-section, there is no evident way to cap off the
ensuing pieces with $3$-dimensional orbifolds so as to preserve
the Hamilton-Ivey pinching condition.
\end{remark}

\section{Hyperbolic regions} \label{sect8}

In this section we show that the $w$-thick part of the evolving
orbifold approaches a finite-volume Riemannian orbifold with
constant curvature $- \: \frac14$.

As a standing assumption in this section,
we suppose that we
have a solution to the Ricci flow with $(r, \delta)$-cutoff
and with normalized initial data.

\subsection{Double sided curvature bounds in the thick part} \label{subsect8.1}

\begin{proposition} \label{prop8.1}
Given $w > 0$, one can find $\tau = \tau(w) > 0$,
$K = K(w) < \infty$, $\overline{r} = \overline{r}(w) > 0$ and
$\theta = \theta(w) > 0$ with the following property. 
Let $h_{max}(t_0)$ be the
maximal surgery radius on $[t_0/2, t_0]$. Let $r_0$ satisfy
\begin{enumerate}
\item $\theta^{-1} h_{max}(t_0) \le r_0 \le \overline{r} \sqrt{t_0}$.
\item The ball $B(p_0,t_0,r_0)$ has sectional curvatures at least
$- r_0^{-2}$ at each point.
\item $\vol(B(p_0,t_0,r_0)) \ge w r_0^3$.
\end{enumerate}
Then the solution is unscathed in $P(p_0,t_0,r_0/4,-\tau r_0^2)$
and satisfies $R < K r_0^{-2}$ there.
\end{proposition}
\begin{proof}
The proof is similar to that in \cite[Sections 81-86]{Kleiner-Lott}.
In particular, it uses Proposition \ref{prop5.18}.
\end{proof}

\subsection{Noncollapsed regions with a lower curvature bound are
almost hyperbolic on a large scale} \label{subsect8.2}

\begin{proposition} \label{prop8.2}

(a) Given $w, r, \xi > 0$, one can find $T = T(w,r,\xi) < 
\infty$ so that the following holds.  
If the ball $B(p_0, t_0, r\sqrt{t_0}) \subset
{\mathcal M}^+_{t_0}$ at some $t_0 \ge T$ has volume at least
$w r^3 r_0^{\frac32}$ and sectional curvatures at least $- r^{-2} t_0^{-1}$
then the curvature at $(p_0,t_0)$ satisfies
\begin{equation} \label{8.3}
|2tR_{ij}(p_0,t_0) + g_{ij}|^2 \le \xi^2.
\end{equation}

(b) Given in addition $A < \infty$ and allowing $T$ to depend on $A$,
we can ensure (\ref{8.3}) for all points in $B(p_0, t_0, Ar\sqrt{t_0})$.

(c) The same is true for $P(p_0, t_0, Ar\sqrt{t_0}, Ar^2t_0)$.
\end{proposition}
\begin{proof}
The proof is similar to that in \cite[Sections 87 and 88]{Kleiner-Lott}.
\end{proof}

\subsection{Hyperbolic rigidity and stabilization of the thick part}
\label{subsect8.3}

\begin{definition}
\label{def-curvaturescale}
Let $\Or$ be a complete Riemannian orbifold. Define the curvature scale
as follows.
Given
$p \in |\Or|$, if the connected component of $\Or$ 
containing $p$ has nonnegative sectional curvature then put
$R_p = \infty$. Otherwise, let $R_p$ be the unique number
$r \in (0, \infty)$ such that
$\inf_{B(p,r)} \Rm = - r^{-2}$.
\end{definition}

\begin{definition}
Let $\Or$ be a complete Riemannian orbifold.
Given $w > 0$, the {\em $w$-thin part}
$\Or^-(w) \subset |\Or|$ is the set of points
$p \in \Or$ so that either 
$R_p = \infty$ or
\begin{equation}
\vol(B(p,R_p) < w R_p^3.
\end{equation}
The {\em $w$-thick part} is $\Or^+(w) = |\Or| - 
\Or^-(w)$.
\end{definition}

In what follows, we take ``hyperbolic'' to mean ``constant curvature
$- \frac14$''. 
When applied to a hyperbolic orbifold, the definitions of
the thick and thin parts are essentially equivalent to those in
\cite[Chapter 6.2]{BMP}, to which we refer for more information
about hyperbolic $3$-orbifolds.

Recall that a hyperbolic $3$-orbifold can be written as
$H^3//\Gamma$ for some discrete group $\Gamma \subset \Isom^+(H^3)$
\cite[Theorem 2.26]{CHK}.

\begin{definition}
A {\em Margulis tube} is a compact quotient of 
a normal neighborhood of
a geodesic in $H^3$ by
an elementary Kleinian group. 

A {\em rank-$2$ cusp neighborhood} is the quotient of a horoball in
$H^3$ by an elementary rank-$2$ parabolic group.

In either case, the boundary is a compact Euclidean $2$-orbifold.
\end{definition}

There is a Margulis constant $\mu_0 > 0$ so that for any 
finite-volume hyperbolic $3$-orbifold $\Or$, if $\mu \le \mu_0$ then the
connected components of the $\mu$-thin part of $\Or$ are Margulis tubes
or rank-$2$ cusp neighborhoods.

Furthermore, given a finite-volume hyperbolic $3$-orbifold $\Or$, 
if $\mu > 0$ is sufficiently small then the
connected components of the $\mu$-thin part are
rank-$2$ cusp neighborhoods.

Mostow-Prasad rigidity works just as well for finite-volume
hyperbolic orbifolds as for finite-volume hyperbolic manifolds.
Indeed, 
the rigidity statements are statements about 
lattices in $\Isom(H^n)$.

\begin{lemma} \label{lemma8.6}
Let $(\Or, p)$ be a pointed complete connected finite-volume three-dimensional
hyperbolic orbifold.  Then for each $\zeta > 0$, there exists $\xi > 0$
such that if $\Or^\prime$ is a
complete connected finite-volume three-dimensional
hyperbolic orbifold with at least as many cusps as $\Or$, and
$f : (\Or, p) \rightarrow \Or^\prime$ is a $\xi$-approximation
in the pointed smooth topology as in \cite[Definition 90.6]{Kleiner-Lott},
then there is an isometry $f^\prime : (\Or, p) \rightarrow \Or^\prime$ which
is $\zeta$-close to $f$
in the pointed smooth topology.  
\end{lemma}
\begin{proof}
The proof is similar to that in \cite[Section 90]{Kleiner-Lott},
replacing ``injectivity radius'' by ``local volume''.
\end{proof}

If ${\mathcal M}$ is a Ricci flow with surgery then we let
$\Or^-(w,t) \subset |{\mathcal M}^+_{t}|$ denote the $w$-thin part
of the orbifold at time $t$ (postsurgery if $t$ is a surgery time), 
and similarly for the $w$-thick part
$\Or^+(w,t)$.

\begin{proposition} \label{prop8.11}
Given a Ricci flow with surgery ${\mathcal M}$, there exist a number
$T_0 < \infty$, a nonincreasing function $\alpha : [T_0, \infty)
\rightarrow (0, \infty)$ with $\lim_{t \rightarrow \infty} \alpha(t) = 0$,
a (possibly empty) collection $\{(H_1, x_1), \ldots, (H_N, x_N)\}$ of
complete connected pointed finite-volume three-dimensional hyperbolic
orbifolds and a family of smooth maps
\begin{equation}
f(t) : B_t = \bigcup_{i=1}^N H_i \Big|_{B(x_i, 1/\alpha(t))} \rightarrow
{\mathcal M}_t,
\end{equation}
defined for $t \in [T_0, \infty)$, such that
\begin{enumerate}
\item $f(t)$ is close to an isometry:
\begin{equation}
\parallel t^{-1} f(t)^* g_{{\mathcal M}_t} - g_{B_t} 
\parallel_{C^{[1/\alpha(t)]}} < \alpha(t).
\end{equation}
\item $f(t)$ defines a smooth family of maps which changes smoothly with time:
\begin{equation}
|\dot{f}(p,t)| < \alpha(t) t^{-\frac12}
\end{equation}
for all $p \in |B_t|$, and
\item $f(t)$ parametrizes more and more of the thick part:
$\Or^+(\alpha(t),t) \subset \Image(|f(t)|)$ for all $t \ge T_0$.
\end{enumerate}
\end{proposition}
\begin{proof}
The proof is similar to that in \cite[Section 90]{Kleiner-Lott}.
\end{proof}

\section{Locally collapsed $3$-orbifolds} \label{sect9}

In this section we consider compact Riemannian 
$3$-orbifolds $\Or$ that are locally collapsed
with respect to a local lower curvature bound. 
Under certain assumptions about smoothness and boundary behavior,
we show that $\Or$ is either the result of performing $0$-surgery
on a strong graph orbifold or is one of a few
special types.  We refer to Definition \ref{defnA.22} for the definition
of a strong graph orbifold.

We first consider the boundaryless case.

\begin{proposition} \label{prop9.1}
Let $c_3$ be the volume of the unit ball in $\R^3$, let $K \ge 10$ be
a fixed integer and let $N$ be a positive integer. Fix a function
$A : (0, \infty) \rightarrow (0, \infty)$. Then there is a 
$w_0 \in (0, c_3/N)$ such that the following holds.

Suppose that $(\Or, g)$ is a 
connected closed orientable Riemannian $3$-orbifold.
Assume in addition that for all $p \in |\Or|$,
\begin{enumerate} 
\item $|G_p| \le N$.
\item $\vol(B(p, R_p)) \le w_0 R_p^3$, where $R_p$ is the curvature scale at $p$, Definition
\ref{def-curvaturescale}.
\item For every $w^\prime \in [w_0, c_3/N)$, $k \in [0, K]$ and $r \le R_p$
such that $\vol(B(p,r)) \ge w^\prime r^3$, the inequality
\begin{equation}
|\nabla^k \Rm| \le A(w^\prime) r^{-(k+2)}
\end{equation}
holds in the ball $B(p,r)$.
\end{enumerate} 

Then $\Or$ is the result of performing $0$-surgeries on a strong graph orbifold or is diffeomorphic to an
isometric quotient of $S^3$ or $T^3$.
\end{proposition}

\begin{remark}
We recall that a strong graph orbifold is allowed to be disconnected.
By Proposition \ref{propA.25},
a weak graph orbifold is the result of performing
$0$-surgeries on a strong graph orbifold.
Because of this, to prove Proposition \ref{prop9.1} it is enough to show that
$\Or$ is the result of performing $0$-surgeries on a weak 
graph orbifold or is diffeomorphic to an
isometric quotient of $S^3$ or $T^3$.
\end{remark}

\begin{remark} \label{remark9.3}
A $3$-manifold which is an isometric quotient of $S^3$ or $T^3$ is
a Seifert $3$-manifold \cite[Section 4]{Scott}.  
The analogous statement for orbifolds is false \cite{Dunbar2}.
\end{remark}
\begin{proof}
We follow the method of proof of \cite{Kleiner-Lott2}. The basic
strategy is to construct a partition of $\Or$ into pieces whose
topology can be recognized.  Many of the arguments in
\cite{Kleiner-Lott2}, such as the stratification, are based
on the underlying Alexandrov space structure.  Such arguments will
extend without change to the orbifold setting. Other arguments
involve smoothness, which also makes sense in the orbifold setting.  We now
mention the relevant places in \cite{Kleiner-Lott2} where
manifold smoothness needs to be replaced by orbifold smoothness.
\begin{itemize}
\item The critical point theory in \cite[Section 3.4]{Kleiner-Lott2}
can be extended to the orbifold setting using the results in
Subsection \ref{subsect2.6}.
\item The results about the topology of nonnegatively curved manifolds
in \cite[Lemma 3.11]{Kleiner-Lott2} can be extended to the
orbifold setting using Lemma \ref{lemma3.16} and 
Proposition \ref{prop5.6}.
\item The smoothing results of \cite[Section 3.6]{Kleiner-Lott2} can
be extended to the orbifold setting using Lemma \ref{lemma2.18} and 
Corollary \ref{corollary2.19}.
\item The $C^K$-precompactness result of \cite[Lemma 6.10]{Kleiner-Lott2}
can be proved in the orbifold setting using Proposition \ref{prop4.1}.
\item The $C^K$-splitting result of 
\cite[Lemma 6.16]{Kleiner-Lott2}
can be proved in the orbifold setting using Proposition \ref{prop3.1}.
\item The result about the topology of the edge region in
\cite[Lemma 9.21]{Kleiner-Lott2}
can be extended to the orbifold setting using Lemma \ref{lemma3.17}.
\item The result about the topology of the slim stratum in
\cite[Lemma 10.3]{Kleiner-Lott2} can be extended to the orbifold
setting using Lemma \ref{lemma3.15}.
\item The results about the topology and geometry of the $0$-ball
regions in \cite[Sections 11.1 and 11.2]{Kleiner-Lott2} can be
extended to the orbifold setting using Lemma \ref{lemma2.17}
and Proposition \ref{prop3.11}.
\item The adapted coordinates in 
\cite[Lemmas 8.2, 9.12, 9.17, 10.1 and 11.3]{Kleiner-Lott2}
and their use
in \cite[Sections 12-14]{Kleiner-Lott2} extend without change
to the orbifold setting.
\end{itemize}

The upshot is that we can extend the results 
of \cite[Sections 1-14]{Kleiner-Lott2} to the orbifold setting.
This gives a partition of
$\Or$ into codimension-zero 
suborbifolds-with-boundary $\Or^{0-stratum}$, $\Or^{slim}$,
$\Or^{edge}$ and $\Or^{2-stratum}$, with the following properties.

\begin{itemize}
\item Each connected component of $\Or^{0-stratum}$ is diffeomorphic
either to a closed nonnegatively curved $3$-dimensional orbifold, or 
to the unit disk bundle in the normal bundle of a soul in a complete connected
noncompact nonnegatively curved $3$-dimensional orbifold.
\item Each connected component of $\Or^{slim}$ is
the total space of an orbibundle whose base is $S^1$ or $I$,
and whose fiber is a spherical or Euclidean orientable compact
$2$-orbifold.
\item Each connected component of $\Or^{edge}$ is
the total space of an orbibundle whose base is $S^1$ or $I$,
and whose fiber is $D^2(k)$ or $D^2(2,2)$.
\item Each connected component of $\Or^{2-stratum}$ is the
total space of a circle bundle over a smooth compact $2$-manifold.
\item Intersections of $\Or^{0-stratum}$, $\Or^{slim}$,
$\Or^{edge}$ and $\Or^{2-stratum}$ are $2$-dimensional
orbifolds, possibly with boundary. 
The fibration structures coming from two intersecting strata are compatible
on intersections. 
\end{itemize}

In order to prove the proposition,
we now follow the method of proof of \cite[Section 15]{Kleiner-Lott2}.

Each connected component of $\Or^{0-stratum}$ has boundary which
is empty, a spherical $2$-orbifold or a Eucldean $2$-orbifold.
By Proposition \ref{prop5.6}, 
if the boundary is empty then the component is diffeomorphic
to a finite isometric quotient of $S^1 \times S^2$, $S^3$ or $T^3$.
In the $S^1 \times S^2$ case, $\Or$ is a Seifert orbifold
\cite[p. 70-71]{Dunbar}.
Hence we can assume that the boundary is nonempty.
By Lemma \ref{lemma3.16}, if the boundary is a spherical $2$-orbifold then
the component is diffeomorphic to $D^3//\Gamma$ or
$I \times_{\Z_2} (S^2//\Gamma)$.
We group together such components as $\Or^{0-stratum}_{Sph}$.
By Lemma \ref{lemma3.16} again, if the boundary is
a Euclidean $2$-orbifold then the component is diffeomorphic to
$S^1 \times D^2$, $S^1 \times D^2(k)$, 
$S^1 \times_{\Z_2} D^2$, $S^1 \times_{\Z_2} D^2(k)$ 
or $I \times_{\Z_2} (T^2//\Gamma)$.
We group together such components as $\Or^{0-stratum}_{Euc}$.

If a connected component of $\Or^{slim}$ fibers over $S^1$ then 
$\Or$ is closed and has a geometric structure based on 
$\R^3$, $\R \times S^2$, $\Nil$ or $\Sol$
\cite[p. 72]{Dunbar}.  If the structure is $\R \times S^2$ or $\Nil$
then $\Or$ is a Seifert orbifold \cite[Theorem 1]{Dunbar}. 
If the structure is $\Sol$ then
$\Or$ can be cut along a fiber to see that it is a weak graph
orbifold.
Hence we can assume that each component of $\Or^{slim}$ fibers over $I$.
We group these components into
$\Or^{slim}_{Sph}$ and $\Or^{slim}_{Euc}$, where the distinction
is whether the fiber is a spherical $2$-orbifold or a 
Euclidean $2$-orbifold.

\begin{lemma} \label{lemma9.4}
Let $\Or^{0-stratum}_i$ be a connected component of $\Or^{0-stratum}$.
If $\Or^{0-stratum}_i \cap \Or^{slim} \neq \emptyset$ then
$\partial \Or^{0-stratum}_i$ is a boundary component of a connected
component of $\Or^{slim}$. 

If 
$\Or^{0-stratum}_i \cap \Or^{slim} = \emptyset$ then
we can write $\partial \Or^{0-stratum}_i = A_i \cup B_i$ where
\begin{enumerate}
\item $A_i = \Or^{0-stratum}_i \cap \Or^{edge}$ is a disjoint union
of discal $2$-orbifolds and $D^2(2,2)$'s.
\item $B_i = \Or^{0-stratum}_i \cap \Or^{2-stratum}$ is the total
space of a circle bundle and
\item $A_i \cap B_i = \partial A_i \cap \partial B_i$ is a union of
circle fibers.
\end{enumerate}
Furthermore, if $\partial \Or^{0-stratum}_i$ is Euclidean then
$A_i = \emptyset$ unless $\partial \Or^{0-stratum}_i = S^2(2,2,2,2)$,
in which case $A_i$ consists of two $D^2(2,2)$'s.
If $\partial \Or^{0-stratum}_i$ is spherical then the possibilities are \\
1. $\partial \Or^{0-stratum}_i = S^2$ and $A_i$ consists of
two disks $D^2$. \\
2. $\partial \Or^{0-stratum}_i = S^2(k,k)$ and $A_i$ consists of two 
$D^2(k)$'s. \\
3. $\partial \Or^{0-stratum}_i = S^2(2,2,k)$ and $A_i$ consists of 
$D^2(2,2)$ and $D^2(k)$.
\end{lemma}
\begin{proof}
The proof is similar to that of \cite[Lemma 15.1]{Kleiner-Lott2}.
\end{proof}

\begin{lemma} \label{lemma9.5}
Let $\Or^{slim}_i$ be a connected component of $\Or^{slim}$.
Let $Y_i$ be one of the connected components of $\partial \Or^{slim}_i$.
If $Y_i \cap \Or^{0-stratum} \neq \emptyset$ then
$Y_i = \partial \Or^{0-stratum}_i$ for some connected
component $\Or^{0-stratum}_i$ of $\Or^{0-stratum}$. 

If $Y_i \cap \Or^{0-stratum} = \emptyset$ then
we can write 
$\partial Y_i = A_i \cup B_i$
where
\begin{enumerate}
\item $A_i = Y_i \cap \Or^{edge}$ is a disjoint union
of discal $2$-orbifolds and $D^2(2,2)$'s,
\item $B_i = Y_i \cap \Or^{2-stratum}$ is the total
space of a circle bundle and
\item $A_i \cap B_i = \partial A_i \cap \partial B_i$ is a union of
circle fibers.
\end{enumerate}
Furthermore, if $Y_i$ is Euclidean then
$A_i = \emptyset$ unless $Y_i = S^2(2,2,2,2)$,
in which case $A_i$ consists of two $D^2(2,2)$'s.
If $Y_i$ is spherical then the possibilities are \\
1. $Y_i = S^2$ and $A_i$ consists of
two disks $D^2$. \\
2. $Y_i = S^2(k,k)$ and $A_i$ consists of two $D^2(k)$'s. \\
3. $Y_i = S^2(2,2,k)$ and $A_i$ consists of $D^2(2,2)$ and $D^2(k)$.
\end{lemma}
\begin{proof}
The proof is similar to that of \cite[Lemma 15.2]{Kleiner-Lott2}.
\end{proof}

Let $\Or^\prime_{Sph}$ be the union of the connected components of
$\Or^{0-stratum}_{Sph} \cup \Or^{slim}_{Sph}$ that do not
intersect $\Or^{edge}$. 
Then $\Or^\prime_{Sph}$ is either empty or is all of $\Or$, 
in which case $\Or$ is diffeomorphic to 
the gluing of two connected components of $\Or^{0-stratum}_{Sph}$
along a spherical $2$-orbifold. As each connected component is
diffeomorphic to some $D^3//\Gamma$ or $I \times_{\Z_2} (S^2//\Gamma)$,
it then follows that
$\Or$ is diffeomorphic to $S^3//\Gamma$, $(S^3//\Gamma)//\Z_2$ or
$S^1 \times_{\Z_2} (S^2//\Gamma)$, the latter of which is
a Seifert $3$-orbifold.
Hence we can assume that each connected component of
$\Or^{0-stratum}_{Sph} \cup \Or^{slim}_{Sph}$ intersects $\Or^{edge}$.
A component of $\Or^{slim}_{Sph}$ which intersects
$\Or^{0-stratum}_{Sph}$ can now only do so on one side, so we can
collapse such a component of $\Or^{slim}_{Sph}$ without changing the
diffeomorphism type. Thus we can assume that
each connected component of $\Or^{0-stratum}_{Sph}$ and
each connected component of 
$\Or^{slim}_{Sph}$ intersects $\Or^{edge}$, and that
$\Or^{0-stratum}_{Sph} \cap \Or^{slim}_{Sph} = \emptyset$.
By Lemmas \ref{lemma9.4} and \ref{lemma9.5},
 each of their boundary components is one of
$S^2$, $S^2(k,k)$ and $S^2(2,2,k)$.

Consider the connected components of
$\Or^{0-stratum}_{Euc} \cup \Or^{slim}_{Euc}$ 
whose boundary components are
$S^2(2,3,6)$, $S^2(2,4,4)$ or $S^2(3,3,3)$. 
They cannot intersect any other strata, so if there is one such
connected component then $\Or$ is formed entirely of such
components.  In this case $\Or$ is diffeomorphic to the result
of gluing together two copies of $I \times_{\Z_2} (T^2//\Gamma)$.
Hence $\Or$ fibers over $S^1//\Z_2$ and has a geometric
structure based on $\R^3$, $\Nil$ or $\Sol$
\cite[p. 72]{Dunbar}. If the structure is
$\Nil$ then $\Or$ is a Seifert orbifold \cite[Theorem 1]{Dunbar}.
If the structure is $\Sol$ then we can cut $\Or$ along a generic
fiber to see that it is a weak graph orbifold.
Hence we can assume that
there are no connected components of
$\Or^{0-stratum}_{Euc} \cup \Or^{slim}_{Euc}$ 
whose boundary components are
$S^2(2,3,6)$, $S^2(2,4,4)$ or $S^2(3,3,3)$. 
Next, consider the connected components of
$\Or^{0-stratum}_{Euc} \cup \Or^{slim}_{Euc}$ with $T^2$-boundary components.
They are weak graph orbifolds that do not intersect any strata
other than $\Or^{2-stratum}$. If $X_1$ is their complement
in $\Or$ then in order to show that $\Or$ is a weak graph orbifold, it
suffices to show that $X_1$ is a weak graph orbifold. Hence we can assume
that each connected component of 
$\Or^{0-stratum}_{Euc} \cup \Or^{slim}_{Euc}$ 
has $S^2(2,2,2,2)$-boundary components, in which case it necessarily
intersects $\Or^{edge}$. 
As above, after collapsing some components of
$\Or^{slim}_{Euc}$, we can assume that
each connected component of $\Or^{0-stratum}_{Euc}$ and
each connected component of 
$\Or^{slim}_{Euc}$ intersects $\Or^{edge}$, and that
$\Or^{0-stratum}_{Euc} \cap \Or^{slim}_{Euc} = \emptyset$.

A connected component of $\Or^{slim}_{Sph}$ is now diffeomorphic to
$I \times \Or^\prime$, where $\Or^\prime$ is diffeomorphic to 
$S^2$, $S^2(k,k)$ or $S^2(2,k,k)$. We cut each such component along
$\{\frac12\} \times \Or^\prime$ and glue on two discal caps.  
If
$X_2$ is the ensuing orbifold then $X_1$ is the result of performing
a $0$-surgery on $X_2$, so it suffices to prove that $X_2$ satisfies
the conclusion of the proposition.   Therefore we assume
henceforth that $\Or^{slim}_{Sph} = \emptyset$.

A remaining connected component of $\Or^{slim}_{Euc}$ is diffeomorphic to
$I \times \Or^\prime$, where $\Or^\prime = S^2(2,2,2,2)$.
It intersects $\Or^{edge}$ in four copies of $D^2(2,2)$.
We cut the connected component of $\Or^{slim}_{Euc}$ along
$\{\frac12\} \times \Or^\prime$. 
The result is two copies of $I \times \Or^\prime$, each with
one free boundary component and another boundary component which
intersects $\Or^{edge}$ in two copies of $D^2(2,2)$.
If the
result $X_3$ of all such cuttings satisfies the conclusion of the proposition 
then so does $X_2$, it being the result of gluing Euclidean boundary
components of $X_3$ together.

A connected component $C$ of $\Or^{edge}$ fibers over $I$ or $S^1$.
Suppose that it fibers over $S^1$. Then it is diffeomorphic to
$S^1 \times D^2(k)$ or $S^1 \times D^2(2,2)$, or else is the
total space of a bundle over $S^1$ with holonomy that interchanges
the two singular points in a fiber $D^2(2,2)$; this is because
the mapping class group of $D^2(2,2)$ is a copy of $\Z_2$,
as follows from
\cite[Proposition 2.3]{Farb-Margalit}. 
If $C$
is diffeomorphic to $S^1 \times D^2(k)$ or $S^1 \times D^2(2,2)$
then it is clearly a weak graph orbifold.  In the third case,
$|C|$ is a solid torus and the singular locus consists of a circle
labelled by $2$ that wraps twice around the solid torus.
See Figure \ref{fig-10}. We can
decompose $C$ as 
$C = (S^1 \times_{\Z_2} D^2) \cup_{S^2(2,2,2,2)}
C_1$, 
where $C_1 = S^1 \times_{\Z_2} (S^2 - 3 B^2)$ with one $B^2$ being sent
to itself by the $\Z_2$-action and the other two $B^2$'s being
switched.
See Figure \ref{fig-11}.
As $C_1$ is a Seifert orbifold, in any case
$C$ is a weak graph orbifold.
Put $X_4 = X_3 - \Int(C)$. If we can
show that $X_4$ is a weak graph orbifold then it follows that $X_3$ is
a weak graph orbifold.  Hence we can assume that each connected
component of $\Or^{edge}$ fibers over $I$.

\begin{figure}[h] 
\begin{center}  
\includegraphics[scale=1.2]{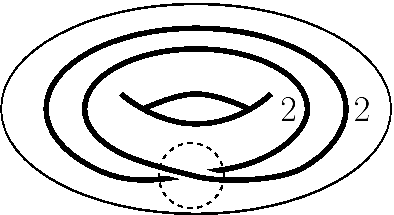} 
\caption{\label{fig-10}}
\end{center}
\end{figure}

\begin{figure}[h] 
\begin{center}  
\includegraphics[scale=1]{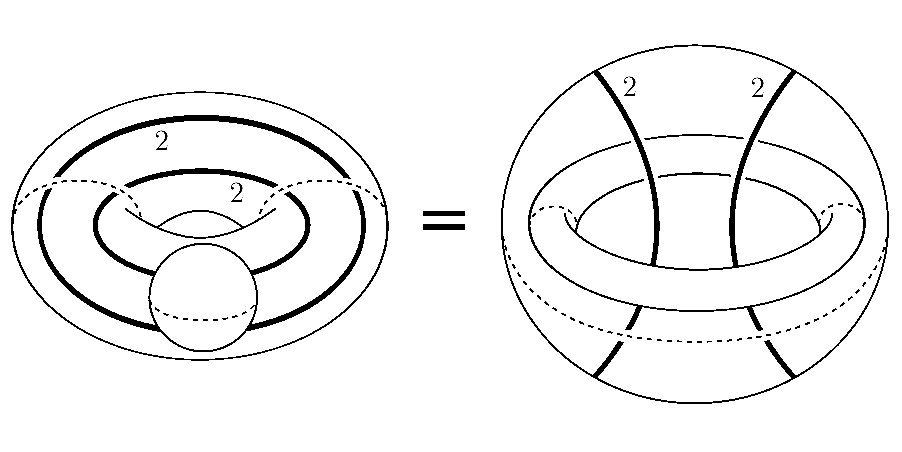} 
\caption{$C_1$\label{fig-11}}
\end{center}
\end{figure}

\begin{figure}
\begin{center}  
\includegraphics[scale=1.3]{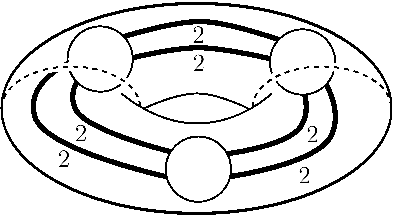} 
\caption{$C_m$, $m=3$\label{fig-12}}
\end{center}
\end{figure}

A connected component $Z$ of $X_4 - \Int(\Or^{2-stratum})$ can be described
by a graph, i.e. a one-dimensional CW-complex, of degree $2$.  
Its vertices correspond to copies of
\begin{itemize}
\item A connected component of $\Or^{0-stratum}_{Sph}$ with
boundary $S^2$ or $S^2(k,k)$,
\item A connected component of $\Or^{0-stratum}_{Euc}$ with
boundary $S^2(2,2,2,2)$, or
\item $I \times S^2(2,2,2,2)$.
\end{itemize}
Each edge corresponds to a copy of
\begin{itemize}
\item $I \times D^2$,
\item $I \times D^2(k)$ or
\item $I \times D^2(2,2)$.
\end{itemize}
If a vertex is of type $I \times S^2(2,2,2,2)$ then the edge orbifolds
only intersect the vertex orbifold
on a single one of its two boundary components.
Note that $|Z|$ is a solid torus with a certain number of balls
removed.

A connected component of $\Or^{0-stratum}_{Sph}$ 
is diffeomorphic to $D^3$, $D^3(k,k)$, $D^3(2,2,k)$,
$I \times_{\Z_2} S^2$, or $I \times_{\Z_2} S^2(2,2,k)$.
Now $I \times_{\Z_2} S^2$ is diffeomorphic to $\R P^3 \# D^3$,
$I \times_{\Z_2} S^2(k,k)$ is diffeomorphic to 
$(S^3(k,k)//\Z_2) \#_{S^2(k,k)} D^3(k,k)$  and
$I \times_{\Z_2} S^2(2,2,k)$ is diffeomorphic to
$(S^3(2,2,k)//\Z_2) \#_{S^2(2,2,k)} D^3(2,2,k)$, where $\Z_2$ acts
by the antipodal action. Hence we can reduce to the case when
each connected component of $\Or^{0-stratum}_{Sph}$ 
is diffeomorphic to $D^3$, $D^3(k,k)$ or $D^3(2,2,k)$, modulo
performing connected sums with the Seifert orbifolds
$\R P^3$, $S^3(k,k)//\Z_2$ and $S^3(2,2,k)//\Z_2$. 

Any connected component of $\Or^{0-stratum}_{Euc}$ with boundary
$S^2(2,2,2,2)$ can be written as the gluing of a weak graph orbifold
with $I \times S^2(2,2,2,2)$. Hence we may assume that there are no
vertices corresponding to 
connected components of $\Or^{0-stratum}_{Euc}$ with boundary
$S^2(2,2,2,2)$.

Suppose that there are no edges of type $I \times D^2(2,2)$.
Then $Z$ is $I \times D^2$ or $I \times D^2(k)$, which is a weak
graph orbifold.

Now suppose that there is an edge of type $I \times D^2(2,2)$.
We build up a skeleton for $Z$. First, the
orbifold corresponding to a graph with a single vertex of
type $I \times S^2(2,2,2,2)$, and a single edge of type
$I \times D^2(2,2)$, can be identified as the Seifert
orbifold $C_1 = S^1 \times_{\Z_2}
(S^2 - 3 B^2)$ of before. Let $C_m$ be the orbifold
corresponding to a graph with $m$ vertices of type
$I \times S^2(2,2,2,2)$ and $m$ edges of type
$I \times D^2(2,2)$.
See Figure \ref{fig-12}.
Then $C_m$ is an $m$-fold cover of $C_1$
and is also a Seifert orbifold.

Returning to the orbifold $Z$, there is some $m$ so that
$Z$ is diffeomorphic to the result of starting with $C_m$
and gluing some $S^1 \times_{\Z_2} D^2(k_i)$'s onto some
of the boundary $S^2(2,2,2,2)$'s, where $k_i \ge 1$. 
See Figure \ref{fig-13} for an illustrated example.

\begin{figure}[htb] 

\begin{center}  
\includegraphics[scale=1.3]{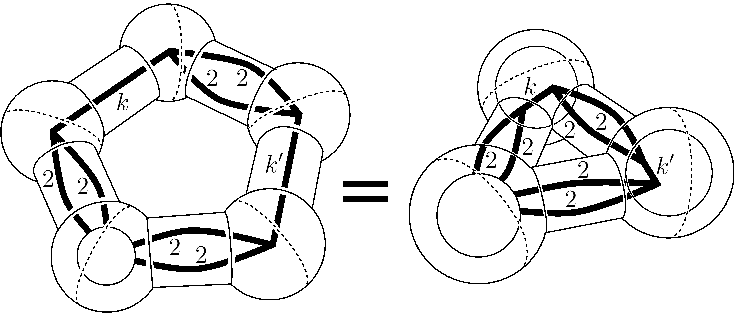} 
\caption{\label{fig-13}}
\end{center}
\end{figure}

Thus $Z$ is a weak
graph orbifold.

As $X_3$ is the result of gluing $Z$ to a circle bundle over a
surface, $X_3$ is a weak graph orbifold.
Along with Proposition \ref{propA.25}, this proves the proposition.
\end{proof}

\begin{proposition} \label{prop9.6}
Let $c_3$ be the volume of the unit ball in $\R^3$, let $K \ge 10$ be
a fixed integer and let $N$ be a positive integer. Fix a function
$A : (0, \infty) \rightarrow (0, \infty)$. Then there is a 
$w_0 \in (0, c_3/N)$ such that the following holds.

Suppose that $(\Or, g)$ is a 
compact connected orientable Riemannian $3$-orbifold with boundary.
Assume in addition that 
\begin{enumerate} 
\item $|G_p| \le N$.
\item The diameters of the connected components of $\partial \Or$ are 
bounded above by $w_0$.
\item For each component $X$ of $\partial \Or$, there is a hyperbolic
orbifold cusp ${\mathcal H}_X$ with boundary $\partial {\mathcal H}_X$,
along with a $C^{K+1}$-embedding of pairs
$e \: : \: (N_{100}(\partial {\mathcal H}_X), \partial {\mathcal H}_X) 
\rightarrow (\Or, X)$ which is $w_0$-close to an isometry.
\item For every $p \in |\Or|$ with 
$d(p, \partial \Or) \ge 10$, we have, $\vol(B(p, R_p)) \le w_0 R_p^3$.
\item For every $p \in |\Or|$, 
$w^\prime \in [w_0, c_3/N)$, $k \in [0, K]$ and $r \le R_p$
such that $\vol(B(p,r)) \ge w^\prime r^3$, the inequality
\begin{equation}
|\nabla^k \Rm| \le A(w^\prime) r^{-(k+2)}
\end{equation}
holds in the ball $B(p,r)$.
\end{enumerate} 

Then $\Or$ is diffeomorphic to
\begin{itemize}
\item  The result of performing $0$-surgeries on a strong graph orbifold,
\item A closed isometric quotient of $S^3$ or $T^3$,
\item $I \times S^2(2,3,6)$, $I \times S^2(2,4,4)$ or 
$I \times S^2(3,3,3)$, or
\item $I \times_{\Z_2} S^2(2,3,6)$, $I \times_{\Z_2} S^2(2,4,4)$ or 
$I \times_{\Z_2} S^2(3,3,3)$.
\end{itemize}
\end{proposition}
\begin{proof}
We follow the method of proof of \cite[Section 16]{Kleiner-Lott}.
The effective difference from the proof of Proposition \ref{prop9.1} is that we
have additional components of $\Or^{0-stratum}$, which are
diffeomorphic to $I \times (T^2//\Gamma)$. If such a component
is diffeomorphic to $I \times T^2$ or $I \times S^2(2,2,2,2)$ then
we can incorporate it into the weak graph orbifold structure.
The other cases give rise to the additional possibilities listed
in the conclusion of the proposition.
\end{proof}

\section{Incompressibility of cuspidal cross-sections and
proof of Theorem \ref{theorem1.1}} \label{sect10}

In this section we complete the proof of Theorem \ref{theorem1.1}.

With reference to Proposition \ref{prop8.11}, 
given a sequence $t^\alpha \rightarrow
\infty$, let $Y^\alpha$ be the truncation of $\coprod_{i=1}^N H_i$
obtained by removing horoballs at distance approximately
$\frac{1}{2\beta(t^\alpha)}$ from the basepoints $x_i$. Put
$\Or^\alpha = \Or_{t^\alpha} - f_{t^\alpha}(Y^\alpha)$.

\begin{proposition} \label{prop10.1}
For large $\alpha$, the orbifold $\Or^\alpha$ satisfies the hypotheses of
Proposition \ref{prop9.6}.
\end{proposition}
\begin{proof}
The proof is similar to that of \cite[Theorem 17.3]{Kleiner-Lott2}.
\end{proof}

So far we know that if $\alpha$ is large then the $3$-orbifold 
$\Or_{t^\alpha}$ has
a (possibly empty) hyperbolic piece whose complement satisfies the
conclusion of Proposition \ref{prop9.6}.  In this section we
show that there is such a decomposition of $\Or_{t^\alpha}$
so that the hyperbolic cusps, 
if any, are incompressible in $\Or_{t^\alpha}$.

The corresponding manifold result was proved by Hamilton in
\cite{Hamilton (1999)} using minimal disks. He used results of
Meeks-Yau \cite{Meeks-Yau (1982a)} 
to find embedded minimal disks with boundary on an
appropriate cross-section of the cusp.  The Meeks-Yau proof
in turn used a tower construction 
\cite{Meeks-Yau (1982b)}
similar to that used in the
proof of Dehn's Lemma in $3$-manifold topology. It is not clear
to us whether this line of proof extends to three-dimensional orbifolds,
or whether there are other methods using minimal disks which do 
extend.  To circumvent these issues, we use an alternative 
incompressibility argument due to
Perelman \cite[Section 8.2]{Perelman2} 
that exploits certain quantities which change monotonically 
under the Ricci flow. Perelman's monotonic quantity involved the smallest
eigenvalue of a certain Schr\"odinger-type operator. We will instead
use a variation of Perelman's argument involving the minimal
scalar curvature, following \cite[Section 93.4]{Kleiner-Lott}.

Before proceeding, we need two lemmas:

\begin{lemma}
\label{lem-freeconnectedsum}
Suppose $\eps>0$, and $\Or'$ is a Riemannian $3$-orbifold with scalar curvature
$\geq -\frac32$.   Then any orbifold $\Or$ obtained from $\Or'$ by $0$-surgeries
admits a  Riemannian metric with scalar
curvature $\geq -\frac32$,  such that 
$\vol(\Or)<\vol(\Or')+\eps$.

\end{lemma}
\begin{proof}
If a $0$-surgery adds a neck $(S^2//\Gamma) \times I$ then we can
put a metric on the neck which is an isometric quotient of a slight
perturbation of the doubled Schwarzschild metric 
\cite[(1.23)]{Anderson} on $S^2 \times I$.
Hence we can perform the $0$-surgery so that the scalar curvature
is bounded below by $- \frac32 + \frac{\epsilon}{10}$ and the volume
increases by at most $\frac{\epsilon}{10}$; see
\cite[p. 155]{Anderson} and \cite{PeteanYun} for the analogous 
result in the manifold case.
The lemma now follows from an overall rescaling
to make $R \: \ge \: - \: \frac32$.
\end{proof}

\begin{lemma}
\label{lem-collapsingstronggraphmanifolds}
Suppose that $\Or$ is a strong graph orbifold with boundary components $C_1,\ldots,C_k$.
Let $H_1,\ldots,H_k$ be 
truncated
hyperbolic cusps, where $\D H_i$ is diffeomorphic to $C_i$ for all 
$i\in \{1,\ldots,k\}$.  Then for all $\eps>0$, there is a metric on $\Or$ with scalar 
curvature $\geq -\frac32$ such that $\vol(\Or)<\eps$, and $C_i$ has a collar which is
isometric to 
one side of a
collar neighborhood of a cuspical $2$-orbifold in $H_i$.
\end{lemma}
\begin{proof}
We first prove the case when $\Or$ is a closed strong graph manifold.
The strong graph manifold structure gives a graph whose vertices
$\{v_a\}$ correspond to the Seifert blocks and whose edges 
$\{e_b\}$ correspond to $2$-tori. For each vertex $v_a$, let 
$M_a$ be the corresponding Seifert block.  We give it a Riemannian
metric $g_a$ which is invariant under the local $S^1$-actions and with
the property that the quotient metric on the orbifold base is
a product near its boundary.  Then $g_a$ has a product structure
near $\partial M_a$. 
Given $\delta > 0$, we uniformly shrink the Riemannian metric on
$g_a$ by $\delta$ in the fiber directions. As $\delta \rightarrow 0$,
the volume of $M_a$ goes to zero while the curvature stays bounded.

Let $T^2_b$ be the torus corresponding to the edge $e_b$. There are
associated toral boundary components $\{B_1, B_2\}$ of Seifert blocks.
Given $\delta > 0$ and $i \in \{1,2\}$, consider the warped product metric
$ds^2 + e^{-2s} g_{B_i}$ on a product manifold $P_{\delta, i} =
[0, L_{\delta, i}] \times B_i$. We attach this at $B_i$ to obtain
a $C^0$-metric, which we will smooth later.  The sectional curvatures of 
$P_{\delta, i}$ are
$-1$ and the volume of $P_{\delta, i}$ is bounded above by the area of $B_i$.
We choose $L_{\delta, i}$ so that the areas of the cross-sections
$\{ L_{\delta, 1} \} \times B_1$ and
$\{ L_{\delta, 2} \} \times B_2$ are both equal to some number 
$A$. Finally, consider
$\R^3$ with the $\Sol$-invariant metric
$e^{-2z} dx^2 + e^{2z} dy^2 + dz^2$. Let $\Gamma$ be a $\Z^2$-subgroup
of the normal $\R^2$-subgroup of $\Sol$.
Note that the curvature of $\R^3/\Gamma$ is independent of
$\Gamma$. The $z$-coordinate
gives a fibering $z \: : \: \R^3/\Gamma \rightarrow \R$
with $T^2$-fibers.  We can choose $\Gamma = \Gamma_\delta$ and an interval 
$[c_1,c_2] \subset \R$ so that $z^{-1}(c_1)$ is isometric to 
$\{ L_{\delta, 1} \} \times B_1$ and $z^{-1}(c_2)$ is isometric to 
$\{ L_{\delta, 2} \} \times B_2$.
Note that $[c_1, c_2]$ can be taken independent of $A$.
We attach $z^{-1}([c_1, c_2])$ to the previously described truncated
cusps, at
the boundary components $\{ L_{\delta, 1} \} \times B_1$ and
$\{ L_{\delta, 2} \} \times B_2$.
See Figure \ref{fig-14}.
\begin{figure}[h] 
\begin{center}  
\includegraphics[scale=1]{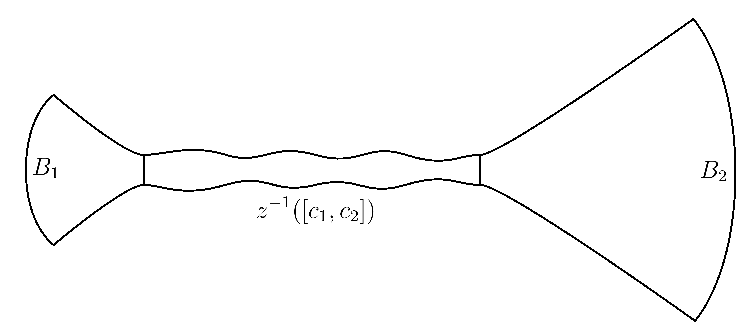} 
\caption{\label{fig-14}}
\end{center}
\end{figure}

Taking $A$ sufficiently small we can
ensure that 
\begin{equation}
\vol(P_{\delta,1}) + \vol(P_{\delta,2}) + 
\vol(z^{-1}([c_1, c_2])) < \area(B_1) + \area(B_2) + \delta.
\end{equation}

We repeat this process for
all of the tori $\{T^2_b\}$, to obtain a piecewise-smooth $C^0$-metric
$g_\delta$ on $\Or$.

As $\delta \rightarrow 0$, the sectional curvature stays
uniformly bounded on the smooth pieces. Furthermore, the volume
of $(\Or, g_\delta)$
goes to zero.  By slightly smoothing $g_\delta$ and performing
an overall rescaling to ensure that the scalar curvature is
bounded below by $- \: \frac32$, if $\delta$ is sufficiently small then
we can ensure that $\vol(\Or, g_\delta) < \epsilon$.
This proves the lemma when $\Or$ is a closed strong graph manifold.

If $\Or$ is a strong graph manifold but has nonempty boundary
components, as in the hypotheses of the lemma, then we treat
each boundary component $C_i$ analogously to a factor $B_1$
in the preceding construction. That is, given
parameters $0 < c_{1,C_i}< c_{2,C_i}$, we start
by putting a truncated hyperbolic metric
$ds^2 + e^{-2s} g_{\partial H_i}$ on 
$\left[ c_{1,C_i}, c_{2,C_i} \right] \times C_i$.
This will be the metric on the collar neighborhood of
$C_i$, where $\{ c_{1,C_i} \} \times C_i$ will end
up becoming a boundary component of $\Or$.
We take $c_{2,C_i}$ so that the area of 
$\{c_{2,C_i}\} \times C_i$ matches the area of a
relevant cross-section of the truncated cusp extending from
a boundary component $B_{2,i}$ of a Seifert block. We then
construct a metric $g_{\delta}$ on $\Or$ as before. 
If we additionally take the parameters 
$\{ c_{1,C_i} \}$ sufficiently large then
we can ensure that $\vol(\Or, g_\delta) < \epsilon$.

Finally, if $\Or$ is a strong graph orbifold then we can
go through the same steps.  The only additional point is
to show that elements of the (orientation-preserving) mapping class group
of an oriented Euclidean $2$-orbifold $T^2//\Gamma$ are represented
by affine diffeomorphisms, in order to apply the preceding construction
using the $\Sol$ geometry. To see this fact,
if $\Gamma$ is trivial then the 
mapping class group of
$T^2$ is isomorphic to $\SL(2, \Z)$ and the claim is clear.
To handle the case when $T^2//\Gamma$ is a sphere with three
singular points, we use the fact that the mapping class group
of a sphere with three marked points is isomorphic to
the permutation group of the three points
\cite[Proposition 2.3]{Farb-Margalit}. 
The mapping class group of the orbifold $T^2//\Gamma$ will then
be the subgroup of the permutation group that preserves the labels.
If $T^2//\Gamma$ is $S^2(2,3,6)$ then its mapping class group
is trivial. If $T^2//\Gamma$ is $S^2(2,4,4)$ then its mapping class group
is isomorphic to $\Z_2$. Picturing $S^2(2,4,4)$ as
two right triangles glued together, the nontrivial mapping class
group element is represented by the affine diffeomorphism which
is a flip around the ``$2$'' vertex that interchanges the two
triangles.
If $T^2//\Gamma$ is $S^2(3,3,3)$
then its mapping class group is isomorphic to $S_3$.
Picturing $S^2(3,3,3)$ as two equilateral triangles glued together,
the nontrivial mapping class group elements are represented by
affine diffeomorphisms as rotations and flips.  Finally, if
$T^2//\Gamma$ is $S^2(2,2,2,2)$ then its mapping class group is
isomorphic to $\PSL(2, \Z) \ltimes (\Z/2 \Z \times \Z/2 \Z)$
\cite[Proposition 2.7]{Farb-Margalit}.
These all lift to $\Z_2$-equivariant affine diffeomorphisms
of $T^2$. Elements of $\PSL(2, \Z)$ are represented by linear
actions of $\SL(2, \Z)$ on $T^2$. Generators of $\Z/2 \Z \times \Z/2 \Z$
are represented by rotations of the $S^1$-factors in 
$T^2 = S^1 \times S^1$ by $\pi$.
\end{proof}

Let $\Or$ be a closed connected orientable three-dimensional orbifold.
If $\Or$ admits a metric of positive scalar curvature then by finite
extinction time, $\Or$ is diffeomorphic to 
the result of performing $0$-surgeries on a disjoint collection of
isometric quotients of $S^3$ and $S^1 \times S^2$.

Suppose that $\Or$ does not admit a metric of positive scalar curvature.
Put 
\begin{equation}
\sigma(\Or) = \sup_g R_{\min}(g) V(g)^{\frac23}.
\end{equation}
Then $\sigma(\Or) \le 0$.

Suppose that we have a
given representation of $\Or$ as the result of performing $0$-surgeries on
the disjoint union of 
an orbifold $\Or^\prime$ and isometric quotients of $S^3$ and 
$S^1 \times S^2$,
 and that there
exists a  (possibly empty, possibly disconnected) 
finite-volume complete hyperbolic orbifold $N$
which can be embedded in $\Or^\prime$ so that the connected components of
the complement (if nonempty)
satisfy the conclusion of Proposition \ref{prop9.6}. 
Let $V_{hyp}$ denote the hyperbolic volume of $N$.
We do not assume that the cusps of $N$ are incompressible in $\Or^\prime$.

Let $\widehat{V}$ denote the minimum of $V_{hyp}$ over all such
decompositions of $\Or$. (As the set of volumes of complete finite-volume
three-dimensional hyperbolic orbifolds is well-ordered, there is a minimum.
If there is a decomposition with $N = \emptyset$ then $V_{hyp} = 0$.)

\begin{lemma} \label{lemma10.3}
\begin{equation}
\sigma(\Or) = - \frac32 \widehat{V}^{\frac23}.
\end{equation}
\end{lemma}
\begin{proof}
Using Lemmas
\ref{lemma7.26},
\ref{lem-freeconnectedsum} and 
\ref{lem-collapsingstronggraphmanifolds}, the proof is similar to that of 
\cite[Proposition 93.10]{Kleiner-Lott}.
\end{proof}

\begin{proposition} \label{prop10.5}
Let $N$ be a hyperbolic orbifold as above for which $\vol(N) = \widehat{V}$.
Then the cuspidal cross-sections of $N$ are incompressible in $\Or^\prime$. 
\end{proposition}

\begin{proof}
As in \cite[Section 93]{Kleiner-Lott}, it suffices to show that if a
cuspidal cross-section of $N$ is compressible in $\Or^\prime$ then there
is a metric $g$ on $\Or$ with $R(g) \ge - \frac32$ and $\vol(\Or, g) <
\vol(N)$.

Put $Y=\Or'-N$. 
Suppose that some connected component $C_0$ of $\D Y$ is compressible,
with compressing discal $2$-orbifold $Z\subset\Or'$.  
We can make $Z$
transverse to $\D Y$ and then count the number of connected components of 
the intersection
$Z\cap \D Y$. Minimizing this number among all such compressing disks for
all compressible components of $\D Y$, we may assume
 -- after possibly replacing  
$C_0$
with a different component
of $\D Y$ -- that $Z$ intersects $\D Y$ only along $\D Z$.

By assumption, the components of $Y$ satisfy the conclusion of Proposition
\ref{prop9.6}. Hence $Y$ has a decomposition into connected
components $Y=Y_0\sqcup \ldots\sqcup  Y_n$, where $Y_0$  is the component
containing $C_0$, and $Y_0$
arises from a 
strong 
graph orbifold by $0$-surgeries, as otherwise there would not be a compressing
discal orbifold.   By Lemma \ref{lem-weakgraphcompressible}, 
$Y_0$ comes from a disjoint union $A\sqcup B$ via $0$-surgeries,
where $A$ is one of the four
solid-toric
possibilities of that Lemma, and $B$ is a
strong graph orbifold.    By Lemmas \ref{lem-freeconnectedsum}
and \ref{lem-collapsingstronggraphmanifolds}, we may assume
without loss of generality that $B=\emptyset$.  

To construct the desired metric on $\Or'$, we proceed as follows.
Let $H_0,\ldots, H_n$
be the cusps of the hyperbolic orbifold $N$, where $H_0$ corresponds to 
the component
$C_0$ of $Y$.   We first truncate $N$ along totally umbilic cuspical $2$-orbifolds
$C_0,\ldots,C_n$.   Pick $\eps>0$. For each $i\geq 1$ such that the component $Y_i$
comes from $0$-surgeries on a 
strong 
graph orbifold, we use Lemmas
\ref{lem-freeconnectedsum} and \ref{lem-collapsingstronggraphmanifolds} to find a  
metric 
with $R\geq -\frac32$ on $Y_i$, which 
glues isometrically along the corresponding cusps in 
$C_1\sqcup\ldots\sqcup C_n$, and which can be arranged to
have  volume $<\eps$ by taking the $C_i$'s to be deep in their respective cusps.
For the components $Y_i$, $i\geq 1$, which do not come from 
a 
strong
graph orbifold via $0$-surgery,  we may also find
metrics with $R\geq -\frac32$ and arbitrarily small volume, which glue isometrically onto
the corresponding truncated cusps of $N$ (when they have nonempty boundary).  
Our final
step   will be to find a metric on $Y_0=A$ with $R\geq -\frac32$ which glues isometrically to $C_0$,
and has volume strictly smaller than the portion of the cusp $H_0$ cut off by $C_0$.  Since 
$\eps$ is arbitrary, this will yield a contradiction.

Suppose first  that $A$ is 
$S^1 \times D^2$ or $S^1 \times D^2(k)$.
In the $S^1 \times D^2$ case, after going
far enough down the cusp, the desired metric $g$ on $S^1 \times D^2$ is
constructed in \cite[Pf. of Theorem 2.9]{Anderson}. 
(The condition $f_2(0) = a > 0$ in \cite[(2.47)]{Anderson} should be
changed to $f_2(0) > 0$.) In the $S^1 \times D^2(k)$-case,
\cite[(2.46)]{Anderson} gets changed to
$f_1^\prime(0) (1-a^2)^{1/2} = 1/k$. One can then make the
appropriate modifications to \cite[(2.54)-(2.56)]{Anderson} to construct
the desired metric $g$ on $S^1 \times D^2(k)$.

If $A$ is
$S^1 \times_{\Z_2} D^2$ or
$S^1 \times_{\Z_2} D^2(k)$  we can perform the construction of the previous paragraph
equivariantly with respect to the $\Z_2$-action, to form the desired
metric on $S^1 \times_{\Z_2} D^2$ (or
$S^1 \times_{\Z_2} D^2(k)$).
\end{proof}
\noindent
{\bf Proof of Theorem \ref{theorem1.1} : }
As mentioned before, if $\Or$ admits a metric of positive scalar curvature
then $\Or$ is diffeomorphic to 
the result of performing $0$-surgeries on a disjoint collection of
isometric quotients of $S^3$ and $S^1 \times S^2$, so the theorem is
true in that case. If $\Or$ does not admit a metric of positive scalar
curvature then by Proposition \ref{prop10.5}, 
\begin{enumerate}
\item $\Or$ is the result of performing
$0$-surgeries on an orbifold $\Or^\prime$ and a disjoint collection of
isometric quotients of $S^3$ and $S^1 \times S^2$, such that 
\item There is a finite-volume complete hyperbolic orbifold $N$ which
can be embedded in $\Or^\prime$ so that each connected component 
${\mathcal P}$ of
the complement (if nonempty) has a metric completion $\overline{\mathcal P}$
which satisfies the conclusion of Proposition
\ref{prop9.6}, and
\item The cuspidal cross-sections of $N$ are incompressible in $\Or^\prime$.
\end{enumerate}

Referring to Proposition \ref{prop9.6}, if $\overline{\mathcal P}$
is an isometric quotient of $S^3$ or $T^3$ then it already has a geometric
structure. If $\overline{\mathcal P}$
is $I \times S^2(p,q,r)$ with $\frac{1}{p} + \frac{1}{q} + \frac{1}{r} = 1$
then we can remove it without losing any information.
If $\overline{\mathcal P}$
is $I \times_{\Z_2} S^2(p,q,r)$ 
with $\frac{1}{p} + \frac{1}{q} + \frac{1}{r} = 1$ then 
${\mathcal P}$ has a Euclidean structure.

Finally, suppose that $\overline{\mathcal P}$ is the result of
performing $0$-surgeries on a collection of
strong graph orbifolds in the sense of
Definition \ref{defnA.22}. 
A Seifert-fibered $3$-orbifold 
with no bad $2$-dimensional suborbifolds is geometric
in the sense of Thurston \cite[Proposition 2.13]{BMP}.
This completes the proof of Theorem \ref{theorem1.1}.

\begin{remark}
The geometric decomposition of $\Or$ that we have produced, using
strong graph orbifolds, will not
be minimal if $\Or$ has $\Sol$ geometry. In such a case,
$\Or$ fibers over a $1$-dimensional orbifold. Cutting
along a fiber and taking the metric completion gives a product orbifold,
which is a graph orbifold. Of course, the
minimal geometric decomposition of $\Or$ would leave it with its $\Sol$
structure.
\end{remark}

\begin{remark}
Theorem \ref{theorem1.1} implies that $\Or$ is very good,
i.e. the quotient of a manifold by a finite group action
\cite[Corollary 1.3]{BLP}. Hence one could obtain the
geometric decomposition of $\Or$ by running Perelman's
proof equivariantly, as is done in detail for elliptic
and hyperbolic manifolds in \cite{Dinkelbach-Leeb}.
However, one cannot prove the geometrization
of orbifolds this way, as the reasoning would be circular;
one only knows that $\Or$ is very good after proving 
Theorem \ref{theorem1.1}. 
\end{remark}
\appendix

\section{Weak and strong graph orbifolds} \label{appendix}

In this appendix we provide proofs of some needed facts about
graph orbifolds.  We show that a weak graph orbifold
is the result of performing $0$-surgeries on a 
strong graph orbifold.
(Since we don't require strong graph orbifolds to be connected, 
we need only one.) 
A similar result appears in
\cite[Section 2.4]{Faessler}.

In order to clarify the arguments, we prove the corresponding
manifold results before proving the orbifold results.

\begin{definition} \label{defnA.1} 
A {\em weak graph manifold} is a compact orientable $3$-manifold $M$
for which there is a collection $\{T_i\}$ of disjoint embedded tori in
$\Int(M)$ so that after splitting $M$ open along $\{T_i\}$,
the result has connected components that
are Seifert-fibered $3$-manifolds (possibly with boundary).
\end{definition}

We do not assume that $M$ is connected.
Here ``splitting $M$ open along $\{T_i\}$'' means taking the metric
completion of $M - \bigcup_i T_i$ 
with respect to an arbitrary Riemannian metric on $M$.

\begin{remark} \label{remarkA.2}
In the definition of a weak graph manifold, we could have instead
required that the connected components of the metric
completion of $M - \bigcup_i T_i$ are
circle bundles over surfaces. This would give an equivalent notion,
since any Seifert-fibered $3$-manifold can be cut along tori into
circle bundles over surfaces.
\end{remark}

For notation, we will write $S^2 - k B^2$ for the complement of
$k$ disjoint separated open $2$-balls in $S^2$.

\begin{definition} \label{defnA.19}
A {\em strong graph manifold} is a compact orientable $3$-manifold $M$
for which there is a collection $\{T_i\}$ of disjoint embedded tori in
$\Int(M)$ such that
\begin{enumerate}
\item After splitting $M$ open 
along $\{T_i\}$, the result has connected components
that are Seifert manifolds (possibly with boundary).
\item For any $T_i$, the two circle fibrations on $T_i$ coming from the
adjacent Seifert bundles are not isotopic.
\item Each $T_i$ is incompressible in $M$.
\end{enumerate}
\end{definition}

\subsection{Weak graph manifolds are connected sums of strong graph manifolds}
\label{subsectA.2}

The next lemma states if we glue two solid tori (respecting orientations)
then the result is a
Seifert manifold.  The lemma itself is trivial, since we know that
the manifold is $S^1 \times S^2$, $S^3$ or a lens space, each of which is a
Seifert manifold.  However, we give a proof of the lemma
which will be useful in the orbifold case.

\begin{lemma} \label{Seifertlemma}
Let $U$ and $V$ be two oriented solid tori.  Let $\phi : \partial U
\rightarrow \partial V$ be an orientation-reversing diffeomorphism.
Then $U \cup_{\phi} V$ admits a Seifert fibration.
\end{lemma}
\begin{proof}
We first note that
the circle fiberings of $T^2$ are classified (up to isotopy) by the
image of the fiber in $(\HH^1(T^2; \Z) - \{0\})/\{\pm 1\} \simeq 
(\Z^2 - \{ 0 \})/\{\pm 1\}$.
There is one circle fibering of $\partial U$ 
(up to isotopy)
whose fibers bound
compressing disks in $U$.  Any other circle fibering of $\partial U$ is
the boundary fibration of a Seifert fibration of $U$.
Hence we can choose a circle fibering ${\mathcal F}$ of $\partial U$
so that ${\mathcal F}$ is the
boundary fibration of a Seifert fibration of $U$, and 
$\phi_* {\mathcal F}$ is the boundary fibration of a Seifert fibration of $V$.
The ensuing Seifert fibrations of $U$ and $V$ join together
to give a Seifert fibration of $U \cup_{\phi} V$.
\end{proof}

\begin{proposition} \label{propA.20}
If a connected strong graph manifold contains an essential
embedded $2$-sphere then it is diffeomorphic
to $S^1 \times S^2$ or $\R P^3 \# \R P^3$.
\end{proposition}
\begin{proof}
Suppose that a connected strong graph manifold $M$ contains
an essential embedded $2$-sphere $S$. We can assume that
$S$ is transverse to $\bigcup_i T_i$. We choose $S$ among all such
essential embedded $2$-spheres so that the
number of connected components of $S \cap \bigcup_i T_i$ is as small
as possible.

If $S \cap \bigcup_i T_i = \emptyset$ then $S$ is an essential
$2$-sphere in one of the Seifert components.

If $S \cap \bigcup_i T_i \neq \emptyset$,
let $C$ be an innermost circle
in $S \cap \bigcup_i T_i$. Then $C \subset T_k$ for some
$k$ and $C = \partial D$ for some $2$-disk $D$ embedded in
a Seifert component $U$ with $T_k \subset \partial U$.
As $T_k$ is incompressible, $C =\partial D^\prime$ for some
$2$-disk $D^\prime \subset T_k$. If $D \cup D^\prime$ bounds
a $3$-ball in $U$ then we can isotope $S$ to remove the intersection
with $T_k$, which contradicts the choice of $S$. Thus
$D \cup D^\prime$ is an essential $2$-sphere in $U$.

In any case, we found an essential $2$-sphere in one of the
Seifert pieces. It follows that the Seifert piece,
and hence all of $M$, is diffeomorphic to $S^1 \times S^2$ or
$\R P^3 \# \R P^3$ \cite[p. 432]{Scott}.
\end{proof}

\begin{proposition} \label{propA.21}
A weak graph manifold is the result of performing $0$-surgeries on
a strong graph manifold.
\end{proposition}
\begin{proof}
Suppose that Proposition \ref{propA.21} fails. 
Let $n$ be the minimal number
of decomposing tori among weak graph manifolds which are counterexamples,
and let $M$ be a counterexample with decomposing tori $\{T_i\}_{i=1}^n$.

We first look for a torus $T_j$ 
for which the two induced circle fibrations
(coming from the adjacent Seifert bundles) are isotopic.  If
there is one then we extend
the Seifert fibration over $T_j$. 
In this case, by
removing $T_j$ from $\{T_i\}_{i=1}^n$, we get a weak graph
decomposition of $M$ with $(n-1)$ tori, contradicting
the definition of $n$.  

Therefore there is no such torus.
Since $M$ is a counterexample
to Proposition \ref{propA.21}, there must be a
torus in $\{T_i\}_{i=1}^n$ which is compressible.
Let $D$ be a compressing disk, which we can assume to be transversal
to $\bigcup_{i=1}^n T_i$. We choose such a compressing disk so that
$D \cap \bigcup_{i=1}^n T_i$ has the 
smallest possible number of connected components.
Let $C$ be an innermost circle in $D \cap \bigcup_{i=1}^n T_i$, say lying
in $T_k$.
Then $C$ bounds a disk $D^\prime$ in a Seifert bundle $V$
which has $T_k$ as a boundary component.

If $C$ also bounds a disk $D^{\prime \prime}
\subset T_k$ then $D^\prime \cup D^{\prime \prime}$ is an embedded
$2$-sphere $S$ in $V$.
If $S$ is not essential in $V$
then we can isotope $D$ so that it does not intersect $T_k$,
which contradicts the choice of $D$.  So $S$ is essential in $V$.
Then $V$ is diffeomorphic to $S^1 \times S^2$ or $\R P^3 \# \R P^3$,
which contradicts the assumption that it has $T_k$ as a boundary
component.

Thus we can assume that 
$D^\prime$ is a compressing disk for $V$, which is 
necessarily a solid torus \cite[Corollary 3.3]{Scott}. 

Let $U$ be the Seifert bundle on the other side of $T_k$ from
$V$.
Let $B$
be the orbifold base of $U$, with projection $\pi : U \rightarrow B$. 
There is a circle boundary component $R \subset \partial B$ so that
$T_k = \pi^{-1}(R)$. That is, 
$V$
is glued to $U$ along $\pi^{-1}(R)$. 
Choose a $D^2$-fibration $\sigma : V \rightarrow R$ that extends
$\pi : T_k \rightarrow R$.

If 
$C = \partial D^\prime \subset T_k$
is not isotopic to a fiber of $\pi \Big|_{T_k}$, 
let $u > 0$ be their algebraic intersection number in $T_k$. 
Then 
$U \cup_{T_k} V$
has a Seifert fibration over $B \cup_R D^2(u)$.
Removing $T_k$ from $\{T_i\}_{i=1}^n$, 
we again have a weak graph decomposition of $M$, now with 
$(n-1)$ tori, which is a contradiction.

Therefore 
$C = \partial D^\prime \subset T_k$
is isotopic to
a fiber of $\pi \Big|_{T_k}$. 
\\ \\
Step 1: If $B$ is 
diffeomorphic to 
$D^2$, $D^2(r)$ or $S^1 \times I$ then
put $M^\prime = M$
and $B^\prime = B$,
and go to Step 2. 
Otherwise, let
$\{\gamma_j\}_{j=1}^J$ be a maximal 
disjoint 
collection of smooth embedded arcs
$\gamma_j : [0,1] \rightarrow B_{reg}$, with
$\{\gamma_j(0), \gamma_j(1)\} \subset R$,
which determine distinct nontrivial homotopy classes
for the pair $(B_{reg},R)$.
(Note that $\partial B \subset B_{reg}$.)
If $B'$ is the result of splitting $B$ open along
$\{\gamma_j\}_{j=1}^J$,  then the connected components of 
$B'$ are diffeomorphic to 
$D^2$, $D^2(r)$ for some $r > 1$, or $S^1 \times I$. 
See Figure \ref{fig-15}.
\begin{figure}[h!] 
\includegraphics[scale=.9]{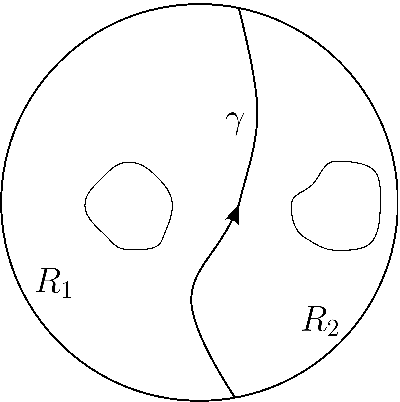} 
\caption{\label{fig-15}}
\end{figure}
Let $R^\prime$ be the result of splitting the $1$-manifold $R$
along the finite subset $\bigcup_{j=1}^J \{\gamma_j(0), \gamma_j(1) \}$.

Define a $2$-sphere $S^2_j \subset M$
by $S^2_j = \sigma^{-1}(\gamma_j(0)) 
\cup_{\pi^{-1}(\ga_j(0))} \pi^{-1}(\gamma_j) 
\cup_{\pi^{-1}(\ga_j(1))}
\sigma^{-1}(\gamma_j(1))$.
Let $Y$ be the result of splitting $M$ open along
$\{S^2_j\}_{j=1}^J$. It has
$2J$ spherical boundary components corresponding to the spherical cuts.
We glue on  $2J$ $3$-disks there, to obtain $M^\prime$.
By construction, $M$ is the result of performing $J$
$0$-surgeries on $M^\prime$.

We claim that $M^\prime$ is a weak graph manifold. To see this,
note that the union 
$W$
of  the $D^2$-bundle over $R^\prime$
and the $2J$ 3-disks is a
disjoint union of solid tori in $M^\prime$; see Figure \ref{fig-15}. 
The metric completion of
$M^\prime - W$
inherits a weak graph structure
from $M$. This shows that $M^\prime$ is a weak graph manifold.
\\ \\
Step 2 : 
For each 
component $P$ of $B^\prime$ that is diffeomorphic to 
$D^2$ or $D^2(r)$,
the corresponding component of $M^\prime$ is 
the result of gluing two solid tori: one being $\pi^{-1}(P)$
and the other one being a connected component
of $W$.
By Lemma \ref{Seifertlemma}, 
this component of $M^\prime$ is Seifert-fibered and
hence is a strong graph manifold.
We discard all such components of $M^\prime$ and 
let $\widehat{M}$ denote what's left.

A  
component $P$ of $B^\prime$ diffeomorphic to $S^1 \times I$
has a boundary consisting of
two circles $C_1$ and $C_2$, of which exactly one, say $C_1$,
does not intersect $R$.
In $\widehat{M}$, the preimage
$\pi^{-1}(C_1)$ is attached to the union of 
$\pi^{-1}(P)$
with a solid torus.  This union is itself a solid torus. 

In this way, we see that $\widehat{M}$ has a weak graph decomposition
with $(n-1)$ tori, since $T_k$ has disappeared. 
Since $M$ was a counterexample to Proposition \ref{propA.21}, it
follows that $\widehat{M}$
is also a counterexample. 
This contradicts the definition of $n$ and so proves the proposition.
\end{proof}

\subsection{Weak graph orbifolds are connected sums of strong graph orbifolds}
\label{subsectA.3}
 
In this section we only consider $3$-dimensional orbifolds that
do not admit embedded bad $2$-dimensional suborbifolds.

\begin{definition} \label{defnA.6}
A {\em weak graph orbifold} is a compact orientable $3$-orbifold $\Or$
for which there is a collection $\{E_i\}$ of disjoint embedded 
orientable Euclidean $2$-orbifolds in
$\Int(\Or)$ so that after splitting $\Or$ open along $\{E_i\}$,
the result has connected components that are Seifert-fibered 
orbifolds (possibly with boundary).
\end{definition}

\begin{definition} \label{defnA.22}
A {\em strong graph orbifold} is a compact orientable $3$-orbifold $\Or$
for which there is a collection $\{E_i\}$ of disjoint embedded
orientable Euclidean $2$-orbifolds in
$\Int(\Or)$ such that 
\begin{enumerate}
\item After splitting $\Or$ open along $\{E_i\}$, the result has
connected components that are
Seifert orbifolds (possibly with boundary).
\item For any $E_i$, the two circle fibrations on $E_i$ coming from the
adjacent Seifert bundles are not isotopic.
\item Each $E_i$ is incompressible in $\Or$.
\end{enumerate}
\end{definition}

From Subsection \ref{subsect2.4},
each $E_i$ is diffeomorphic to $T^2$ or $S^2(2,2,2,2)$.

\begin{lemma} \label{solidtoricorbigluing}
Let $U$ and $V$ be two oriented solid-toric $3$-orbifolds with
diffeomorphic boundaries. Let
$\phi : \partial U \rightarrow \partial V$ be an orientation-reversing
diffeomorphism.  Then $U \cup_{\phi} V$ admits a Seifert orbifold structure.
\end{lemma}
\begin{proof}
Suppose first that $\partial U$ is a $2$-torus. Then $U$ is diffeomorphic
to $S^1 \times D^2$ or $S^1 \times D^2(k)$. The
Seifert orbifold structures on $U$ are in one-to-one correspondence with
the Seifert manifold structures on $|U|$ 
\cite[p. 36-37]{Bonahon-Siebenmann (1985)}.
There is one circle fibering of $\partial U$
(up to isotopy)
whose fibers bound compressing
discal $2$-orbifolds in $U$. Any other circle fibering of $\partial U$
is the boundary fibration
of a Seifert fibration of $U$. As in the proof of Lemma \ref{Seifertlemma},
we can choose a circle fibering ${\mathcal F}$ of $\partial U$ so that
${\mathcal F}$ is the boundary fibration of a Seifert fibration of $U$, and
$\phi_* {\mathcal F}$ is the boundary fibration of a Seifert fibration of
$V$. The ensuing
Seifert fibrations of $U$ and $V$ join together to give a Seifert
fibration of $U \cup_{\phi} V$. 

Now suppose that $\partial U$ is diffeomorphic to
$S^2(2,2,2,2)$. The orbifiberings of $S^2(2,2,2,2)$ with one-dimensional
fiber are the $\Z_2$-quotients of $\Z_2$-invariant
circle fiberings of $T^2$. In particular, there is an infinite number of
such orbifiberings up to isotopy. (More concretely, given an orbifibering,
there are two disjoint arc fibers connecting pairs of singular points.
The complement of the two arcs in $|S^2(2,2,2,2)|$ is an open
cylinder with an induced circle fibering. The isotopy class of
the orbifibering is specified by the isotopy class of the two 
disjoint arcs.)

From \cite[p. 38-39]{Bonahon-Siebenmann (1985)}, the Seifert fibrations
of $U$ are the $\Z_2$-quotients of $\Z_2$-invariant Seifert fibrations
of its solid-toric double cover. It follows that there is one orbifibering of
$\partial U$ (up to isotopy)
whose fibers bound compressing discal $2$-orbifolds in $U$.
Any other orbifibering of $\partial U$ is the boundary fibration of a
Seifert fibration of $U$. Hence we can choose an orbifibering ${\mathcal F}$ 
of $\partial U$ so that ${\mathcal F}$ is the boundary fibration
of a Seifert fibration of $U$, and $\phi_* {\mathcal F}$ is the 
boundary fibration of a Seifert fibration of $V$. The ensuing
Seifert fibrations of $U$ and $V$ join together to give a Seifert
fibration of $U \cup_{\phi} V$. 
\end{proof}

\begin{proposition} \label{propA.23}
If a connected strong graph orbifold
contains an essential embedded spherical $2$-orbifold then it is diffeomorphic
to a finite isometric quotient of $S^1 \times S^2$.
\end{proposition}
\begin{proof}
Suppose that a connected strong graph orbifold $\Or$ contains
an essential embedded spherical $2$-orbifold $S$.

\begin{lemma} \label{lemmaA.24}
After an isotopy of $S$, we can assume that $S \cap \bigcup_i E_i$ is a
disjoint collection of closed curves in the regular part of $S$.
\end{lemma}
\begin{proof}
If $E_i$ is diffeomorphic to $T^2$ then a neighborhood of $E_i$ lies
in $|\Or|_{reg}$ and 
after isotopy,
$S \cap E_i$ is a
disjoint collection of closed curves in the regular part of $S$.
Suppose that 
$E_i$ is diffeomorphic to $S^2(2,2,2,2)$.
A neighborhood of $E_i$ is
diffeomorphic to $I \times E_i$. Suppose that $p \in S$ 
is
a singular
point of $E_i$. Then the local group of $p$ in $S$ must be $\Z_2$. 
After pushing a neighborhood of $p \in S$ slightly
in the $I$-direction of $I \times E_i$, we can remove the intersection of
$S$ with that particular singular point of $E_i$. In this way, we can
arrange so that $S$ intersects $\bigcup_i E_i$ transversely, with the
intersection lying in the regular part of $S$.
\end{proof}

We choose $S$ among all such essential
embedded spherical $2$-orbifolds so that the
number of connected components of $|S \cap \bigcup_i E_i|$ is as small
as possible.

If $S \cap \bigcup_i E_i =  \emptyset$ then $S$ is an essential
embedded spherical $2$-orbifold in one of the Seifert pieces.

If $S \cap \bigcup_i E_i \neq \emptyset$,
let $C \subset |S|$ be an innermost circle
in $|S \cap \bigcup_i E_i|$. Then $C \subset |E_k|$ for some
$k$, and $C = \partial D$ for some discal $2$-orbifold $D$ embedded in
a Seifert component $U$ with $E_k \subset \partial U$.
As $E_k$ is incompressible, $C =\partial D^\prime$ for some
discal $2$-orbifold $D^\prime \subset E_k$. Then $D \cup D^\prime$ is
an embedded $2$-orbifold with underlying space $S^2$ and at most two
singular points. As $\Or$ has no bad $2$-suborbifolds, $D \cup D^\prime$
must be 
diffeomorphic to $S^2(r,r)$
for some $r \ge 1$.
If $D \cup D^\prime$ bounds some $D^3(r,r)$ in $U$ 
then we can isotope $S$ to remove the intersection
with $E_k$, which contradicts the choice of $S$. Thus
$D \cup D^\prime$ is an essential embedded spherical $2$-orbifold in $U$.

In any case, we found an essential embedded 
spherical $2$-orbifold in one of the
Seifert pieces. Then the universal cover of the Seifert piece contains
an essential embedded $S^2$. It follows that the universal cover of the
Seifert piece is $\R \times S^2$ \cite[Proposition 2.13]{BMP}.
The Seifert piece,
and hence all of $\Or$, must then be diffeomorphic to a finite isometric
quotient of $S^1 \times S^2$.
\end{proof}

\begin{proposition} \label{propA.25}
A weak graph orbifold is the result of performing $0$-surgeries on a
strong graph orbifold.
\end{proposition}
\begin{proof}
Suppose that Proposition \ref{propA.25} fails.  Let $n$ be the minimal
number of decomposing Euclidean $2$-orbifolds among weak graph orbifolds
which are counterexamples, and let $\Or$ be a counterexample with
decomposing Euclidean $2$-orbifolds 
$\{E_i\}_{i=1}^n$.

We first look for a $2$-orbifold $E_j$ 
for which the two induced circle fibrations
(coming from the adjacent Seifert bundles) are isotopic,
in the sense of \cite[Chapter 2.5]{BMP}. 
If there is one then we extend the
Seifert fibration over $E_j$.
In this case, by removing $E_j$ from $\{E_i\}$, we get a weak graph
decomposition of $\Or$ with $(n-1)$ Euclidean $2$-orbifolds, contradicting
the definition of $n$.

Therefore there is no such Euclidean $2$-orbifold. Since $\Or$ is a 
counterexample to Proposition \ref{propA.25}, there must be a 
Euclidean $2$-orbifold in $\{E_i\}$ which is compressible.
Let $D$ be a compressing discal $2$-orbifold. As in Lemma \ref{lemmaA.24}, 
we can assume that $D$ intersects
$\bigcup_i E_i$ transversally, with the intersection lying 
in the regular part of $D$. 
We choose such a compressing discal $2$-orbifold so that $D \cap
\bigcup_i E_i$ has the smallest possible number of connected components.
Let $C$ be an innermost circle in $D \cap \bigcup_i E_i$, say lying
in $|E_k|$.
Then $C$ bounds a discal $2$-orbifold $D^\prime$ lying in 
a Seifert bundle $V$ which has $E_k$ as a boundary component.

If $C$ also bounds a discal $2$-orbifold $D^{\prime \prime}
\subset E_k$ then $D^\prime \cup D^{\prime \prime}$ is an embedded
$2$-orbifold $S$ in the Seifert bundle. As there are no 
bad $2$-orbifolds in $\Or$,  the suborbifold $S$ must be
diffeomorphic to $S^2(r,r)$ for some $r \ge 1$.
If $S$ is not essential in $V$ then
it bounds a $D^3(r,r)$ in $V$ and
we can isotope $D$ so that it does not intersect $E_k$,
which contradicts the choice of $D$.
So $S$ is essential in $V$.
From Proposition \ref{propA.23}, the Seifert bundle $V$
is diffeomorphic to a finite isometric quotient of $S^1 \times S^2$,
which contradicts the assumption that it has $E_k$ as a 
boundary component. 

Thus we can assume that 
$C$ bounds a compressing discal $2$-orbifold for $V$, which is 
necessarily a solid-toric orbifold
diffeomorphic to 
$S^1 \times D^2(r)$ or $S^1 \times_{\Z_2} D^2(r)$ for some
$r \ge 1$ \cite[Lemma 2.47]{CHK}. 

Let $U$ be the Seifert bundle on the other side of $E_k$ from
$V$. Let $B$
be the orbifold base of $U$, with projection $\pi : U \rightarrow B$. 
There is a $1$-orbifold boundary component $R \subset \partial B$,
diffeomorphic to $S^1$ or $S^1//\Z_2$, so that
$E_k = \pi^{-1}(R)$. That is, $V$ is glued to $U$
along $\pi^{-1}(R)$.  
Choose a discal orbifibration
$\sigma : V \rightarrow R$ that extends $\pi : E_k \rightarrow R$.

We refer to \cite[Chapter 2.5]{BMP} for a discussion of Dehn fillings,
i.e. gluings of $V$ to 
$\pi^{-1}(R)$.
If the meridian curve of $V$ is not isotopic to
a fiber of $\pi \Big|_{E_k}$, 
let $u > 0$ be the algebraic intersection number (computed
using the maximal abelian subgroup of $\pi_1(E_k)$). 
Then the gluing of $V$ to $U$, along
$\pi^{-1}(R)$, has a Seifert fibration.
Removing $E_k$ from $\{E_i\}$, we again have a weak graph orbifold 
decomposition of $\Or$, now with $(n-1)$ Euclidean 2-orbifolds, 
which is a contradiction.

Therefore, the meridian curve of $V$ is isotopic to
a fiber of $\pi \Big|_{E_k}$. \\ \\
Step 1 :
If one of the following possibilities holds then put $\Or^\prime = \Or$
and $B^\prime = B$,
and go to Step 2:
\begin{enumerate}
\item $B=D^2$.
\item $B=D^2(s)$ for some $s > 1$.
\item $B=D^2//\Z_2$.
\item $B=D^2(s)//\Z_2$ for some $s > 1$.
\item $B = S^1 \times I$.
\item $B = (S^1//\Z_2) \times I$.
\end{enumerate}
Otherwise, we split $B$ open along a disjoint collection of
smooth embedded arcs $\{\gamma_{j}\}_{j=1}^J \cup 
\{\gamma^\prime_{j^\prime}\}_{j^\prime = 1}^{J^\prime}$
of the following type.  A curve $\gamma_{j} : [0,1] \rightarrow B$
lies in $B_{reg}$ and
has $|\gamma_{j}|(0), |\gamma_j|(1) \in \Int(|R|)$. A curve
$\gamma_{j^\prime} : [0,1] \rightarrow B$ has $|\gamma_{j^\prime}|(0) \in 
\Int(|R|)$
and lies in $B_{reg}$, except for its endpoint $|\gamma_{j^\prime}|(1)$
which is in the interior of a reflector component of 
$\partial |B|$
but is not a corner reflector point. We can find a collection of such
curves so that if 
$B^\prime$ is the result of splitting $B$ open along them, then each 
connected component of $B^\prime$ is of type (1)-(6) above.
Put 
\begin{equation}
R^\prime = R - \bigcup_{j=1}^J \{ |\gamma_j|(0), |\gamma_j|(1) \} -
\bigcup_{j^\prime = 1}^{J^\prime} \{ |\gamma_{j^\prime}|(0) \}.
\end{equation}

Associated to $\gamma_j$ is a spherical $2$-orbifold $X_j$,
diffeomorphic to $S^2(r,r)$, given by
\begin{equation}
X_j = \sigma^{-1}(\gamma_j(0)) \cup_{\pi^{-1}(\gamma_j(0))} 
\pi^{-1}(\gamma_j) \cup_{\pi^{-1}(\gamma_j(1))}
\sigma^{-1}(\gamma_j(1)).
\end{equation}
Associated to $\gamma^\prime_{j^\prime}$ is a spherical $2$-orbifold
$X^\prime_{j^\prime}$, diffeomorphic to $S^2(2,2,r)$, given by
\begin{equation}
X^\prime_{j^\prime} = \sigma^{-1}(\gamma_{j^\prime}(0)) 
\cup_{\pi^{-1}(\gamma_{j^\prime}(0))} 
\pi^{-1}(\gamma_{j^\prime}).
\end{equation} 
Let $Y$ be the result of splitting $\Or$ open along
$\{X_j\}_{j=1}^J \cup \{X^\prime_{j^\prime}\}_{j^\prime=1}^{J^\prime}$.
It has $2(J + J^\prime)$ spherical boundary components corresponding to the
spherical cuts. We glue on $2J$ copies of $D^3(r,r)$ and $2J^\prime$
copies of $D^3(2,2,r)$, to obtain $\Or^\prime$. By construction,
$\Or$ is the result of performing $0$-surgeries on $\Or^\prime$.

We claim that $\Or^\prime$ is a weak graph orbifold. To see this, note
that the union
$W$
of $\sigma^{-1}(R^\prime)$ and the $2(J+J^\prime)$ 
discal $3$-orbifolds is a disjoint union of solid-toric $3$-orbifolds
in $\Or^\prime$.  The metric completion of 
$|\Or^\prime| - |W|$
in $|\Or^\prime|$ 
inherits a weak graph orbifold structure from $\Or$.
This shows that $\Or^\prime$ is a weak graph orbifold. \\ \\
Step 2 : 
For each connected component of $B^\prime$
of type (1)-(4) above,
the corresponding component of $\Or^\prime$ is the result of gluing
two solid-toric orbifolds:  one being the Seifert orbifold over that
component of $B^\prime$, and the other one being a connected component
of $W$.
By Lemma \ref{solidtoricorbigluing}, 
this component of $\Or^\prime$ is Seifert-fibered and
hence is a strong graph orbifold.
We discard all such components of $\Or^\prime$ and let
$\widehat{\Or}$ denote what's left.

Turning to the remaining possibilities,
an annular component $P$ of $B^\prime$
has a boundary consisting of two
circles $C_1$ and $C_2$, of which exactly one, say $C_1$,
does not intersect $R$.
In $\widehat{\Or}$, the preimage
$\pi^{-1}(C_1)$ is attached to the union of $\pi^{-1}(P)$ with
a solid-toric orbifold diffeomorphic to
$S^1 \times D^2(r)$.
This union is itself diffeomorphic to 
$S^1 \times D^2(r)$, since $\pi^{-1}(P)$ is diffeomorphic
to $S^1 \times S^1 \times I$.

Finally, if a 
component $P$ of $B^\prime$ is diffeomorphic to
$(S^1//\Z_2) \times I$ then
$\partial |P|$
consists
of a circle with two reflector components and two nonreflector 
components.
Exactly one of the nonreflector components, say $C_1$,
does not intersect $ R$.
In $\widehat{\Or}$, the preimage
$\pi^{-1}(C_1)$ is attached to the union of $\pi^{-1}(P)$ with
a solid-toric orbifold diffeomorphic to 
$S^1 \times_{\Z_2} D^2(r)$.
This union is itself diffeomorphic to
$S^1 \times_{\Z_2} D^2(r)$, since $\pi^{-1}(P)$ is diffeomorphic
to $(S^1 \times_{\Z_2} S^1) \times I$.

In this way, we see that $\widehat{\Or}$ has a weak graph orbifold
decomposition with $(n-1)$ Euclidean $2$-orbifolds, since
$E_k$ has disappeared.  Since $\Or$ was a counterexample to
Proposition \ref{propA.25}, it follows that $\widehat{\Or}$ is
also a counterxample.  This contradicts the definition of $n$ and
so proves the proposition.
\end{proof}

\subsection{Weak graph orbifolds with a compressible boundary component}
\label{subsectA.4}

\mbox{}

\begin{lemma}
\label{lem-weakgraphcompressible}
Suppose that $\Or$ is a weak graph orbifold, and that
$C\subset\D\Or$ is a compressible
boundary component.  Then $\Or$ arises from $0$-surgery on a disjoint collection
$\Or_0\sqcup\ldots \sqcup\Or_n$, where: 
\begin{itemize}
\item $\Or_i$ is a strong graph manifold for all $i$.
\item
$\D \Or_0= C$.
\item $\Or_0$ is 
a solid-toric $3$-orbifold.
\end{itemize}
\end{lemma}
\begin{proof}
Let $Z$ be a compressing discal orbifold for $C$.   

By Proposition \ref{propA.25}  we know that $\Or$ comes from $0$-surgery on a collection
$\Or_0,\ldots,\Or_n$ of  strong graph orbifolds, where $\D\Or_0$ contains $C$.
Consider a collection $\mathcal{S}=\{S_1,\ldots,S_k\}\subset\Or$ of spherical  $2$-suborbifolds
associated with such a $0$-surgery description of $\Or$.   We may assume
that  $Z$ is
transverse  to $\mathcal{S}$, and that the number of connected components in the 
intersection $Z\cap \mathcal{S}$ is minimal among such compressing discal orbifolds.   
Reasoning as in the proof of Lemma \ref{lemmaA.24}, we conclude
that $Z$ is disjoint from $\mathcal{S}$.    Therefore after 
splitting $\Or$ open along
$\mathcal{S}$ and filling in the boundary components to undo the $0$-surgeries,
we get that $Z$ lies in $\Or_0$.  
Similar reasoning shows that $Z$ must lie in a single Seifert component $U$ of
$\Or_0$.  
An orientable Seifert $3$-orbifold with a compressible boundary component
must be a solid-toric $3$-orbifold \cite[Lemma 2.47]{CHK}. The
lemma follows.
\end{proof}

\end{document}